\documentclass[final]{siamart190516}
\usepackage{amsfonts}
\usepackage{amsfonts,amsmath,amssymb}
\usepackage{mathrsfs,mathtools,stmaryrd,wasysym}
\usepackage{enumerate}
\usepackage{enumitem}
\usepackage{esint}
\usepackage{graphicx,psfrag}
\usepackage{hyperref}
\usepackage{amsmath,stackengine}
\usepackage{yfonts}
\DeclareMathAlphabet{\mathpzc}{OT1}{pzc}{m}{it}

\newsiamremark{remark}{Remark}
\newcommand{\TheTitle}{A pointwise tracking optimal control problem for the stationary Navier--Stokes equations}
\newcommand{\ShortTitle}{A pointwise tracking problem for the Navier--Stokes equations}
\newcommand{\TheAuthors}{F. Fuica, E. Ot\'arola}

\headers{\ShortTitle}{\TheAuthors}

\title{{\TheTitle}\thanks{FF is supported by ANID through FONDECYT postdoctoral project 3230126. EO is partially supported by ANID through FONDECYT Project 1220156.}}

\author{Francisco Fuica\thanks{Facultad de Matem\'aticas, Pontificia Universidad Cat\'olica de Chile, Avenida Vicu\~{n}a Mackenna 4860, Santiago, Chile.
(\email{francisco.fuica@mat.uc.cl}).}
\and
Enrique Ot\'arola\thanks{Departamento de Matem\'atica, Universidad T\'ecnica Federico Santa Mar\'ia, Valpara\'iso, Chile.
(\email{enrique.otarola@usm.cl}, \url{http://eotarola.mat.utfsm.cl/}).}
}

\ifpdf
\hypersetup{
  pdftitle={\TheTitle},
  pdfauthor={\TheAuthors}
}
\fi

\date{Draft version of \today.}

\begin{document}

\maketitle

\begin{abstract}
We study a pointwise tracking optimal control problem for the stationary Navier--Stokes equations; control constraints are also considered. The problem entails the minimization of a cost functional involving point evaluations of the state velocity field, thus leading to an adjoint problem with a linear combination of Dirac measures as a forcing term in the momentum equation, and whose solution has reduced regularity properties. We analyze the existence of optimal solutions and derive first and, necessary and sufficient, second order optimality conditions in the framework of regular solutions for the Navier--Stokes equations. We develop two discretization strategies: a semidiscrete strategy in which the control variable is not discretized, and a fully discrete scheme in which the control variable is discretized with piecewise constant functions. For each solution technique, we analyze convergence properties of discretizations and derive a priori error estimates.
\end{abstract}




\section{Introduction}\label{sec:intro}
We study existence results, optimality conditions, and finite element methods for a pointwise tracking optimal control problem of the stationary Navier--Stokes equations. This control-constrained optimization problem entails the minimization of a cost functional containing pointwise evaluations of the state velocity field. More precisely, for $d\in \{2, 3\}$, let $\Omega\subset\mathbb{R}^d$ be an open and bounded domain with Lipschitz boundary $\partial\Omega$ and $\mathcal{D} \subset \Omega$ be a finite ordered set. Given a set of desired velocity fields $\{\mathbf{y}_t\}_{t\in \mathcal{D}}\subset\mathbb{R}^d$, a regularization parameter $\alpha>0$, and the functional
\begin{equation}\label{def:cost_func}
 J(\mathbf{y},\mathbf{u}):=\frac{1}{2}\sum_{t\in \mathcal{D}} |\mathbf{y}(t)-\mathbf{y}_t|^{2}+\frac{\alpha}{2}\|\mathbf{u}\|_{\mathbf{L}^2(\Omega)}^2,
\end{equation} 
the problem reads as follows: Find $\min J(\mathbf{y},\mathbf{u})$ subject to the Navier--Stokes equations
\begin{equation}\label{eq:state_equations}
-\nu \Delta \mathbf{y} + (\mathbf{y}\cdot\nabla) \mathbf{y} + \nabla p =  \mathbf{u}  \text{ in }  \Omega, 
\qquad 
\text{div }\mathbf{y} = 0 \text{ in }  \Omega, 
\qquad
\mathbf{y}  =  \mathbf{0}  \text{ on }  \partial\Omega,
\end{equation}
and the control constraints
\begin{equation}\label{def:box_constraints}
\mathbf{u} \in \mathbf{U}_{ad},
\qquad 
\mathbf{U}_{ad}:=\{\mathbf{v} \in \mathbf{L}^2(\Omega):  \mathbf{a} \leq \mathbf{v}(x) \leq \mathbf{b} \text{ a.e.}~x \in \Omega \}.
\end{equation}
The control bounds $\mathbf{a},\mathbf{b} \in \mathbb{R}^d$ are chosen so that $\mathbf{a} < \mathbf{b}$. At the outset, we note that vector inequalities must always be understood componentwise in this work. In \eqref{def:cost_func}, $|\cdot|$ stands for the Euclidean norm.

The study of finite element discretization schemes for optimal control problems governed by the stationary Navier--Stokes equations has already been addressed in the literature. For a different cost functional $J$ which, in contrast to \eqref{def:cost_func}, considers the square of $\|\mathbf{y} - \mathbf{y}_{d}\|_{\mathbf{L}^2(\Omega)}$, with $\mathbf{y}_{d} \in \mathbf{L}^2(\Omega)$, instead of $\sum |\mathbf{y}(t)-\mathbf{y}_t|^{2}$, the authors of \cite{MR2338434} derive a priori error estimates for appropriate finite element discretizations of the corresponding optimal control problem. More precisely, the authors develop two discretization schemes: a fully discrete scheme that discretizes the admissible control set with piecewise constant functions, and a semidiscrete scheme based on the so-called variational approach. The authors prove that strict local nonsingular solutions can be approximated by a sequence of solutions of the discrete control problems \cite[Theorem 4.11]{MR2338434}, and then, assuming that the local solution satisfies a second order optimality condition, derive error estimates in $\mathbf{L}^2(\Omega)$ for the error committed in the control approximation: $\mathcal{O}(h)$ for the fully discrete scheme and $\mathcal{O}(h^2)$ for the semidiscrete scheme \cite[Theorem 4.18]{MR2338434}. Later, the authors of the paper \cite{MR4301394} develop residual-type a posteriori error estimators for the finite element schemes presented in \cite{MR2338434}. 

For optimal control problems involving a pointwise tracking cost functional, there are several finite element discretization methods in the literature, but mostly for linear and elliptic state equations. We mention the works \cite{MR3449612,MR3523574,MR3800041,MR3973329} in which a priori and a posteriori error estimates were derived for a pointwise tracking optimal control problem subject to a Poisson problem. We also mention the works \cite{MR4013930,MR4304887,MR4218114} for extensions to the Stokes system. In contrast to these advances, the analysis of approximation techniques for pointwise tracking optimal control problems subject to nonlinear partial differential equations (PDEs) is rather scarce. To our knowledge, the only work that considers this type of problem in a semilinear elliptic scenario is \cite{MR4438718}. Here, the authors derive the existence of solutions, analyze first and second order optimality conditions, and prove error estimates for two discretization strategies for approximating a strict local solution of the optimal control problem.

As far as we know, this is the first paper dealing with approximation techniques for a pointwise optimal control problem of the Navier–Stokes equations. In the following, we list what we consider to be the most important contributions of the manuscript:
\begin{itemize}
\item[$\bullet$] \emph{Existence of an optimal control:} We show on Lipschitz domains that our optimal control problem admits at least one solution (Theorem \ref{thm:existence_of_sol_ocp}).

\item[$\bullet$] \emph{Well-posedness of the adjoint problem:} 
Since the cost functional of our problem considers point evaluations of the velocity field of the state, the momentum equation of the adjoint equations involves a linear combination of Dirac measures as a forcing term. On Lipschitz domains, we prove that the adjoint problem is well posed in $\mathbf{W}_0^{1,\mathsf{p}}(\Omega)\times L_0^{\mathsf{p}}(\Omega)$, where $\mathsf{p}<d/(d-1)$ is arbitrarily close to $d/(d-1)$; see Theorem \ref{thm:well-posedness-adjoint}. The proof relies on the fact that $(\mathbf{y},p)$ is a solution of the Navier--Stokes equations such that $\mathbf{y}$ is regular.
\item[$\bullet$] \emph{Optimality conditions:} Assuming that $\Omega$ is Lipschitz and $(\bar{\mathbf{y}},\bar{p},\bar{\mathbf{u}})$ is a local nonsingular solution, we obtain first order optimality conditions in Theorem \ref{thm:optimality_cond} and second order necessary and sufficient optimality conditions with a minimal gap in section \ref{sec:2nd_order}.
Since the adjoint velocity field $\bar{\mathbf{z}}\in \mathbf{W}_0^{1,\mathsf{p}}(\Omega) \setminus (\mathbf{H}_0^1(\Omega) \cap \mathbf{C}(\bar{\Omega}))$, where $\mathsf{p}<d/(d-1)$ is arbitrarily close to $d/(d-1)$, the analysis requires a suitable adaptation of the arguments available in \cite{MR2338434}.
\item[$\bullet$] \emph{Error estimates:}  We develop two discretization strategies: a semidiscrete strategy in which the control variable is not discretized, and a fully discrete scheme in which the control variable is discretized with piecewise constant functions. Assuming that $\Omega$ is a convex polytope, we derive error bounds in $\mathbf{L}^{2}(\Omega)$ for the error approximation of a suitable optimal control variable. The analysis involves estimates in the $\mathbf{L}^{\infty}(\Omega)$-norm and suitable $\mathbf{W}^{1,\mathsf{p}}(\Omega)$-spaces, combined with the treatment of first and second order optimality conditions. This interweaving of concepts is one of the main contributions of this manuscript.
\end{itemize}

The outline of our work is as follows. In section \ref{sec:not_and_prel}, we establish the notation and introduce the functional framework we will work with. In section \ref{sec:ocp}, we analyze a weak version of the optimal control problem \eqref{def:cost_func}--\eqref{def:box_constraints}. In particular, we present the existence of solutions. In section \ref{sec:opt_cond}, we establish first and second order optimality conditions. In section \ref{sec:fem}, we present two finite element discretizations for \eqref{def:cost_func}--\eqref{def:box_constraints}, we discuss some results related to the discretization of the state and adjoint equations, and we prove convergence results for the discretizations. In section \ref{sec:error_estimates}, we derive error estimates for the approximation of a suitable optimal control variable.


\section{Notation and preliminary remarks}\label{sec:not_and_prel}
Let us establish the notation and describe the framework we will work with.


\subsection{Notation}\label{sec:notation}

Throughout the paper, we use standard notation for Lebesgue and Sobolev spaces and their norms. We denote by $L^2_0(\Omega)$ the space of functions in $L^2(\Omega)$ that have zero average on $\Omega$. We use uppercase bold letters to denote the vector-valued counterparts of the aforementioned spaces, while lowercase bold letters are used to denote vector-valued functions. In particular, we set $\mathbf{V}(\Omega):= \{ \mathbf{v} \in \mathbf{H}_{0}^{1}(\Omega): \text{div }\mathbf{v}  = 0\}$.

If $\mathfrak{X}$ and $\mathfrak{Y}$ are normed vector spaces, we write $\mathfrak{X}\hookrightarrow\mathfrak{Y}$ to denote that $\mathfrak{X}$ is continuously embedded in $\mathfrak{Y}$. We denote by $\mathfrak{X}'$ and $\|\cdot\|_{\mathfrak{X}}$ the dual and norm, respectively, of $\mathfrak{X}$. We denote by $\langle \cdot,\cdot \rangle_{\mathfrak{X}',\mathfrak{X}}$ the duality pairing between $\mathfrak{X}'$ and $\mathfrak{X}$. When the spaces $\mathfrak{X}'$ and $\mathfrak{X}$ are clear from the context, we simply denote $\langle \cdot,\cdot \rangle_{\mathfrak{X}',\mathfrak{X}}$ by $\langle \cdot,\cdot \rangle$. Given $q \in (1,\infty)$, we denote by $q'$ the real number such that $1/q + 1/q' =1$. The relation $a \lesssim b$ indicates that $a \leq C b$, with a positive constant that does not depend on $a$, $b$, or the discretization parameters. The value of $C$ might change at each occurrence. If the particular value of a constant is significant, then we assign it a name.


\subsection{Preliminary remarks on the Navier--Stokes equations}\label{sec:preliminaries}

In this section, we present some standard results in the analysis of the Navier--Stokes equations, which will be used frequently in the following sections.

To present a weak formulation, we introduce 
$
b(\mathbf{v}_1;\mathbf{v}_2,\mathbf{v}_3):=((\mathbf{v}_1\cdot \nabla)\mathbf{v}_2,\mathbf{v}_3)_{\mathbf{L}^2(\Omega)}.
$ 
The form $b$ satisfies the following properties \cite[Chapter IV, Lemma 2.2]{MR851383}, \cite[Chapter II, Lemma 1.3]{MR603444}: If $\mathbf{v}_1\in \mathbf{V}(\Omega)$ and $\mathbf{v}_2, \mathbf{v}_3 \in \mathbf{H}_0^1(\Omega)$, then
\begin{equation}\label{eq:properties_trilinear}
b(\mathbf{v}_1;\mathbf{v}_2,\mathbf{v}_3)=-b(\mathbf{v}_1;\mathbf{v}_3,\mathbf{v}_2), 
\qquad  
b(\mathbf{v}_1;\mathbf{v}_2,\mathbf{v}_2)=0.
\end{equation}
Moreover, the form $b$ is well-defined and continuous on $\mathbf{H}_0^1(\Omega)^3$ and satisfies the bound
\begin{equation}\label{eq:trilinear_embedding}
|b(\mathbf{v}_1;\mathbf{v}_2,\mathbf{v}_3)|\leq \mathcal{C}_b\|\nabla \mathbf{v}_1\|_{\mathbf{L}^2(\Omega)}\|\nabla \mathbf{v}_2\|_{\mathbf{L}^2(\Omega)}\|\nabla \mathbf{v}_3\|_{\mathbf{L}^2(\Omega)},
\end{equation}
where $\mathcal{C}_b>0$; see \cite[Lemma IX.1.1]{Gal11} and \cite[Chapter II, Lemma 1.1]{MR603444}.

We note that the divergence operator is surjective from $\mathbf{H}_0^1(\Omega)$ to $L_0^2(\Omega)$. This implies that there exists $\beta > 0$ such that \cite[Chapter I, \S 5.1]{MR851383}, \cite[Corollary B. 71]{MR2050138}
\begin{equation}\label{eq:inf_sup_cond}
\sup_{\mathbf{v}\in\mathbf{H}_0^1(\Omega)}\frac{(q,\text{div }\mathbf{v})_{L^2(\Omega)}}{\|\nabla \mathbf{v}\|_{\mathbf{L}^2(\Omega)}}\geq \beta\|q\|_{L^2(\Omega)}  \quad \forall q\in L_0^2(\Omega),
\qquad
\beta = \beta(d,\Omega).
\end{equation}

Let us now present a weak formulation for the Navier--Stokes system \cite[Chapter IV, (2.8)]{MR851383}: Given $\mathbf{f} \in \mathbf{H}^{-1}(\Omega)$, find $(\mathbf{y},p) \in \mathbf{V}(\Omega)\times L_0^2(\Omega)$ such that 
\begin{equation}\label{eq:weak_ns_divergence_free}
\nu(\nabla \mathbf{y}, \nabla \mathbf{v})_{\mathbf{L}^2(\Omega)}+b(\mathbf{y};\mathbf{y},\mathbf{v})-(p,\text{div } \mathbf{v})_{L^2(\Omega)}  = \langle \mathbf{f},\mathbf{v} \rangle_{\mathbf{H}^{-1}(\Omega),\mathbf{H}_0^1(\Omega)}  ~~~  \forall \mathbf{v} \in \mathbf{H}_0^1(\Omega).
\end{equation}

The following result states the existence of solutions to \eqref{eq:weak_ns_divergence_free} for general data (see \cite[Chapter IV, Theorem 2.1]{MR851383} and \cite[Chapter II, Theorem 1.2]{MR603444}).
 
\begin{theorem}[existence of solutions]\label{thm:well_posedness_navier_stokes}
If $\nu>0$ and $\mathbf{f} \in \mathbf{H}^{-1}(\Omega)$, then problem \eqref{eq:weak_ns_divergence_free} admits at least one solution $(\mathbf{y},p) \in \mathbf{V}(\Omega)\times L_0^2(\Omega)$ which satisfies the estimate
\begin{equation}\label{eq:stability_state_eq}
\|\nabla \mathbf{y}\|_{\mathbf{L}^2(\Omega)} \leq \nu^{-1}\|\mathbf{f}\|_{\mathbf{H}^{-1}(\Omega)}.
\end{equation}
\end{theorem}

\begin{remark}[equivalent formulation]\label{rmk:equivalent}
We note that problem \eqref{eq:weak_ns_divergence_free} can be equivalently formulated as follows: Find $(\mathbf{y},p)\in\mathbf{H}_0^1(\Omega)\times L_0^2(\Omega)$ such that 
\begin{equation}\label{eq:weak_navier_stokes_eq}
\nu(\nabla \mathbf{y}, \nabla \mathbf{v})_{\mathbf{L}^2(\Omega)}+b(\mathbf{y};\mathbf{y},\mathbf{v})-(p,\text{div } \mathbf{v})_{L^2(\Omega)} \\
= \langle \mathbf{f},\mathbf{v} \rangle, \qquad (q,\text{div } \mathbf{y})_{L^2(\Omega)}  =   0,
\end{equation}
for all $(\mathbf{v},q) \in \mathbf{H}_0^1(\Omega)\times L^2_0(\Omega)$ \cite[Chapter IV, Section 2.1]{MR851383}.
\end{remark}

\begin{theorem}[regularity estimates]\label{thm:reg_velocity}
Let $\nu>0$ and $\mathbf{f} \in \mathbf{L}^2(\Omega)$. If $(\mathbf{y},p) \in \mathbf{H}_0^1(\Omega) \times L_0^2(\Omega)$ denotes a solution to \eqref{eq:weak_navier_stokes_eq}, then there exists $\kappa$ such that $(\mathbf{y},p) \in \mathbf{W}^{1,\kappa}_{0}(\Omega)\times L_0^{\kappa}(\Omega)$. Here, $\kappa > 4$ if $d = 2$ and $\kappa > 3$ if $d=3$. Consequently, $\mathbf{y}\in \mathbf{C}(\bar{\Omega})$.
\end{theorem}
\begin{proof}
We begin the proof by writing the momentum equation of the Navier--Stokes equations as $-\nu\Delta \mathbf{y} + \nabla p =  \mathbf{f} -  (\mathbf{y}\cdot\nabla) \mathbf{y}  \text{ in }  \Omega$.
Let us now define the functional 
$
\mathcal{G}: \mathbf{H}_0^1(\Omega) \rightarrow \mathbb{R}
$
by
$
\mathcal{G}(\mathbf{v}):= (\mathbf{f},\mathbf{v})_{\mathbf{L}^2(\Omega)}  - b (\mathbf{y};\mathbf{y},\mathbf{v}).
$

Depending on the spatial dimension, we proceed differently. Let $d=3$. We prove that there exists $\kappa>3$ such that $\mathcal{G} \in \mathbf{W}^{-1,\kappa}(\Omega)$ based on a bootstrap argument. As a first step, we observe that, in view of H\"older's inequality and the continuous embeddings $\mathbf{H}_0^1(\Omega)\hookrightarrow \mathbf{L}^6(\Omega)$ and $\mathbf{W}^{1,3/2}(\Omega)\hookrightarrow \mathbf{L}^3(\Omega)$ \cite[Theorem 4.12]{MR2424078}, we have
\begin{equation}
|b(\mathbf{y};\mathbf{y},\mathbf{v})|
\leq \|\mathbf{y}\|_{\mathbf{L}^{6}(\Omega)}\|\nabla \mathbf{y}\|_{\mathbf{L}^{2}(\Omega)}\|\mathbf{v}\|_{\mathbf{L}^{3}(\Omega)}
\lesssim
\|\nabla \mathbf{y}\|^2_{\mathbf{L}^{2}(\Omega)}\|\nabla\mathbf{v}\|_{\mathbf{L}^{3/2}(\Omega)}.
\label{eq:convective_W13}
\end{equation}
It follows that $b(\mathbf{y};\mathbf{y},\cdot) \in \mathbf{W}^{-1,3}(\Omega)$ and hence $\mathcal{G} \in \mathbf{W}^{-1,3}(\Omega)$. Since $\Omega$ is Lipschitz, we can use the regularity results for the Stokes problem from \cite[Theorem 2.9]{Brown_Shen} (see also \cite[Corollary 1.7]{MR2987056} with $\alpha=-1$ and $q=2$) to conclude that $(\mathbf{y},p)\in \mathbf{W}_{0}^{1,3}(\Omega)\times L_0^3(\Omega)$. With this regularity result, we can rebound the convective term as follows:
\[
|b(\mathbf{y};\mathbf{y},\mathbf{v})|
\leq \|\mathbf{y}\|_{\mathbf{L}^{\mu}(\Omega)}\|\nabla \mathbf{y}\|_{\mathbf{L}^{3}(\Omega)}\|\mathbf{v}\|_{\mathbf{L}^{\tau}(\Omega)},
\qquad
\mu^{-1} + \tau^{-1} = \tfrac{2}{3}.
\]
Since in three dimensions, $\mathbf{W}^{1,3}_0(\Omega)\hookrightarrow \mathbf{L}^\mu(\Omega)$ holds for every $\mu<\infty$, we consider $\tau > 3/2$, where $\tau$ is sufficiently close to $3/2$. We now invoke the embedding $\mathbf{W}_{0}^{1,\sigma}(\Omega)\hookrightarrow \mathbf{L}^{\tau}(\Omega)$, which holds for $\sigma > 1$, sufficiently close to $1$, to control the convective term:
\[
|b(\mathbf{y};\mathbf{y},\mathbf{v})| \lesssim \|\nabla \mathbf{y}\|^2_{\mathbf{L}^{3}(\Omega)}\|\nabla\mathbf{v}\|_{\mathbf{L}^{\sigma}(\Omega)}.
\]
This implies that $b(\mathbf{y};\mathbf{y},\cdot) \in \mathbf{W}^{-1,\mathfrak{q}}(\Omega)$, where $\mathfrak{q}=\sigma'<\infty$. Since $\mathbf{f} \in \mathbf{L}^2(\Omega)$, we deduce the existence of some $\kappa >3$ such that $\mathcal{G}  \in \mathbf{W}^{-1,\kappa}(\Omega)$. We again invoke \cite[Theorem 2.9]{Brown_Shen} to conclude that $(\mathbf{y},p) \in \mathbf{W}_0^{1,\kappa}(\Omega)\times L_0^{\kappa}(\Omega)$ for some $\kappa>3$. 

Let $d=2$. We have the bound $|b(\mathbf{y};\mathbf{y},\mathbf{v})|
\leq \|\mathbf{y}\|_{\mathbf{L}^{\mu}(\Omega)} \|\nabla \mathbf{y}\|_{\mathbf{L}^{2}(\Omega)}\|\mathbf{v}\|_{\mathbf{L}^{\tau}(\Omega)}$, where $\mu^{-1} + \tau^{-1} = 2^{-1}$. Since in two dimensions, $\mathbf{H}_0^1(\Omega)\hookrightarrow \mathbf{L}^{\mu}(\Omega)$ holds for every $\mu < \infty$, we consider $\tau>2$, where $\tau$ is sufficiently close to $2$. Consequently, for $\sigma >1$, sufficiently close to $1$, we have the following control of the convective term:
\[
|b(\mathbf{y};\mathbf{y},\mathbf{v})| \lesssim \|\nabla \mathbf{y}\|^2_{\mathbf{L}^{2}(\Omega)}\|\nabla \mathbf{v}\|_{\mathbf{L}^{\sigma}(\Omega)}.
\]
This shows that $b(\mathbf{y};\mathbf{y},\cdot) \in \mathbf{W}^{-1,\mathfrak{q}}(\Omega)$, where $\mathfrak{q} = \sigma' < \infty$. Exploiting the fact that $\mathbf{f} \in \mathbf{L}^2(\Omega)$, we thus obtain the existence of some $\kappa >4$ such that $\mathcal{G}  \in \mathbf{W}^{-1,\kappa}(\Omega)$. Based on \cite[Corollary 1.7]{MR2987056}, with $\alpha=-1$ and $q=2$, we can conclude.
\end{proof}

\begin{remark}[regularity properties on convex domains]
Let $(\mathbf{y},p)$ be a solution to \eqref{eq:weak_navier_stokes_eq}. If $\Omega$ is a convex polygon/polyhedron and $\mathbf{f}\in \mathbf{L}^2(\Omega)$, then $(\mathbf{y},p) \in \mathbf{H}^2(\Omega) \times H^1(\Omega)$; see \cite[Corollary 7.3.3.5]{MR775683} for $d=2$ and \cite[Theorem 11.3.1]{MR2641539} for $d=3$.
\label{remark:further_reg_NS}
\end{remark}


\subsection{Regular solutions to the Navier--Stokes equations}

We begin this section by introducing the concept of \emph{regular solution} to the Navier--Stokes equations; see \cite[Definition 2.3]{MR2338434} and \cite[Definition 2.7]{MR3936891}.

\begin{definition}[regular solution]\label{def:reg_sol}
Let $(\mathbf{y},p)$ be a solution to \eqref{eq:state_equations} associated to some $\mathbf{u} \in \mathbf{U}_{ad}$. We say that $\mathbf{y}$ is regular if for every $\mathbf{g}\in\mathbf{H}^{-1}(\Omega)$ the problem: Find
\begin{multline}\label{eq:first_deriv_S*}
(\boldsymbol{\varphi},\zeta)\in \mathbf{H}_0^1(\Omega)\times L_0^2(\Omega):
\quad
\nu(\nabla \boldsymbol{\varphi}, \nabla \mathbf{v})_{\mathbf{L}^2(\Omega)} + b(\mathbf{y};\boldsymbol{\varphi},\mathbf{v}) + b(\boldsymbol{\varphi};\mathbf{y},\mathbf{v}) \\
- (\zeta, \textnormal{div }\mathbf{v})_{L^2(\Omega)}
= \langle \mathbf{g},\mathbf{v} \rangle_{\mathbf{H}^{-1}(\Omega),\mathbf{H}_0^1(\Omega)}, \qquad  (q,\textnormal{div }\boldsymbol{\varphi})_{L^2(\Omega)} = 0,
\end{multline}
for all $(\mathbf{v},q)\in\mathbf{H}_0^1(\Omega)\times L_0^2(\Omega)$, is well posed.
\label{def:regular_solution}
\end{definition}

\begin{remark}[isomorphism]
We note that, if $\mathbf{y}$ is regular, then the map
\begin{equation}
T: \mathbf{V}(\Omega) \times L_0^2(\Omega) \rightarrow \mathbf{H}^{-1}(\Omega),
\quad
( \boldsymbol{\varphi},\zeta) \mapsto -\nu \Delta  \boldsymbol{\varphi} + (\mathbf{y} \cdot \nabla)  \boldsymbol{\varphi} + ( \boldsymbol{\varphi} \cdot \nabla) \mathbf{y} + \nabla \zeta,
\label{eq:T}
\end{equation}
is an isomorphism.
\label{rmk:isomorphism}
\end{remark}

\begin{remark}[$\mathbf{y}$ is regular if $\nu$ is sufficiently large]
Let $\mathfrak{M}_{ad}:=\sup_{\mathbf{u}\in\mathbf{U}_{ad}}\|\mathbf{u}\|_{\mathbf{L}^2(\Omega)}$ and suppose that $\nu$ is such that $\nu^{-2}\mathcal{C}_b C_{2} \mathfrak{M}_{ad} < 1$ holds \cite{Ru,MR2192070}. Here, $C_2$ denotes the best constant in the embedding $\mathbf{H}_0^1(\Omega) \hookrightarrow \mathbf{L}^2(\Omega)$ and $\mathcal{C}_b$ is given as in \eqref{eq:trilinear_embedding}. We prove that under this assumption \eqref{eq:first_deriv_S*} is well posed. To achieve this, we define
\begin{equation*}
\mathcal{B}:\mathbf{H}_0^1(\Omega)\times \mathbf{H}_0^1(\Omega) \rightarrow \mathbb{R},
\qquad
\mathcal{B}(\mathbf{w},\mathbf{v}) 
:= 
\nu(\nabla \mathbf{w}, \nabla \mathbf{v})_{\mathbf{L}^2(\Omega)} +
b(\mathbf{y};\mathbf{w},\mathbf{v}) + b(\mathbf{w};\mathbf{y},\mathbf{v}).
\end{equation*}
Since $\mathcal{C}_b C_{2} \mathfrak{M}_{ad} < \nu^{2}$, \cite[Chapter IV, Theorem 2.2]{MR851383} reveals that, for each $\mathbf{u}\in\mathbf{U}_{ad}$, there exists a unique solution $(\mathbf{y},p)$ to \eqref{eq:state_equations} such that $\|\nabla \mathbf{y}\|_{\mathbf{L}^2(\Omega)}< \mathcal{C}_{b}^{-1}\nu$. Hence,
\begin{equation*}
\mathcal{B}(\mathbf{v},\mathbf{v}) \geq (\nu - \mathcal{C}_{b}\|\nabla\mathbf{y}\|_{\mathbf{L}^2(\Omega)}) \|\nabla\mathbf{v}\|_{\mathbf{L}^2(\Omega)}^2  \geq C \|\mathbf{v}\|_{\mathbf{L}^2(\Omega)}^2, \quad
C>0.
\end{equation*}
Thus, we have established that $\mathcal{B}$ is continuous and coercive on $\mathbf{H}_0^1(\Omega) \times \mathbf{H}_0^1(\Omega)$. Based on the inequality \eqref{eq:inf_sup_cond}, the standard inf-sup theory for saddle point problems \cite[Theorem 2.34]{MR2050138} yields the well-posedness of problem \eqref{eq:first_deriv_S*}.
\end{remark}

In what follows, $\mathtt{p}, \mathtt{q} \in (1,\infty)$ are such that $\mathtt{p}^{-1} + \mathtt{q}^{-1}  = 1$.

\begin{lemma}[regularity result]
Let $\mathbf{u} \in \mathbf{U}_{ad}$ and let $(\mathbf{y},p)$ be a solution to \eqref{eq:state_equations} such that $\mathbf{y}$ is regular. Then, there exists $\mathtt{q} > 4$ if $d=2$ and $\mathtt{q}>3$ if $d=3$ such that $(\boldsymbol{\varphi},\zeta)$, the solution to \eqref{eq:first_deriv_S*}, belongs to $\mathbf{W}_{0}^{1,\mathtt{q}}(\Omega) \times L_{0}^{\mathtt{q}}(\Omega)$ provided $\mathbf{g} \in \mathbf{W}^{-1,\mathtt{q}}(\Omega)$.
\label{lemma:regularity_estimate}
\end{lemma}
\begin{proof}
We begin by noting that since $\mathbf{g} \in \mathbf{W}^{-1,\mathtt{q}}(\Omega)$, $\mathbf{W}^{-1,\mathtt{q}}(\Omega) \hookrightarrow \mathbf{H}^{-1}(\Omega)$, and $\mathbf{y}$ is regular, there exists a unique $(\boldsymbol{\varphi},\zeta) \in \mathbf{H}_0^1(\Omega)\times L_0^{2}(\Omega)$ that solves \eqref{eq:first_deriv_S*}. To prove the desired regularity property, we rewrite \eqref{eq:first_deriv_S*} as the following Stokes system:
\[
\nu(\nabla \boldsymbol{\varphi}, \nabla \mathbf{v})_{\mathbf{L}^2(\Omega)} - (\zeta, \textnormal{div }\mathbf{v})_{L^2(\Omega)}
=  
 \langle \mathbf{g},\mathbf{v} \rangle
-
 \mathcal{H}(\mathbf{v}), 
\qquad
(q,\textnormal{div } \boldsymbol{\varphi})_{L^2(\Omega)} = 0,
\]
for all $(\mathbf{v},q)\in\mathbf{H}_0^1(\Omega)\times L_0^2(\Omega)$. Here, $\mathcal{H}: \mathbf{H}_0^1(\Omega) \rightarrow \mathbb{R}$ is given by $\mathcal{H}(\mathbf{v}) := b(\mathbf{y};\boldsymbol{\varphi},\mathbf{v}) +b(\boldsymbol{\varphi};\mathbf{y},\mathbf{v})$. In what follows, we analyze the boundedness of $\mathcal{H}$ in $\mathbf{W}^{-1,\mathtt{q}}(\Omega)$ in three dimensions; the arguments in two dimensions are simpler. Let $ d=3$. In view of the embedding $\mathbf{H}_0^1(\Omega) \hookrightarrow \mathbf{L}^6(\Omega)$ and the regularity results of Theorem \ref{thm:reg_velocity}, we obtain
\begin{align}
\| b(\boldsymbol{\varphi};\mathbf{y},\cdot) \|_{\mathbf{W}^{-1,\mathtt{q}}(\Omega)} 
& \leq 
\sup_{\mathbf{v} \in \mathbf{W}_0^{1,\mathtt{p}}(\Omega) } 
\frac{
\| \boldsymbol{\varphi} \|_{\mathbf{L}^6(\Omega)} \| \nabla  \mathbf{y} \|_{\mathbf{L}^3(\Omega)}\| \mathbf{v}\|_{\mathbf{L}^2(\Omega)}}
{ \|  \nabla \mathbf{v} |_{\mathbf{L}^{\mathtt{p}}(\Omega)}},
\label{eq:convective_estimate_1}
\\
\| b(\mathbf{y};\boldsymbol{\varphi},\cdot) \|_{\mathbf{W}^{-1,\mathtt{q}}(\Omega)} 
& \leq 
\sup_{\mathbf{v} \in \mathbf{W}_0^{1,\mathtt{p}}(\Omega) } 
\frac{ \| \mathbf{y} \|_{\mathbf{L}^{\infty}(\Omega)} \|  \nabla \boldsymbol{\varphi} \|_{\mathbf{L}^{2}(\Omega)}\| \mathbf{v}\|_{\mathbf{L}^2(\Omega)} }{\| \nabla \mathbf{v} \|_{\mathbf{L}^{\mathtt{p}} (\Omega)}}.
\label{eq:convective_estimate_2}
\end{align}
Let us now consider $\mathtt{q}$ such that $\mathtt{p} \geq 6/5$. Since $\mathbf{W}_{0}^{1,6/5}(\Omega) \hookrightarrow \mathbf{L}^{2}(\Omega)$, we can thus conclude that $\mathcal{H} \in \mathbf{W}^{-1,\mathtt{q}}(\Omega)$. We are thus able to refer to \cite[Corollary 1.7]{MR2987056}, with $\alpha=-1$ and $q=2$, to obtain 
the existence of $\mathtt{q}>3$ such that $(\boldsymbol{\varphi},\zeta)\in \mathbf{W}_{0}^{1,\mathtt{q}}(\Omega) \times L_{0}^{\mathtt{q}}(\Omega)$. 
Moreover, we have the bounds \cite[Theorem 2.9]{Brown_Shen} (see also \cite[Corollary 1.7]{MR2987056})
\begin{equation}
\label{eq:stab_rho'}
\begin{aligned}
\|\nabla\boldsymbol{\varphi}\|_{\mathbf{L}^{\mathtt{q}}(\Omega)} + \|\zeta\|_{L^{\mathtt{q}}(\Omega)} 
& \lesssim
\|\mathbf{g}\|_{\mathbf{W}^{-1,\mathtt{q}}(\Omega)}
+
\| \nabla  \mathbf{y} \|_{\mathbf{L}^{\kappa}(\Omega)} \|\nabla \boldsymbol{\varphi} \|_{\mathbf{L}^2(\Omega)}
\\
& \lesssim \|\mathbf{g}\|_{\mathbf{W}^{-1,\mathtt{q}}(\Omega)} ( 1 + \| \nabla  \mathbf{y} \|_{\mathbf{L}^{\kappa}(\Omega)} ),
\end{aligned}
\end{equation}
where $\kappa$ is as in the statement of Theorem \ref{thm:reg_velocity}. This concludes the proof.
\end{proof}

We now obtain the well-posedness of \eqref{eq:first_deriv_S_W01p} in $\mathbf{W}_0^{1,\mathsf{q}}(\Omega)\times L_0^{\mathsf{q}}(\Omega)$.

\begin{theorem}[well-posedness]
\label{thm:well-posedness-derivative}
Let $\mathbf{u} \in \mathbf{U}_{ad}$ and let $(\mathbf{y},p)$ be a solution to \eqref{eq:state_equations}  such that $\mathbf{y}$ is regular. Let $\mathsf{p}<d/(d-1)$ be such that it is arbitrarily close to $d/(d-1)$ and let $\mathsf{q}$ be such that $\mathsf{p}^{-1} + \mathsf{q}^{-1} = 1$. Let $\mathbf{g} \in \mathbf{W}^{-1,\mathsf{q}}(\Omega)$. Then, the problem: Find $(\boldsymbol{\varphi},\zeta)\in \mathbf{W}_0^{1,\mathsf{q}}(\Omega)\times L_0^{\mathsf{q}}(\Omega)$ such that
\begin{equation}
\label{eq:first_deriv_S_W01p}
\begin{aligned}
\nu(\nabla \boldsymbol{\varphi}, \nabla \mathbf{v})_{\mathbf{L}^2(\Omega)} + b(\mathbf{y};\boldsymbol{\varphi},\mathbf{v}) + b(\boldsymbol{\varphi};\mathbf{y},\mathbf{v}) - (\zeta, \textnormal{div }\mathbf{v})_{L^2(\Omega)}
& = \langle \mathbf{g},\mathbf{v} \rangle,
\\ 
(q,\textnormal{div }\boldsymbol{\varphi})_{L^2(\Omega)} & = 0,
\end{aligned}
\end{equation}
for all $(\mathbf{v},q)\in \mathbf{W}_0^{1,\mathsf{p}}(\Omega)\times L_0^{\mathsf{p}}(\Omega)$, is well posed.
\end{theorem}
\begin{proof}
Since $\mathbf{g} \in \mathbf{W}^{-1,\mathsf{q}}(\Omega)$ and $\mathbf{W}^{-1,\mathsf{q}}(\Omega) \hookrightarrow \mathbf{H}^{-1}(\Omega)$, there exists a unique pair $(\boldsymbol{\varphi},\zeta)\in \mathbf{H}_0^{1}(\Omega)\times L_0^{2}(\Omega)$ that solves \eqref{eq:first_deriv_S*}. As a consequence of the regularity arguments in the proof of Lemma \ref{lemma:regularity_estimate}, we deduce that $(\boldsymbol{\varphi},\zeta)\in \mathbf{W}_0^{1,\mathsf{q}}(\Omega)\times L_0^{\mathsf{q}}(\Omega)$. Consequently, problem \eqref{eq:first_deriv_S_W01p} admits at least one solution. Since problem \eqref{eq:first_deriv_S_W01p} is linear, estimate \eqref{eq:stab_rho'} shows that such a solution is unique. This concludes the proof.
\end{proof}


\section{The optimal control problem}\label{sec:ocp}
In this section, we analyze the following weak formulation of the pointwise tracking optimal control problem \eqref{def:cost_func}--\eqref{def:box_constraints}: Find
\begin{equation}\label{eq:weak_cost}
\min \{J(\mathbf{y},\mathbf{u}): (\mathbf{y},p,\mathbf{u})\in \mathbf{H}_0^1(\Omega)\times L_0^2(\Omega)\times\mathbf{U}_{ad}\},
\end{equation}
subject to the following weak formulation of the stationary Navier--Stokes equations: Find $(\mathbf{y},p)\in \mathbf{H}_0^1(\Omega)\times L_0^2(\Omega)$ such that
\begin{equation}
\label{eq:weak_st_eq}
\nu(\nabla \mathbf{y},\nabla \mathbf{v})_{\mathbf{L}^2(\Omega)} + b(\mathbf{y};\mathbf{y},\mathbf{v}) - (p,\text{div }\mathbf{v})_{L^2(\Omega)} = (\mathbf{u},\mathbf{v})_{\mathbf{L}^2(\Omega)}, ~~ (q,\text{div }\mathbf{y})_{L^2(\Omega)} = 0,
\end{equation}
for all $(\mathbf{v},q)\in \mathbf{H}_0^1(\Omega)\times L_0^2(\Omega)$. We immediately note that, in view of the results in Theorem \ref{thm:well_posedness_navier_stokes}, for every $\mathbf{u}\in\mathbf{L}^2(\Omega)$ there exists at least one solution $(\mathbf{y},p) \in \mathbf{H}_0^1(\Omega)\times L_0^2(\Omega)$ to problem \eqref{eq:weak_st_eq} \emph{without having to assume any smallness conditions}. Moreover, since the corresponding velocity component $\mathbf{y}\in \mathbf{C}(\bar{\Omega})$ (see Theorem \ref{thm:reg_velocity}), point evaluations of $\mathbf{y}$ in the cost functional $J$ are well-defined.


\subsection{Existence of optimal solutions} 
The existence of at least one optimal solution follows from the direct method of calculus of variations \cite[Chapter 1]{MR1201152}; see, for instance, \cite[Theorem 3.1]{MR2109045}. For the sake of completeness, we give a proof.

\begin{theorem}[existence of an optimal solution]\label{thm:existence_of_sol_ocp}
The control problem \eqref{eq:weak_cost}--\eqref{eq:weak_st_eq} admits at least one global solution $(\bar{\mathbf{y}},\bar{p},\bar{\mathbf{u}})\in \mathbf{H}_0^1(\Omega)\times L_0^2(\Omega)\times\mathbf{U}_{ad}$.
\end{theorem}
\begin{proof}
Let $\{(\mathbf{y}_{k},  p_k, \mathbf{u}_{k})\}_{k\in\mathbb{N}}$ be a minimizing sequence, i.e., for $k\in \mathbb{N}$, $(\mathbf{y}_{k},  p_k) \in \mathbf{H}_0^1(\Omega)\times L_0^2(\Omega)$ solves  \eqref{eq:weak_st_eq} where $\mathbf{u}$ is replaced by $\mathbf{u}_{k}$ and $\{ (\mathbf{y}_{k}, p_k,\mathbf{u}_k) \}_{k \in \mathbb{N}}$ is such that $J(\mathbf{y}_{k},\mathbf{u}_{k}) \rightarrow \mathfrak{i}:=\inf\{ J(\mathbf{y}, \mathbf{u}):(\mathbf{y}, p, \mathbf{u})\in \mathbf{H}_0^1(\Omega)\times L_0^2(\Omega) \times\mathbf{U}_{ad}\}$ as $k\uparrow \infty$. Since $\mathbf{U}_{ad}$ is weakly sequentially compact in $\mathbf{L}^2(\Omega)$, there exists a nonrelabeled subsequence $\{\mathbf{u}_{k}\}_{k\in\mathbb{N}} \subset \mathbf{U}_{ad}$ such that $\mathbf{u}_{k}\rightharpoonup \bar{\mathbf{u}}$ in $\mathbf{L}^2(\Omega)$; $\bar{\mathbf{u}}\in\mathbf{U}_{ad}$. On the other hand, Theorems \ref{thm:well_posedness_navier_stokes} and \ref{thm:reg_velocity} show that $\{(\mathbf{y}_{k},p_k)\}_{k\in\mathbb{N}}$ is uniformly bounded in $\mathbf{H}_0^1(\Omega)\cap \mathbf{W}^{1,\kappa}(\Omega)\times L_0^2(\Omega)$. Thus, we deduce the existence of a nonrelabeled subsequence $\{(\mathbf{y}_{k},p_k)\}_{k\in\mathbb{N}}$ such that $(\mathbf{y}_{k},p_k)\rightharpoonup (\bar{\mathbf{y}},\bar{p})$ in $\mathbf{H}_0^1(\Omega)\cap \mathbf{W}^{1,\kappa}(\Omega)\times L_0^2(\Omega)$ as $k\uparrow \infty$; $(\bar{\mathbf{y}},\bar{p})$ is the natural candidate for an optimal state. We now prove that $(\bar{\mathbf{y}},\bar{p})$ solves \eqref{eq:weak_st_eq} where $\mathbf{u}$ is replaced by $\bar{\mathbf{u}}$ and that $(\bar{\mathbf{y}},\bar{p},\bar{\mathbf{u}})$ is optimal.

With the weak convergence $(\mathbf{y}_{k},p_k)\rightharpoonup (\bar{\mathbf{y}},\bar{p})$ in $\mathbf{H}_0^1(\Omega)\cap \mathbf{W}^{1,\kappa}(\Omega)\times L_0^2(\Omega)$, as $k\uparrow \infty$, at hand, we obtain that, for every $\mathbf{v} \in \mathbf{H}_0^1(\Omega)$ and $q \in L_0^2(\Omega)$,
\begin{equation*}
|\nu (\nabla (\bar{\mathbf{y}} - \mathbf{y}_{k}),\nabla \mathbf{v})_{\mathbf{L}^2(\Omega)}| \rightarrow 0, 
~~
|(\bar{p} - p_{k},\text{div }\mathbf{v})_{L^2(\Omega)}| \rightarrow 0,
~~
|(q,\text{div}(\bar{\mathbf{y}} - \mathbf{y}_{k}))_{L^2(\Omega)}| \rightarrow 0,
\end{equation*}
as $k\uparrow \infty$. On the other hand, $\mathbf{u}_{k}\rightharpoonup \bar{\mathbf{u}}$ in $\mathbf{L}^2(\Omega)$ yields $|(\bar{\mathbf{u}} - \mathbf{u}_{k},\mathbf{v})_{\mathbf{L}^2(\Omega)}|\rightarrow 0$ as $k\uparrow \infty$. Thus, it is sufficient to analyze the convergence of the convective term. For this purpose, we use the weak convergence $\mathbf{y}_{k}\rightharpoonup \bar{\mathbf{y}}$ in $\mathbf{H}_0^1(\Omega)$ as $k\uparrow \infty$ and the compact embedding $\mathbf{H}_0^1(\Omega)\hookrightarrow \mathbf{L}^4(\Omega)$  \cite[Theorem 6.3, Part I]{MR2424078} to obtain
\begin{equation*}
|b(\bar{\mathbf{y}};\bar{\mathbf{y}},\mathbf{v}) - b(\mathbf{y}_{k};\mathbf{y}_{k},\mathbf{v})| \leq | b(\bar{\mathbf{y}};\bar{\mathbf{y}} - \mathbf{y}_{k},\mathbf{v})| + |b(\bar{\mathbf{y}} - \mathbf{y}_{k};\mathbf{y}_{k},\mathbf{v}) | \rightarrow 0, \quad k\uparrow \infty.
\end{equation*}

Finally, we prove that $(\bar{\mathbf{y}}, \bar p, \bar{\mathbf{u}})$ is optimal. Note that  $\mathbf{y}_{k}\rightharpoonup \bar{\mathbf{y}}$ in $\mathbf{W}^{1,\kappa}(\Omega)$, as $k\uparrow \infty$, together with the compact embedding $\mathbf{W}^{1,\kappa}(\Omega)\hookrightarrow \mathbf{C}(\bar{\Omega})$ \cite[Theorem 6.3, Part III]{MR2424078} yields the strong convergence $\mathbf{y}_{k}\rightarrow \bar{\mathbf{y}}$ in $\mathbf{C}(\bar{\Omega})$ as $k\uparrow \infty$. The fact that $\mathbf{L}^2(\Omega) \ni \mathbf{v} \mapsto \|\mathbf{v}\|_{\mathbf{L}^2(\Omega)}^2 \in \mathbb{R}$ is weakly lower semicontinuous allows us to conclude.
\end{proof}


\section{Optimality conditions}\label{sec:opt_cond}
In this section, we analyze first and second order optimality conditions for the optimal control problem \eqref{eq:weak_cost}--\eqref{eq:weak_st_eq}. 


\subsection{First order optimality conditions} 
We begin our studies with a result establishing differentiability properties for the solution map $\mathbf{u}\mapsto (\mathbf{y},p)$ associated with problem \eqref{eq:weak_st_eq} around a regular velocity field $\bar{\mathbf{y}}$. 

\subsubsection{Differentiability properties for the solution map}

\begin{theorem}[differentiability of $\mathbf{u}\mapsto (\mathbf{y},p)$]
\label{thm:properties_C_to_S}
Let $(\bar{\mathbf{y}},\bar p)$ be a solution to \eqref{eq:weak_st_eq} associated to $\bar{\mathbf{u}} \in \mathbf{U}_{ad}$. If $\bar{\mathbf{y}}$ is regular, then there exist open neighborhoods $\mathcal{O}(\bar{\mathbf{u}}) \subset \mathbf{L}^{2}(\Omega)$, $\mathcal{O}(\bar{\mathbf{y}}) \subset \mathbf{V}(\Omega)$, and $\mathcal{O}(\bar{p}) \subset L_0^2(\Omega)$ of  $\bar{\mathbf{u}}$, $\bar{\mathbf{y}}$, and $\bar{p}$, respectively, and a $C^2$ map
\begin{equation}
\mathcal{S}: \mathcal{O}(\bar{\mathbf{u}}) \to \mathcal{O}(\bar{\mathbf{y}}) \times \mathcal{O}(\bar{p}),
\label{eq:mathcalS}
\end{equation}
such that $\mathcal{S}(\bar{\mathbf{u}}) = (\bar{\mathbf{y}},\bar{p})$. In addition, the neighborhood $\mathcal{O}(\bar{\mathbf{u}})$ can be taken such that, for each $\mathbf{u}\in \mathcal{O}(\bar{\mathbf{u}})$, 
\begin{itemize}
\item the pair $(\mathbf{y},p) = \mathcal{S}(\mathbf{u})$ uniquely solves \eqref{eq:weak_st_eq} in $\mathcal{O}(\bar{\mathbf{y}}) \times \mathcal{O}(\bar{p})$,
\item $\mathcal{S}'(\mathbf{u}):\mathbf{H}^{-1}(\Omega)\to \mathbf{V}(\Omega)\times L_0^2(\Omega)$ is an isomorphism,
\item if $\mathbf{g}\in \mathbf{H}^{-1}(\Omega)$, then $(\boldsymbol{\varphi},\zeta)=\mathcal{S}'(\mathbf{u})\mathbf{g}\in \mathbf{H}_0^1(\Omega)\times L_0^2(\Omega)$ corresponds to the unique solution to \eqref{eq:first_deriv_S*}, and
\item if $\mathbf{g}_1,\mathbf{g}_2\in \mathbf{H}^{-1}(\Omega)$, then $(\boldsymbol{\Psi},\xi)=\mathcal{S}''(\mathbf{u})(\mathbf{g}_1,\mathbf{g}_2)\in \mathbf{H}_0^1(\Omega) \times L_0^2(\Omega)$ corresponds to the unique solution to
\begin{multline}\label{eq:second_deriv_2}
\nu(\nabla \boldsymbol{\Psi}, \nabla \mathbf{v})_{\mathbf{L}^2(\Omega)} + b(\mathbf{y};\boldsymbol{\Psi},\mathbf{v}) + b(\boldsymbol{\Psi};\mathbf{y},\mathbf{v}) - (\xi, \textnormal{div }\mathbf{v})_{L^2(\Omega)} \\
= - b(\boldsymbol{\varphi}_{\mathbf{g}_1};\boldsymbol{\varphi}_{\mathbf{g}_2},\mathbf{v}) - b(\boldsymbol{\varphi}_{\mathbf{g}_2};\boldsymbol{\varphi}_{\mathbf{g}_1},\mathbf{v}), \qquad (q,\textnormal{div }\boldsymbol{\Psi})_{L^2(\Omega)} = 0,
\end{multline}
for all $(\mathbf{v},q)\in\mathbf{H}_0^1(\Omega)\times L_0^2(\Omega)$, where $(\boldsymbol{\varphi}_{\mathbf{g}_i},\zeta_{\mathbf{g}_i})=\mathcal{S}'(\mathbf{u})\mathbf{g}_i$ and $i \in \{ 1,2 \}$.
\end{itemize}
\end{theorem}
\begin{proof}
The proof follows from a slight modification of that in \cite[Theorem 2.10]{MR3936891}, which is based on the implicit function theorem; see also \cite[Theorem 2.5]{MR2338434}. For the sake of brevity, we omit details.
\end{proof}

We now present a basic Lipschitz property for the map $\mathcal{S}$.

\begin{proposition}[$\mathcal{S}$ is Lipschitz]
In the framework of Theorem \ref{thm:properties_C_to_S}, we have 
\[
\| \nabla( \bar{\mathbf{y}}- \mathbf{y}) \|_{\mathbf{L}^2(\Omega)}
+
\| \bar{p}- p\|_{L^2(\Omega)}
\lesssim
\|  \bar{\mathbf{u}}- \mathbf{u} \|_{\mathbf{H}^{-1}(\Omega)},
\]
with a hidden constant depending on $\mathcal{S}'$ and a suitable neighborhood $\mathcal{O}(\bar{\mathbf{u}})$ of $\bar{\mathbf{u}}$.
\label{pro:Lipschitz_property}
\end{proposition}
\begin{proof}
Using the results of Theorem \ref{thm:properties_C_to_S}, we can choose an open, bounded, and convex neighborhood $\mathcal{O}(\bar{\mathbf{u}})$ of $\bar{\mathbf{u}}$ such that, for every $\mathbf{u} \in \mathcal{O}(\bar{\mathbf{u}})$, $\mathcal{S}'(\mathbf{u}):\mathbf{H}^{-1}(\Omega)\to \mathbf{V}(\Omega)\times L_0^2(\Omega)$ is an isomorphism and $\| \mathcal{S}'(\mathbf{u}) \| \leq \mathcal{M}_{\mathcal{S}}$. Here, $\mathcal{M}_{\mathcal{S}} >0$ and $\| \cdot \|$ denotes the norm in the space of linear and continuous operators from $\mathbf{H}^{-1}(\Omega)$ into $\mathbf{V}(\Omega)\times L_0^2(\Omega)$. An application of a mean value theorem for operators shows that
\[
\| \mathcal{S}(\bar{\mathbf{u}}) -  \mathcal{S}(\mathbf{u}) \|_{\mathbf{H}_0^1(\Omega) \times L^2(\Omega)} \leq \sup_{t \in [0,1]} \| \mathcal{S}'( (1-t) \bar{\mathbf{u}}  + t \mathbf{u} )\|  
\| \bar{\mathbf{u}} - \mathbf{u}  \|_{\mathbf{H}^{-1}(\Omega)}
\]
for every $\mathbf{u} \in \mathcal{O}(\bar{\mathbf{u}})$ \cite[Proposition 5.3.11]{MR2511061}. Since the line segment connecting $\bar{\mathbf{u}}$ and $\mathbf{u}$ is contained in $\mathcal{O}(\bar{\mathbf{u}})$, the bound $\| \mathcal{S}'(\mathbf{u}) \| \leq \mathcal{M}_{\mathcal{S}}$, for $\mathbf{u} \in \mathcal{O}(\bar{\mathbf{u}})$, allows us to conclude.
\end{proof}

We conclude this section with the following improved Lipschitz property.

\begin{theorem}[Lipschitz property]
In the framework of Theorem \ref{thm:properties_C_to_S}, we have 
\[
\| \nabla( \bar{\mathbf{y}}- \mathbf{y}) \|_{\mathbf{L}^{\kappa}(\Omega)}
+
\| \bar{p}- p\|_{L^{\kappa}(\Omega)}
\lesssim
\|  \bar{\mathbf{u}}- \mathbf{u} \|_{\mathbf{W}^{-1,\kappa}(\Omega)},
\]
where $\kappa$ is as in the statement of Theorem \ref{thm:reg_velocity} and the involved hidden constant depends on $\| \nabla \bar{\mathbf{y}} \|_{\mathbf{L}^{\kappa}(\Omega)}$, $\| \nabla \mathbf{y} \|_{\mathbf{L}^{\kappa}(\Omega)}$, $\mathcal{S}'$, and a suitable neighborhood $\mathcal{O}(\bar{\mathbf{u}})$ of $\bar{\mathbf{u}}$.
\label{thm:Lipschitz_property}
\end{theorem}
\begin{proof}
We begin the proof by noting that $(\bar{\mathbf{y}} - \mathbf{y},\bar{p} - p)\in \mathbf{H}_0^1(\Omega)\times L_0^2(\Omega)$ can be seen as the solution to the following Stokes problem:
\begin{multline}
\nu(\nabla (\bar{\mathbf{y}} - \mathbf{y}),\nabla \mathbf{v})_{\mathbf{L}^2(\Omega)} 
- (\bar{p} - p,\text{div }\mathbf{v})_{L^2(\Omega)}
= (\bar{\mathbf{u}} - \mathbf{u},\mathbf{v})_{\mathbf{L}^2(\Omega)}
\\
- 
b(\bar{\mathbf{y}};\bar{\mathbf{y}} - \mathbf{y} ,\mathbf{v}) 
- 
b(\bar{\mathbf{y}} - \mathbf{y};\mathbf{y},\mathbf{v}) 
, \qquad (q,\text{div}(\bar{\mathbf{y}}- \mathbf{y}))_{L^2(\Omega)} = 0,
\label{eq:Lipschitz_property_Stokes}
\end{multline}
for all $(\mathbf{v},q)\in \mathbf{H}_0^1(\Omega)\times L_0^2(\Omega)$. We now prove that the forcing term of the momentum equation in \eqref{eq:Lipschitz_property_Stokes} belongs to $\mathbf{W}^{-1,\kappa}(\Omega)$, where $\kappa$ is the same as in the statement of Theorem \ref{thm:reg_velocity}. Since $\bar{\mathbf{y}}, \mathbf{y}$ belong to $\mathbf{W}_0^{1,\kappa}(\Omega)$ (see Theorem \ref{thm:reg_velocity}) and $\mathbf{W}_0^{1,\kappa}(\Omega) \hookrightarrow \mathbf{C}(\bar \Omega)$, similar arguments to those elaborated in the proof of Lemma \ref{lemma:regularity_estimate} reveal that
\[
\|b(\bar{\mathbf{y}};\bar{\mathbf{y}} - \mathbf{y} ,\cdot) \|_{\mathbf{W}^{-1,\kappa}(\Omega)}
\lesssim 
\| \bar{\mathbf{y}} \|_{\mathbf{L}^{\infty}(\Omega)}
\| \nabla( \bar{\mathbf{y}} - \mathbf{y}) \|_{\mathbf{L}^{2}(\Omega)}
\lesssim
\| \bar{\mathbf{y}} \|_{\mathbf{L}^{\infty}(\Omega)}\| \bar{\mathbf{u}} - \mathbf{u}  \|_{\mathbf{W}^{-1,\kappa}(\Omega)},
\]
upon first utilizing that $\mathbf{W}_0^{1,\kappa'}(\Omega) \hookrightarrow \mathbf{L}^2(\Omega)$ and then the Lipschitz property of Proposition \ref{pro:Lipschitz_property} combined with the Sobolev embedding $\mathbf{W}^{-1,\kappa}(\Omega) \hookrightarrow \mathbf{H}^{-1}(\Omega)$. Similarly,
\[
\|b(\bar{\mathbf{y}} - \mathbf{y};\mathbf{y},\cdot) \|_{\mathbf{W}^{-1,\kappa}(\Omega)}
\lesssim 
\| \bar{\mathbf{y}} - \mathbf{y}\|_{\mathbf{L}^{\mu}(\Omega)}\| \nabla \mathbf{y} \|_{\mathbf{L}^{\kappa}(\Omega)},
\quad
\mu^{-1} + \kappa^{-1} + \sigma^{-1} = 1,
\]
upon utilizing that $\mathbf{W}_0^{1,\kappa'}(\Omega) \hookrightarrow \mathbf{L}^{\sigma}(\Omega)$. Here, $\kappa$ is dictated by Theorem \ref{thm:reg_velocity} and $\mu$ is such that $\mu < \infty$ when $d=2$ and $\mu \leq 6$ when $d=3$. We thus invoke the Sobolev embedding $\mathbf{H}_0^1(\Omega)\hookrightarrow \mathbf{L}^{\mu}(\Omega)$, the Lipschitz property of Proposition \ref{pro:Lipschitz_property}, and the fact that $\mathbf{W}^{-1,\kappa}(\Omega) \hookrightarrow \mathbf{H}^{-1}(\Omega)$ to obtain
\[
\|b(\bar{\mathbf{y}} - \mathbf{y};\mathbf{y},\cdot) \|_{\mathbf{W}^{-1,\kappa}(\Omega)} 
\lesssim 
\| \nabla( \bar{\mathbf{y}} - \mathbf{y}) \|_{\mathbf{L}^{2}(\Omega)}
\lesssim
\| \bar{\mathbf{u}} - \mathbf{u}  \|_{\mathbf{W}^{-1,\kappa}(\Omega)}.
\]
We have thus proved that the forcing term of the momentum equation in \eqref{eq:Lipschitz_property_Stokes} belongs to $\mathbf{W}^{-1,\kappa}(\Omega)$. Applying \cite[Corollary 1.7]{MR2987056}, with $\alpha=-1$ and $q=2$, we conclude.
\end{proof}


\subsubsection{Local solutions}
\label{subsec:local_solutions}
In the absence of convexity, we discuss optimality conditions in the context of local solutions in $\mathbf{L}^2(\Omega)$ \cite[Section 4.4.2]{Troltzsch}, \cite[Definition 3.1]{MR2338434}: We say that $(\bar{\mathbf{y}},\bar{p},\bar{\mathbf{u}})$ is a local solution to \eqref{eq:weak_cost}--\eqref{eq:weak_st_eq} if there exist neighborhoods $\mathcal{A} \subset \mathbf{H}_0^1(\Omega) \times L_0^2(\Omega)$ and $\mathcal{B} \subset \mathbf{L}^2(\Omega) \cap \mathbf{U}_{ad}$ of $(\bar{\mathbf{y}},\bar{p})$ and $\bar{\mathbf{u}}$, respectively, such that $J(\bar{\mathbf{y}},\bar{\mathbf{u}}) \leq J(\mathbf{y},\mathbf{u})$ for all $( (\mathbf{y},p),\mathbf{u}) \in \mathcal{A} \times \mathcal{B}$. If the inequality is strict for every $( (\mathbf{y},p),\mathbf{u}) \in \mathcal{A} \times \mathcal{B} \setminus \{ ((\bar{\mathbf{y}},\bar{p}),\bar{\mathbf{u}}) \}$, we say that $(\bar{\mathbf{y}},\bar{p},\bar{\mathbf{u}})$ is a strict local solution.

From now on we assume that $(\bar{\mathbf{y}},\bar{p},\bar{\mathbf{u}})$ is a local solution to \eqref{eq:weak_cost}--\eqref{eq:weak_st_eq} such that $\bar{\mathbf{y}}$ is regular. These local solutions are called \emph{local nonsingular solutions}. We note that in this framework the results of Theorem \ref{thm:properties_C_to_S} hold.

Having introduced the concept of local nonsingular solution, we consider the optimal control problem \cite[Section 3.1]{MR2338434}: Find
\begin{equation}\label{eq:weak_cost_reduced}
\min \{j(\mathbf{u}): \mathbf{u}\in \mathbf{U}_{ad} \cap \mathcal{O}(\bar{\mathbf{u}})\},
\end{equation}
where $j:\mathcal{O}(\bar{\mathbf{u}}) \to \mathbb{R}$ is defined by $j(\mathbf{u}):=J(\mathbf{y},\mathbf{u})$,
where $(\mathbf{y},p) = \mathcal{S}\mathbf{u}$. We note that $\bar{\mathbf{u}}$ is a local solution of \eqref{eq:weak_cost_reduced}. Moreover, in view of Theorem \ref{thm:properties_C_to_S}, $j$ is Gate\^aux differentiable in $\mathcal{O}(\bar{\mathbf{u}})$. Then, $\bar{\mathbf{u}}$ satisfies the variational inequality \cite[Lemma 4.18]{Troltzsch}
\begin{equation}\label{eq:variational_inequality}
j'(\bar{\mathbf{u}}) (\mathbf{u} - \bar{\mathbf{u}}) \geq 0 \quad \forall \mathbf{u} \in \mathbf{U}_{ad}.
\end{equation}


\subsubsection{The adjoint problem} In order to explore \eqref{eq:variational_inequality}, we introduce the \emph{adjoint variable} $(\mathbf{z},r)$ as the solution to the \emph{adjoint equations}: Find $(\mathbf{z},r)$ such that
\begin{equation}
-\nu \Delta \mathbf{z} - (\mathbf{y}\cdot\nabla) \mathbf{z} + (\nabla \mathbf{y})^{\intercal}  \mathbf{z}  +  \nabla r =  \sum_{t\in\mathcal{D}}(\mathbf{y}(t)-\mathbf{y}_{t})\delta_{t}   \text{ in }  \Omega, 
\quad 
\text{div }\mathbf{z} = 0\text{ in }  \Omega, 
\label{eq:adj_eq_strong}
\end{equation}
complemented with the Dirichlet boundary condition $\mathbf{z}  =  \mathbf{0}$ on $\partial\Omega$.
A weak formulation for \eqref{eq:adj_eq_strong} reads as follows: Find $(\mathbf{z},r) \in \mathbf{W}_{0}^{1,\mathsf{p} }(\Omega)\times L^{\mathsf{p}}_{0}(\Omega)$ such that
\begin{multline}\label{eq:adj_eq}
\nu(\nabla \mathbf{w}, \nabla \mathbf{z})_{\mathbf{L}^2(\Omega)} +
b(\mathbf{y};\mathbf{w},\mathbf{z}) + b(\mathbf{w};\mathbf{y},\mathbf{z}) - (r,\text{div } \mathbf{w})_{L^2(\Omega)} \\ =  \sum_{t\in\mathcal{D}}\langle (\mathbf{y}(t)-\mathbf{y}_{t})\delta_{t},\mathbf{w}\rangle_{\mathbf{W}^{-1,\mathsf{p} }(\Omega),\mathbf{W}^{1,\mathsf{q} }_0(\Omega)}, \qquad
(s,\text{div } \mathbf{z})_{L^2(\Omega)} = 0,
\end{multline}
for all $(\mathbf{w},s)\in \mathbf{W}_{0}^{1,\mathsf{q} }(\Omega)\times L_0^{\mathsf{q}}(\Omega)$. Here, $\mathsf{p} < d/(d-1)$ is arbitrarily close to $d/(d-1)$, $\mathsf{q}$ is such that $\mathsf{p}^{-1} + \mathsf{q}^{-1}  = 1$, and $\delta_{t}$ corresponds to the Dirac delta supported at the interior point $t \in \Omega$. We immediately notice that  $\mathsf{q} >d$.

We continue with the study of the well-posedness of the adjoint problem \eqref{eq:adj_eq}.

\begin{theorem}[well-posedness]
\label{thm:well-posedness-adjoint}
Let $\mathbf{u} \in \mathbf{U}_{ad}$ and let $(\mathbf{y},p)$ be a solution to  \eqref{eq:weak_st_eq} such that $\mathbf{y}$ is regular. Let $\mathsf{p} < d/(d-1)$ be arbitrarily close to $d/(d-1)$ and let $\mathbf{h}\in\mathbf{W}^{-1,\mathsf{p}}(\Omega)$. Thus, the weak problem: Find
\begin{multline}\label{eq:aux_eq_adj}
(\boldsymbol\Upsilon,\vartheta)\in \mathbf{W}_{0}^{1,\mathsf{p}}(\Omega)\times L^{\mathsf{p}}_{0}(\Omega):
\quad
\nu(\nabla \mathbf{w}, \nabla \boldsymbol\Upsilon)_{\mathbf{L}^2(\Omega)} +
b(\mathbf{y};\mathbf{w},\boldsymbol\Upsilon) + b(\mathbf{w};\mathbf{y},\boldsymbol\Upsilon)  \\
- (\vartheta,\textnormal{div }  \mathbf{w})_{L^2(\Omega)} =  \langle \mathbf{h},\mathbf{w}\rangle_{\mathbf{W}^{-1,\mathsf{p}}(\Omega),\mathbf{W}^{1,\mathsf{q}}_0(\Omega)}, 
\qquad
(s,\textnormal{div } \boldsymbol\Upsilon)_{L^2(\Omega)} = 0,
\end{multline}
for all $(\mathbf{w},s)\in \mathbf{W}_{0}^{1,\mathsf{q}}(\Omega)\times L_0^{\mathsf{q}}(\Omega)$, admits a unique solution. In addition, we have 
\begin{equation}\label{eq:stab_aux_adj}
\|\nabla\boldsymbol\Upsilon\|_{\mathbf{L}^{\mathsf{p}}(\Omega)} + \|\vartheta\|_{L^{\mathsf{p}}(\Omega)} 
\lesssim
\|\mathbf{h}\|_{\mathbf{W}^{-1,\mathsf{p}}(\Omega)}.
\end{equation}
\end{theorem}
\begin{proof}
We follow the duality argument elaborated in the proof of \cite[Theorem 2.9]{MR3936891} and proceed in three steps.

\emph{Step 1.} \emph{Well-posedness of \eqref{eq:aux_eq_adj} on $\mathbf{H}_0^1(\Omega)\times L_0^2(\Omega)$}. Let $\mathbf{h} \in \mathbf{H}^{-1}(\Omega)$. As a first step we prove that the problem: Find $(\boldsymbol{\Upsilon},\vartheta)\in\mathbf{H}_0^1(\Omega)\times L_0^2(\Omega)$ such that
\begin{multline}\label{eq:aux_eq_adj_H01*}
\nu(\nabla \mathbf{w}, \nabla \boldsymbol\Upsilon)_{\mathbf{L}^2(\Omega)} +
b(\mathbf{y};\mathbf{w},\boldsymbol\Upsilon) + b(\mathbf{w};\mathbf{y},\boldsymbol\Upsilon) - (\vartheta,\textnormal{div }  \mathbf{w})_{L^2(\Omega)} \\ =  \langle \mathbf{h},\mathbf{w}\rangle_{\mathbf{H}^{-1}(\Omega),\mathbf{H}_0^{1}(\Omega)}, 
\qquad
(s,\textnormal{div } \boldsymbol\Upsilon)_{L^2(\Omega)} = 0 \quad \forall (\mathbf{w},s)\in \mathbf{H}_0^1(\Omega)\times L_0^2(\Omega),
\end{multline}
is well posed. To accomplish this task, we introduce the map 
\[
S: \mathbf{V}(\Omega) \times L_0^2(\Omega) \rightarrow \mathbf{H}^{-1}(\Omega),
\quad
(\boldsymbol{\Upsilon},\vartheta) \mapsto -\nu \Delta \boldsymbol{\Upsilon} - (\mathbf{y}\cdot\nabla) \boldsymbol{\Upsilon}+ (\nabla \mathbf{y})^{\intercal}  \boldsymbol{\Upsilon} +  \nabla \vartheta,
\]
and prove that $S$ is an isomorphism. Note that $S$ is linear and bounded. Since $\mathbf{y}$ is regular, $T:  \mathbf{V}(\Omega) \times L_0^2(\Omega) \rightarrow \mathbf{H}^{-1}(\Omega)$, defined in \eqref{eq:T}, is an isomorphism. Let $(\boldsymbol{\phi},\pi) \in \mathbf{V}(\Omega) \times L_0^2(\Omega)$ be such that $T(\boldsymbol{\phi},\pi) = \mathbf{h}$ and let $(\boldsymbol{\Upsilon},\vartheta) \in \mathbf{V}(\Omega) \times L_0^2(\Omega)$. Integration by parts combined with the fact that $ \boldsymbol{\Upsilon}, \mathbf{y}, \boldsymbol{\phi} \in \mathbf{V}(\Omega)$ reveal that
\[
\langle \mathbf{h},\boldsymbol{\Upsilon} \rangle_{\mathbf{H}^{-1}(\Omega),\mathbf{H}_0^{1}(\Omega)}  
=
\langle T(\boldsymbol{\phi},\pi), \boldsymbol{\Upsilon} \rangle 
=
\langle  \boldsymbol{\phi},S(\boldsymbol{\Upsilon},\vartheta)  \rangle 
\leq \| \nabla \boldsymbol{\phi} \|_{\mathbf{L}^2(\Omega)}  \| S(\boldsymbol{\Upsilon},\vartheta) \|_{\mathbf{H}^{-1}(\Omega)}.
\]
The fact that $T$ is an isomorphism yields $\| \nabla \boldsymbol{\phi} \|_{\mathbf{L}^2(\Omega)} \lesssim \| \mathbf{h} \|_{\mathbf{H}^{-1}(\Omega)}$. This estimate, given the arbitrariness of $\mathbf{h} \in \mathbf{H}^{-1}(\Omega)$, implies that $ \| \nabla \boldsymbol{\Upsilon} \|_{\mathbf{L}^2(\Omega)} \lesssim \| S(\boldsymbol{\Upsilon},\vartheta) \|_{\mathbf{H}^{-1}(\Omega)}$. On the other hand, the definition of $S$ and estimates for the convective term yield
\begin{multline*}
\| \nabla \vartheta \|_{\mathbf{H}^{-1}(\Omega)} \leq \| S(\boldsymbol{\Upsilon},\vartheta) \|_{\mathbf{H}^{-1}(\Omega)} + \| \nu \Delta \boldsymbol{\Upsilon} + (\mathbf{y}\cdot\nabla) \boldsymbol{\Upsilon}- (\nabla \mathbf{y})^{\intercal}  \boldsymbol{\Upsilon}  \|_{\mathbf{H}^{-1}(\Omega)} 
\\
\leq 
\| S(\boldsymbol{\Upsilon},\vartheta) \|_{\mathbf{H}^{-1}(\Omega)}
+ \nu \|\Delta \boldsymbol{\Upsilon} \|_{\mathbf{H}^{-1}(\Omega)}
+C\|\nabla \boldsymbol{\Upsilon} \|_{\mathbf{L}^2(\Omega)} \|\nabla \mathbf{y} \|_{\mathbf{L}^2(\Omega)}
\lesssim  \| S(\boldsymbol{\Upsilon},\vartheta) \|_{\mathbf{H}^{-1}(\Omega)}.
\end{multline*}
Since the analysis is arbitrary on $(\boldsymbol{\Upsilon}, \vartheta)$, we have thus deduced that the map $S$ is such that $\| \nabla \boldsymbol{\Upsilon} \|_{\mathbf{L}^2(\Omega)} + \|  \vartheta \|_{L^{2}(\Omega)} \lesssim  \| S(\boldsymbol{\Upsilon},\vartheta) \|_{\mathbf{H}^{-1}(\Omega)}$ for every $(\boldsymbol{\Upsilon}, \vartheta) \in \mathbf{V}(\Omega) \times L_0^2(\Omega)$. This bound implies immediately that $S$ is injective in $\mathbf{V}(\Omega) \times L_0^2(\Omega)$ with closed range in $\mathbf{H}^{-1}(\Omega)$. To prove that $S$ is surjective we proceed as in the proof of Step 1 in \cite[Theorem 2.9]{MR3936891}. These arguments show that $S:\mathbf{V}(\Omega) \times L_0^2(\Omega) \rightarrow \mathbf{H}^{-1}(\Omega)$ is an isomorphism and that \eqref{eq:aux_eq_adj_H01*} is therefore well posed.

\emph{Step 2.} \emph{Existence of solutions in} $ \mathbf{W}_{0}^{1,\mathsf{p}}(\Omega)\times L_{0}^{\mathsf{p}}(\Omega)$. Let $\mathbf{h}\in \mathbf{W}^{-1,\mathsf{p}}(\Omega)$.  We prove that problem \eqref{eq:aux_eq_adj} has a solution. To do so, we use a density argument based on the fact that $\mathbf{H}^{-1}(\Omega)$ is dense in $\mathbf{W}^{-1,\mathsf{p}}(\Omega)$ to derive the existence of a sequence $\{\mathbf{h}_{k}\}_{k\in\mathbb{N}}\subset \mathbf{H}^{-1}(\Omega)$ such that $\mathbf{h}_{k}\rightarrow \mathbf{h}$ in $\mathbf{W}^{-1,\mathsf{p}}(\Omega)$ as $k\uparrow\infty$. From the results obtained in \emph{Step} 1, it follows that for every $k \in \mathbb{N}$ there exists a unique pair $(\boldsymbol\Upsilon_{k},\vartheta_{k})\in\mathbf{H}_0^1(\Omega)\times L_0^2(\Omega)$ that solves \eqref{eq:aux_eq_adj_H01*}, where $\mathbf{h}$ is replaced by $\mathbf{h}_{k}$.

We now prove the boundedness of $\{ (\boldsymbol\Upsilon_{k},\vartheta_{k}) \}_{k\in\mathbb{N}}$ in $\mathbf{W}_0^{1,\mathsf{p}}(\Omega)\times L_0^{\mathsf{p}}(\Omega)$ based on a duality argument. Let $\mathbf{g}\in \mathbf{W}^{-1,\mathsf{q}}(\Omega)$. The results of Theorem \ref{thm:well-posedness-derivative} guarantee the existence of a unique pair $(\boldsymbol{\varphi},\zeta)\in \mathbf{W}_0^{1,\mathsf{q}}(\Omega)\times L_0^{\mathsf{q}}(\Omega)$ that solves \eqref{eq:first_deriv_S_W01p}. We now set $(\mathbf{v},q)=(\boldsymbol\Upsilon_{k},0)$ in \eqref{eq:first_deriv_S_W01p} and $(\mathbf{w},s)=(\boldsymbol{\varphi},0)$ in \eqref{eq:aux_eq_adj_H01*}, and use the stability estimate \eqref{eq:stab_rho'} and the fact that $\{\mathbf{h}_{k}\}_{k\in\mathbb{N}}$ is convergent in $\mathbf{W}^{-1,\mathsf{p}}(\Omega)$ to obtain
\begin{multline}\label{eq:stab_aux_var}
\langle \mathbf{g}, \boldsymbol\Upsilon_{k} \rangle_{\mathbf{W}^{-1,\mathsf{q}}(\Omega),\mathbf{W}_0^{1,\mathsf{p}}(\Omega)}
=\mathcal{B}(\boldsymbol{\varphi},\boldsymbol\Upsilon_{k}) 
=
\langle \mathbf{h}_k,\boldsymbol{\varphi}
\rangle_{\mathbf{W}^{-1,\mathsf{p}}(\Omega),\mathbf{W}_0^{1,\mathsf{q}}(\Omega)} 
\\
\leq \|\mathbf{h}_k\|_{\mathbf{W}^{-1,\mathsf{p}}(\Omega)}\|\boldsymbol{\varphi}\|_{\mathbf{W}_0^{1,\mathsf{q}}(\Omega)} 
\lesssim \|\mathbf{g}\|_{\mathbf{W}^{-1,\mathsf{q}}(\Omega)} ( 1 + \| \nabla  \mathbf{y} \|_{\mathbf{L}^{\kappa}(\Omega)}).
\end{multline}
Since $\mathbf{g}$ is arbitrary, one can conclude that $\{ \nabla \boldsymbol{\Upsilon}_{k} \}_{k\in\mathbb{N}}$ is uniformly bounded in $\mathbf{L}^{\mathsf{p}}(\Omega)$. On the other hand, standard estimates for the convective term and the regularity result of Theorem \ref{thm:reg_velocity} combined with the inf-sup condition of \cite[Corollary B. 71]{MR2050138} show that $\|\vartheta_{k}\|_{L^{\mathsf{p}}(\Omega)} \lesssim 1$ for every $k\in\mathbb{N}$. Thus, there exists a nonrelabeled subsequence
\begin{equation}
\{ (\boldsymbol\Upsilon_{k},\vartheta_{k}) \}_{k\in\mathbb{N}} \subset \mathbf{W}_0^{1,\mathsf{p}}(\Omega)\times L_0^{\mathsf{p}}(\Omega):
\quad
(\boldsymbol\Upsilon_{k},\vartheta_{k}) 
\rightharpoonup (\boldsymbol\Upsilon,\vartheta) 
~\mathrm{in}~
\mathbf{W}_0^{1,\mathsf{p}}(\Omega)\times L_0^{\mathsf{p}}(\Omega)
\label{eq:upsilon_theta_convergence}
\end{equation}  
as $k \uparrow \infty$. The rest of this step is devoted to prove that $(\boldsymbol\Upsilon,\vartheta)$ solves  \eqref{eq:aux_eq_adj}. Applying \eqref{eq:upsilon_theta_convergence} and 
$\mathbf{h}_{k}\rightarrow \mathbf{h}$ in $\mathbf{W}^{-1,\mathsf{p}}(\Omega)$, we obtain for each $(\mathbf{w},s)\in\mathbf{W}_0^{1,\mathsf{q}}(\Omega)\times L_0^{\mathsf{q}}(\Omega)$ that
\begin{equation*}
\left|\nu(\nabla \mathbf{w}, \nabla (\boldsymbol\Upsilon_{k} - \boldsymbol\Upsilon))_{\mathbf{L}^2(\Omega)}\right| \to 0,  
\quad 
\left|\langle \mathbf{h}_{k} - \mathbf{h},\mathbf{w}\rangle_{\mathbf{W}^{-1,\mathsf{p}}(\Omega),\mathbf{W}_0^{1,\mathsf{q}}(\Omega)}\right| \to 0,
\end{equation*}
$|(\vartheta_k - \vartheta,\textnormal{div }\mathbf{w})_{L^2(\Omega)}| \to 0$, and $|(s,\textnormal{div} (\boldsymbol\Upsilon_{k} - \boldsymbol\Upsilon))_{L^2(\Omega)}| \to 0$ as $k\uparrow\infty$. To analyze the convective terms, we use the compact Sobolev embedding $\mathbf{W}^{1,\mathsf{p}}(\Omega)\hookrightarrow \mathbf{L}^{\mathsf{r}}(\Omega)$, which holds for every $\mathsf{r} < d\mathsf{p}/(d-\mathsf{p})$ \cite[Theorem 6.3, Part I]{MR2424078}, and the regularity results of Theorem \ref{thm:reg_velocity}, which guarantee that $\mathbf{y} \in \mathbf{W}^{1,3}(\Omega) \cap \mathbf{L}^{\infty}(\Omega)$, to conclude that
\begin{equation*}
\begin{aligned}
\left|b\left(\mathbf{y};\mathbf{w},\boldsymbol\Upsilon_{k} - \boldsymbol\Upsilon\right)\right|
& 
\lesssim
\|\mathbf{y}\|_{\mathbf{L}^\infty(\Omega)} \|\nabla \mathbf{w}\|_{\mathbf{L}^{\mathsf{q}}(\Omega)}\|\boldsymbol\Upsilon_{k} - \boldsymbol\Upsilon\|_{\mathbf{L}^{\mathsf{p}}(\Omega)} \to 0,
\quad
 k\uparrow\infty,
 \\
 \left|b\left(\mathbf{w};\mathbf{y},\boldsymbol\Upsilon_{k} - \boldsymbol\Upsilon\right)\right|
 & 
 \lesssim 
 \|\mathbf{w}\|_{\mathbf{L}^\infty(\Omega)} \|\nabla \mathbf{y}\|_{\mathbf{L}^{3}(\Omega)}\|\boldsymbol\Upsilon_{k} - \boldsymbol\Upsilon\|_{\mathbf{L}^{3/2}(\Omega)} \to 0,
 \quad
 k\uparrow\infty.
 \end{aligned}
\end{equation*}
A collection of these arguments reveals that the pair $(\boldsymbol\Upsilon,\vartheta)$ solves \eqref{eq:aux_eq_adj}.

\emph{Step 3.} \emph{Uniqueness and stability.} We begin by proving that $(\boldsymbol\Upsilon,\vartheta)$ satisfies the stability bound \eqref{eq:stab_aux_adj}. Let $\mathbf{g}\in \mathbf{W}^{-1,\mathsf{q}}(\Omega)$ and let $(\boldsymbol{\varphi},\zeta) \in \mathbf{W}_0^{1,\mathsf{q}}(\Omega)\times L_0^{\mathsf{q}}(\Omega)$ be the unique solution to \eqref{eq:first_deriv_S_W01p}. By arguments similar to those used to obtain the relations and inequalities in \eqref{eq:stab_aux_var}, one can deduce that
\begin{equation*}
\langle \mathbf{g}, \boldsymbol\Upsilon\rangle_{\mathbf{W}^{-1,\mathsf{q}}(\Omega),\mathbf{W}_0^{1,\mathsf{p}}(\Omega)}
=
\langle \mathbf{h},\boldsymbol{\varphi} \rangle_{\mathbf{W}^{-1,\mathsf{p}}(\Omega),\mathbf{W}_0^{1,\mathsf{q}}(\Omega)}
\lesssim \|\mathbf{h}\|_{\mathbf{W}^{-1,\mathsf{p}}(\Omega)}
\|\mathbf{g}\|_{\mathbf{W}^{-1,\mathsf{q}}(\Omega)}.
\end{equation*}
Since $\mathbf{g}$ is arbitrary, we can conclude that $\| \nabla \boldsymbol\Upsilon\|_{\mathbf{L}^{\mathsf{p}}(\Omega)}\lesssim \|\mathbf{h}\|_{\mathbf{W}^{-1,\mathsf{p}}(\Omega)}$. From this, due to an inf-sup condition \cite[Corollary B. 71]{MR2050138}, we obtain $\|\vartheta\|_{L^{\mathsf{p}}(\Omega)} \lesssim \|\mathbf{h}\|_{\mathbf{W}^{-1,\mathsf{p}}(\Omega)}$. A collection of these estimates immediately yields the bound \eqref{eq:stab_aux_adj}.

We now show the uniqueness of solutions to \eqref{eq:aux_eq_adj}. To do so, we assume that there is another pair $({\boldsymbol\Upsilon}_{\star},\vartheta_{\star})\in \mathbf{W}_0^{1,\mathsf{p}}(\Omega)\times L_0^{\mathsf{p}}(\Omega)$ that solves \eqref{eq:aux_eq_adj}. Let $\mathbf{g}\in \mathbf{W}^{-1,\mathsf{q}}(\Omega)$ and let $(\boldsymbol{\varphi},\zeta)\in \mathbf{W}_0^{1,\mathsf{q}}(\Omega)\times L_0^{\mathsf{q}}(\Omega)$ be the unique solution to \eqref{eq:first_deriv_S_W01p}. Set $(\mathbf{v},q)=({\boldsymbol\Upsilon}_{\star} - \boldsymbol\Upsilon,0)$ in \eqref{eq:first_deriv_S_W01p} and $(\mathbf{w},s)=(\boldsymbol{\varphi},0)$ in the problem that  $({\boldsymbol\Upsilon}_{\star} - \boldsymbol\Upsilon,\vartheta_{\star} - \vartheta)$ solves to obtain
\begin{equation}
\langle \mathbf{g}, {\boldsymbol\Upsilon}_{\star} - \boldsymbol\Upsilon\rangle_{\mathbf{W}^{-1,\mathsf{q}}(\Omega),\mathbf{W}_0^{1,\mathsf{p}}(\Omega)}
=
0.
\label{eq:identity=0}
\end{equation}
Since $\mathbf{g}$ is arbitrary, \eqref{eq:identity=0} holds for every $\mathbf{g} \in \mathbf{W}^{-1,\mathsf{q}}(\Omega)$. Consequently, $ {\boldsymbol\Upsilon}_{\star} = \boldsymbol\Upsilon$ and $\vartheta_{\star} = \vartheta$; the latter follows from an inf-sup condition. We have thus proved that \eqref{eq:aux_eq_adj} admits a unique solution. This concludes the proof.
\end{proof}


\subsubsection{The variational inequality}

We are now in a position to establish necessary first order optimality conditions. For this purpose, in the context of the setting described in section \ref{subsec:local_solutions}, we introduce the map 
\begin{equation}
\mathcal{G}: \mathcal{O}(\bar{\mathbf{u}}) \subset \mathbf{L}^2(\Omega ) \rightarrow \mathcal{O}(\bar{\mathbf{y}}) \subset \mathbf{V}(\Omega) :
\qquad
\mathbf{u} \mapsto \mathbf{y},
\label{eq:mathcal_G}
\end{equation}
where $\mathbf{y}$ corresponds to the velocity component of the pair $(\mathbf{y},p) = \mathcal{S}(\mathbf{u})$.

\begin{theorem}[first order necessary optimality conditions]
\label{thm:optimality_cond}
If $(\bar{\mathbf{y}}, \bar p, \bar{\mathbf{u}})$ denotes a local nonsingular solution to \eqref{eq:weak_cost}--\eqref{eq:weak_st_eq}, then
$\bar{\mathbf{u}}$ satisfies the variational inequality
\begin{equation}\label{eq:var_ineq}
(\bar{\mathbf{z}}+\alpha \bar{\mathbf{u}},\mathbf{u}-\bar{\mathbf{u}})_{\mathbf{L}^2(\Omega)}\geq 0 \quad \forall \mathbf{u}\in \mathbf{U}_{ad},
\end{equation}
where $(\bar{\mathbf{z}}, \bar{r}) \in \mathbf{W}_{0}^{1,\mathsf{p}}(\Omega) \times L_0^{\mathsf{p}}(\Omega)$ is the solution to \eqref{eq:adj_eq}, where $\mathbf{y}$ is replaced by $\bar{\mathbf{y}}= \mathcal{G} \bar{\mathbf{u}}$.
\end{theorem}
\begin{proof}
We begin the proof by invoking the map $\mathcal{G}$ defined in \eqref{eq:mathcal_G} and rewriting the variational inequality \eqref{eq:variational_inequality} as follows:
\begin{equation}\label{eq:derivative_cost}
\sum_{t\in\mathcal{D}}\left(\mathcal{G}\bar{\mathbf{u}}(t)-\mathbf{y}_t\right)\cdot\mathcal{G}'(\bar{\mathbf{u}})(\mathbf{u}-\bar{\mathbf{u}})(t)+\alpha(\bar{\mathbf{u}},\mathbf{u}-\bar{\mathbf{u}})_{\mathbf{L}^2(\Omega)}\geq 0 \qquad \forall \mathbf{u}\in\mathbf{U}_{ad}.
\end{equation}

Since $\alpha(\bar{\mathbf{u}},\mathbf{u}-\bar{\mathbf{u}})_{\mathbf{L}^2(\Omega)}$ is already present in \eqref{eq:var_ineq}, we focus on the sum on the left-hand side of \eqref{eq:derivative_cost}. Define $(\boldsymbol{\varphi},\zeta):=\mathcal{S}'(\bar{\mathbf{u}})(\mathbf{u}-\bar{\mathbf{u}})$ and note that $(\boldsymbol{\varphi},\zeta)$ corresponds to the unique solution to \eqref{eq:first_deriv_S*}, replacing $\mathbf{y}$ and $\mathbf{g}$ by $\bar{\mathbf{y}}$ and $\mathbf{u}-\bar{\mathbf{u}}$, respectively. Since $\mathbf{u}, \bar{\mathbf{u}} \in \mathbf{W}^{-1,\mathsf{q}}(\Omega)$, Theorem \ref{thm:well-posedness-derivative} shows that $(\boldsymbol{\varphi},\zeta) \in \mathbf{W}_0^{1,\mathsf{q}}(\Omega)\times L_0^{\mathsf{q}}(\Omega)$ corresponds to the unique solution to \eqref{eq:first_deriv_S_W01p}. We can thus substitute $(\mathbf{w},s) = (\boldsymbol{\varphi},0)$ into the adjoint problem \eqref{eq:adj_eq} and use the fact that $(\bar{r},\text{div }\boldsymbol{\varphi})_{L^2(\Omega)}=0$ to obtain
\begin{equation}\label{eq:adj_eq_w=z}
\nu(\nabla\boldsymbol{\varphi}, \nabla \bar{\mathbf{z}})_{\mathbf{L}^2(\Omega)} + b(\bar{\mathbf{y}};\boldsymbol{\varphi},\bar{\mathbf{z}}) + b(\boldsymbol{\varphi};\bar{\mathbf{y}},\bar{\mathbf{z}})
= 
\sum_{t\in\mathcal{D}}(\bar{\mathbf{y}}(t)-\mathbf{y}_{t})\cdot\boldsymbol{\varphi}(t).
\end{equation}
On the other hand, we set  $(\mathbf{v},q) = (\bar{\mathbf{z}},0)$ into the problem that $(\boldsymbol{\varphi},\zeta) = \mathcal{S}'(\bar{\mathbf{u}})(\mathbf{u}-\bar{\mathbf{u}})$ solves, i.e., problem \eqref{eq:first_deriv_S_W01p} with $\mathbf{y}=\bar{\mathbf{y}}$ and $\mathbf{g}=\mathbf{u} - \bar{\mathbf{u}}$ to arrive at
\begin{equation}\label{eq:aux_adjoint_2}
\nu(\nabla\boldsymbol{\varphi}, \nabla \bar{\mathbf{z}})_{\mathbf{L}^2(\Omega)} + b(\bar{\mathbf{y}};\boldsymbol{\varphi},\bar{\mathbf{z}}) + b(\boldsymbol{\varphi};\bar{\mathbf{y}},\bar{\mathbf{z}})
= 
(\mathbf{u} - \bar{\mathbf{u}},\bar{\mathbf{z}})_{\mathbf{L}^2(\Omega)}.
\end{equation}
Notice that, since $\zeta \in L_0^{\mathsf{q}}(\Omega)$, $(\zeta,\text{div }\bar{\mathbf{z}})_{L^2(\Omega)} = 0$.

Consequently, the desired result \eqref{eq:var_ineq} follows from \eqref{eq:derivative_cost}, \eqref{eq:adj_eq_w=z}, and \eqref{eq:aux_adjoint_2}.
\end{proof}

The following projection formula result is classical:
If $\bar{\mathbf{u}}$ denotes a locally optimal control for problem \eqref{eq:weak_cost}--\eqref{eq:weak_st_eq}, then \cite[Section 4.6]{Troltzsch}, \cite[equation (3.9)]{MR2338434}
\begin{equation}\label{eq:projection_control} 
\bar{\mathbf{u}}(x):=\Pi_{[\textbf{a},\textbf{b}]}(-\alpha^{-1}\bar{\mathbf{z}}(x)) \textrm{ a.e.}~x \in \Omega,
\end{equation}
where $\Pi_{[\textbf{a},\textbf{b}]} : \mathbf{L}^1(\Omega) \rightarrow  \mathbf{U}_{ad}$ is defined by $\Pi_{[\textbf{a},\textbf{b}]}(\mathbf{v}) := \min\{ \textbf{b}, \max\{ \mathbf{v}, \textbf{a}\} \}$ a.e.~in  $\Omega$. We immediately notice the following basic regularity result: Since $\bar{\mathbf{z}}\in \mathbf{W}^{1,\mathsf{p}}(\Omega)$, where $\mathsf{p} < d/(d-1)$ is arbitrarily close to $d/(d-1)$, then $\bar{\mathbf{u}}$ also belongs to $\mathbf{W}^{1,\mathsf{p}}(\Omega)$; see \cite[Theorem A.1]{MR1786735} and \cite[Theorem 1]{MR1173747}.


\subsection{Second order optimality conditions}\label{sec:2nd_order}
In this section, we derive necessary and sufficient second order optimality conditions. 


\subsubsection{Preliminaries}
We begin our studies with an auxiliary estimate.

\begin{lemma}[auxiliary estimate]
\label{lemma:auxiliary_estimate}
Let $(\bar{\mathbf{y}}, \bar{p}, \bar{\mathbf{u}})$ be a local nonsingular solution to \eqref{eq:weak_cost}--\eqref{eq:weak_st_eq}. Let $\mathbf{u}_{1},\mathbf{u}_{2}\in \mathcal{O}(\bar{\mathbf{u}})$ and $\mathbf{g}\in \mathbf{L}^{2}(\Omega)$. Let $(\mathbf{y}_{1},p_1)=\mathcal{S}(\mathbf{u}_{1})$, $(\mathbf{y}_{2},p_2)=\mathcal{S}(\mathbf{u}_{2})$, $(\boldsymbol{\varphi}_1,\zeta_1)= \mathcal{S}'(\mathbf{u}_{1})\mathbf{g}$, and $(\boldsymbol{\varphi}_2,\zeta_2) = \mathcal{S}'(\mathbf{u}_{2})\mathbf{g}$. Let $\mathsf{p}<d/(d-1)$ and $\mathsf{q}$ be such that $\mathsf{p}$ is arbitrarily close to $d/(d-1)$ and $\mathsf{p}^{-1} + \mathsf{q}^{-1} = 1$. Then, we have the estimate
\begin{equation}\label{eq:estimate_varphis}
\|\nabla(\boldsymbol{\varphi}_1 - \boldsymbol{\varphi}_2)\|_{\mathbf{L}^{\mathsf{q}}(\Omega)} \lesssim \|\mathbf{u}_{1} - \mathbf{u}_{2}\|_{\mathbf{L}^2(\Omega)} \|\mathbf{g}\|_{\mathbf{L}^2(\Omega)}.
\end{equation}
\end{lemma}
\begin{proof}
We begin the proof by noting that the pair $(\boldsymbol{\varphi}_1 - \boldsymbol{\varphi}_2,\zeta_1-\zeta_2)$ solves the following weak problem: Find $(\boldsymbol{\varphi}_1 - \boldsymbol{\varphi}_2,\zeta_1-\zeta_2) \in \mathbf{H}_0^1(\Omega)\times L_0^2(\Omega)$ such that
\begin{multline*}
\nu(\nabla(\boldsymbol{\varphi}_1 - \boldsymbol{\varphi}_2), \nabla \mathbf{v})_{\mathbf{L}^2(\Omega)} \!+ b(\mathbf{y}_{1};\boldsymbol{\varphi}_1 - \boldsymbol{\varphi}_2,\mathbf{v}) \!+ b(\boldsymbol{\varphi}_1 - \boldsymbol{\varphi}_2;\mathbf{y}_{1},\mathbf{v}) 
- (\zeta_1-\zeta_2, \textnormal{div }\mathbf{v})_{L^2(\Omega)}
\\
= b(\mathbf{y}_{2} - \mathbf{y}_{1};\boldsymbol{\varphi}_2,\mathbf{v}) + b(\boldsymbol{\varphi}_2;\mathbf{y}_{2} - \mathbf{y}_{1},\mathbf{v}), \quad  (q,\textnormal{div}(\boldsymbol{\varphi}_1 - \boldsymbol{\varphi}_2))_{L^2(\Omega)} = 0,
\end{multline*}
for all $(\mathbf{v},q)\in\mathbf{H}_0^1(\Omega)\times L_0^2(\Omega)$. Since $b(\mathbf{y}_{2} - \mathbf{y}_{1};\boldsymbol{\varphi}_2,\cdot) + b(\boldsymbol{\varphi}_2;\mathbf{y}_{2} - \mathbf{y}_{1},\cdot) \in \mathbf{W}^{-1,\mathsf{q}}(\Omega)$, the arguments in Lemma \ref{lemma:regularity_estimate} show that $(\boldsymbol{\varphi}_1 - \boldsymbol{\varphi}_2,\zeta_1-\zeta_2)\in \mathbf{W}_0^{1,\mathsf{q}}(\Omega)\times L_0^{\mathsf{q}}(\Omega)$ and
\begin{equation*}
\|\nabla(\boldsymbol{\varphi}_1 - \boldsymbol{\varphi}_2)\|_{\mathbf{L}^{\mathsf{q}}(\Omega)} 
\lesssim 
\|\mathbf{y}_{2} - \mathbf{y}_{1}\|_{\mathbf{L}^{\infty}(\Omega)} \|\nabla\boldsymbol{\varphi}_2\|_{\mathbf{L}^{2}(\Omega)} + \| \nabla \boldsymbol{\varphi}_2 \|_{\mathbf{L}^{2}(\Omega)}
\|\nabla(\mathbf{y}_{2} - \mathbf{y}_{1})\|_{\mathbf{L}^{3}(\Omega)},
\end{equation*}
upon using \eqref{eq:convective_estimate_1} and \eqref{eq:convective_estimate_2}, where $\mathtt{q}$ is replaced by $\mathsf{q}$. Since $\mathbf{u}_2 \in \mathcal{O}(\bar{\mathbf{u}})$ and $(\boldsymbol{\varphi}_2,\zeta_2) = \mathcal{S}'(\mathbf{u}_{2})\mathbf{g}$, the results of Theorem \ref{thm:properties_C_to_S} yield $\|\nabla\boldsymbol{\varphi}_2\|_{\mathbf{L}^{2}(\Omega)} \lesssim \|\mathbf{g}\|_{\mathbf{H}^{-1}(\Omega)}$. Thus,
\begin{equation}\label{eq:stab_hat_tilde}
\|\nabla(\boldsymbol{\varphi}_1 - \boldsymbol{\varphi}_2)\|_{\mathbf{L}^{\mathsf{q}}(\Omega)} 
\lesssim 
\left(\|\mathbf{y}_{2} - \mathbf{y}_{1}\|_{\mathbf{L}^{\infty}(\Omega)} + \|\nabla(\mathbf{y}_{2} - \mathbf{y}_{1})\|_{\mathbf{L}^{3}(\Omega)}\right) \|\mathbf{g}\|_{\mathbf{H}^{-1}(\Omega)}.
\end{equation}

The control of $\|\nabla(\mathbf{y}_{2} - \mathbf{y}_{1})\|_{\mathbf{L}^{3}(\Omega)}$ and $\|\mathbf{y}_{2} - \mathbf{y}_{1}\|_{\mathbf{L}^{\infty}(\Omega)}$ follows from a direct application of the Lipschitz property derived in Theorem \ref{thm:Lipschitz_property}. In fact,
\begin{equation}\label{eq:Lipschitz_NS_aux}
\|\mathbf{y}_{2} - \mathbf{y}_{1}\|_{\mathbf{L}^{\infty}(\Omega)}
\lesssim
\|\nabla(\mathbf{y}_{2} - \mathbf{y}_{1})\|_{\mathbf{L}^{\kappa}(\Omega)} \lesssim 
\|\mathbf{u}_{2} - \mathbf{u}_{1}\|_{\mathbf{W}^{-1,\kappa}(\Omega)}
\lesssim
\|\mathbf{u}_{2} - \mathbf{u}_{1}\|_{\mathbf{L}^2(\Omega)},
\end{equation}
where we have also used that $\mathbf{u}_{2} - \mathbf{u}_{1} \in  \mathbf{L}^2(\Omega)$ and that $\mathbf{L}^2(\Omega) \hookrightarrow \mathbf{W}^{-1,\kappa}(\Omega)$. Here, $\kappa$ is as in the statement of Theorem \ref{thm:reg_velocity}: $\kappa > 4$ when $d=2$ and $\kappa>3$ when $d=3$.

The desired estimate \eqref{eq:estimate_varphis} thus follows from replacing suitable estimates from \eqref{eq:Lipschitz_NS_aux} into the estimate \eqref{eq:stab_hat_tilde}. This concludes the proof.
\end{proof}

\begin{theorem}[$j$ is of class $C^2$ and $j''$ is locally Lipschitz]
\label{thm:diff_properties_j} 
Let $(\bar{\mathbf{y}}, \bar{p}, \bar{\mathbf{u}})$ be a local nonsingular solution to \eqref{eq:weak_cost}--\eqref{eq:weak_st_eq}. Let $\mathsf{p}<d/(d-1)$ be arbitrarily close to $d/(d-1)$. Then, the functional $j: \mathcal{O}(\bar{\mathbf{u}}) \rightarrow \mathbb{R}$ is of class $C^2$. Moreover, for $\mathbf{u}\in \mathcal{O}(\bar{\mathbf{u}})$ and $\mathbf{g} \in \mathbf{L}^{2}(\Omega)$, we have the identity
\begin{equation}\label{eq:charac_j2}
j''(\mathbf{u})\mathbf{g}^2
=
\alpha\|\mathbf{g}\|_{\mathbf{L}^2(\Omega)}^2 - 2b(\boldsymbol{\varphi};\boldsymbol{\varphi},\mathbf{z}) + \sum_{t\in \mathcal{D}}\boldsymbol{\varphi}^2(t),
\end{equation}
where $(\mathbf{z},r)$ solves \eqref{eq:adj_eq} and $(\boldsymbol{\varphi},\zeta)=\mathcal{S}'(\mathbf{u})\mathbf{g}$. Finally, for $\mathbf{u}_{1},\mathbf{u}_{2}\in\mathcal{O}(\bar{\mathbf{u}})$, we have
\begin{equation}\label{eq:continuity_of_j2}
|j''(\mathbf{u}_1)\mathbf{g}^2-j''(\mathbf{u}_2)\mathbf{g}^2|
\lesssim
\|\mathbf{u}_1-\mathbf{u}_2\|_{\mathbf{L}^2(\Omega)}\|\mathbf{g}\|_{\mathbf{L}^2(\Omega)}^2.
\end{equation}
\end{theorem}
\begin{proof}
The fact that $j$ is of class $C^2$ on $\mathcal{O}(\bar{\mathbf{u}})$ is a direct consequence of the differentiability properties of the map $\mathcal{S}$ given in Theorem \ref{thm:properties_C_to_S}, so it suffices to derive the identity \eqref{eq:charac_j2} and the estimate \eqref{eq:continuity_of_j2}. To accomplish this task, we start with a simple calculation which shows that for $\mathbf{u}\in\mathcal{O}(\bar{\mathbf{u}})$ and $\mathbf{g}\in \mathbf{L}^2(\Omega)$ we have 
\begin{equation}\label{eq:charac_j2_prev}
j''(\mathbf{u})\mathbf{g}^2
=
\alpha\|\mathbf{g}\|_{\mathbf{L}^2(\Omega)}^2
+
\sum_{t\in \mathcal{D}}
\left[\mathbf{\Psi}(t)\cdot(\mathcal{G}\mathbf{u}(t) - \mathbf{y}_t) +\boldsymbol{\varphi}^2(t)\right],
\end{equation}
where $(\boldsymbol{\varphi},\zeta) = \mathcal{S}'(\mathbf{u})\mathbf{g} \in \mathbf{H}_0^{1}(\Omega) \times L_0^2(\Omega)$ and $(\mathbf{\Psi},\xi) = \mathcal{S}''(\mathbf{u})\mathbf{g}^2  \in \mathbf{H}_0^{1}(\Omega) \times L_0^2(\Omega)$ solve \eqref{eq:first_deriv_S*} and \eqref{eq:second_deriv_2}, respectively. We immediately note that, since $\mathbf{g} \in \mathbf{W}^{-1,\mathsf{q}}(\Omega)$ with $\mathsf{p}^{-1} + \mathsf{q}^{-1} = 1$, analogous arguments to those in Lemma \ref{lemma:regularity_estimate} show that $\boldsymbol{\varphi},\mathbf{\Psi} \in \mathbf{W}_0^{1,\mathsf{q}}(\Omega)\hookrightarrow \mathbf{C}(\bar \Omega)$. Consequently, the point evaluations of $\boldsymbol{\varphi}$ and $\mathbf{\Psi}$ in \eqref{eq:charac_j2_prev} are well defined. We now set $(\mathbf{w},s) = (\mathbf{\Psi},0)$ in \eqref{eq:adj_eq} and invoke an approximation argument based on the fact that $(\mathbf{\Psi},\xi)\in \mathbf{W}_0^{1,\mathsf{q}}(\Omega) \times L_0^{\mathsf{q}}(\Omega)$, which essentially allows us to set $(\mathbf{v},q) = (\mathbf{z},0)$ in \eqref{eq:second_deriv_2} to obtain
\[
\sum_{t\in \mathcal{D}}\mathbf{\Psi}(t)\cdot(\mathcal{G}\mathbf{u}(t) - \mathbf{y}_t) = - 2b(\boldsymbol{\varphi};\boldsymbol{\varphi},\mathbf{z}).
\]
Replacing the previous identity in \eqref{eq:charac_j2_prev}, we get \eqref{eq:charac_j2}.

We now prove \eqref{eq:continuity_of_j2}. Let $\mathbf{u}_1,\mathbf{u}_2 \in \mathcal{O(\bar{\mathbf{u}}})$ and $\mathbf{g}\in \mathbf{L}^2(\Omega)$. Define $(\boldsymbol{\varphi}_1,\zeta_1)= \mathcal{S}'(\mathbf{u}_{1})\mathbf{g}$ and $(\boldsymbol{\varphi}_2,\zeta_2) = \mathcal{S}'(\mathbf{u}_{2})\mathbf{g}$. Given the identity \eqref{eq:charac_j2}, we obtain
\begin{multline}\label{eq:j''u1_j''u2}
[j''(\mathbf{u}_{1}) - j''(\mathbf{u}_{2})]\mathbf{g}^2
\!=
2b(\boldsymbol{\varphi}_2 - \boldsymbol{\varphi}_1;\boldsymbol{\varphi}_2,\mathbf{z}_{2}) 
+
2b(\boldsymbol{\varphi}_1;\boldsymbol{\varphi}_2 - \boldsymbol{\varphi}_1,\mathbf{z}_{2})
\\
+
2b(\boldsymbol{\varphi}_1;\boldsymbol{\varphi}_1,\mathbf{z}_{2} - \mathbf{z}_{1}) 
+ 
\sum_{t\in \mathcal{D}}(\boldsymbol{\varphi}_1^2(t)  - \boldsymbol{\varphi}_2^2(t))
=: \mathbf{I} + \mathbf{II} + \mathbf{III} + \mathbf{IV}.
\end{multline}
Here, $(\mathbf{z}_{i},r_i) \in \mathbf{W}_0^{1,\mathsf{p}}(\Omega) \times L_0^p(\Omega)$ denotes the solution to \eqref{eq:adj_eq}, where $\mathbf{y}$ is replaced by $\mathbf{y}_{i} = \mathcal{G}(\mathbf{u}_i)$; $i\in\{1,2\}$. We bound each term on the right-hand side of \eqref{eq:j''u1_j''u2} separately. We begin by estimating $\mathbf{I}$. To do this, we use H\"older's inequality, \eqref{eq:estimate_varphis} combined with the Sobolev embedding $\mathbf{W}_0^{1,\mathsf{q}}(\Omega)\hookrightarrow \mathbf{C}(\bar{\Omega})$, and $\|\nabla\boldsymbol{\varphi}_2\|_{\mathbf{L}^{2}(\Omega)} \lesssim \|\mathbf{g}\|_{\mathbf{L}^2(\Omega)}$:
\begin{equation*}
\label{eq:estimation_of_I}
|\mathbf{I}| 
\lesssim
\|\boldsymbol{\varphi}_2 - \boldsymbol{\varphi}_1\|_{\mathbf{L}^\infty(\Omega)}
\|\nabla\boldsymbol{\varphi}_2\|_{\mathbf{L}^{2}(\Omega)}
\|\mathbf{z}_{2}\|_{\mathbf{L}^{2}(\Omega)}
\lesssim
\|\mathbf{u}_1-\mathbf{u}_2\|_{\mathbf{L}^2(\Omega)}
\|\mathbf{g}\|_{\mathbf{L}^2(\Omega)}^2
\|\mathbf{z}_{2}\|_{\mathbf{L}^{2}(\Omega)}.
\end{equation*}
We control  $\|\mathbf{z}_{2}\|_{\mathbf{L}^{2}(\Omega)}$ in view of the Sobolev embedding $\mathbf{W}_0^{1,\mathsf{p}}(\Omega) \hookrightarrow \mathbf{L}^2(\Omega)$, the estimate \eqref{eq:stab_aux_adj}, and the regularity results of Theorem \ref{thm:reg_velocity}. These arguments yield
\begin{equation*}
\label{eq:estimate_z2}
\|\mathbf{z}_{2}\|_{\mathbf{L}^{2}(\Omega)} 
\lesssim
\|\mathbf{y}_{2}\|_{\mathbf{L}^{\infty}(\Omega)} 
+ 
\Lambda(\{ \mathbf{y}_t\})
\lesssim
\| \nabla \mathbf{y}_{2} \|_{\mathbf{L}^{\kappa}(\Omega)}
+ \Lambda(\{ \mathbf{y}_t\}),
\quad
\Lambda(\{ \mathbf{y}_t \}) = \sum_{t\in\mathcal{D}}|\mathbf{y}_{t}|,
\end{equation*}
where $\{ \mathbf{y}_t\} = \{ \mathbf{y}_t\}_{t \in \mathcal{D}}$. Thus, $|\mathbf{I}| \lesssim \|\mathbf{u}_1-\mathbf{u}_2\|_{\mathbf{L}^2(\Omega)}\|\mathbf{g}\|_{\mathbf{L}^2(\Omega)}^2$, with a hidden constant depending on $\Omega$, $\mathbf{a}$, $\mathbf{b}$, and $\Lambda(\{ \mathbf{y}_t\})$. The control of $\mathbf{II}$ follows similar arguments. In fact, in view of the estimate \eqref{eq:estimate_varphis} we have
\begin{equation*}
|\mathbf{II}| 
\lesssim
\|\boldsymbol{\varphi}_1\|_{\mathbf{L}^\infty(\Omega)}\|\nabla(\boldsymbol{\varphi}_2 - \boldsymbol{\varphi}_1)\|_{\mathbf{L}^{\mathsf{q}}(\Omega)}\| \mathbf{z}_{2} \|_{\mathbf{L}^{\mathsf{p}}(\Omega)}
\lesssim
\|\mathbf{u}_1-\mathbf{u}_2\|_{\mathbf{L}^2(\Omega)}\|\mathbf{g}\|_{\mathbf{L}^2(\Omega)}^2.
\end{equation*}
Note that we also used $\|\boldsymbol{\varphi}_1 \|_{\mathbf{L}^{\infty}(\Omega)} \lesssim \|\nabla\boldsymbol{\varphi}_1 \|_{\mathbf{L}^{\mathsf{q}}(\Omega)} \lesssim \|\mathbf{g}\|_{\mathbf{L}^2(\Omega)}$; the second estimate is a consequence of Theorem \ref{thm:well-posedness-derivative}. To control $\mathbf{III}$ we proceed as follows:
\begin{equation}\label{eq:estimation_of_III}
|\mathbf{III}| 
\lesssim 
\|\boldsymbol{\varphi}_1\|_{\mathbf{L}^{\infty}(\Omega)}\|\nabla\boldsymbol{\varphi}_1\|_{\mathbf{L}^{\mathsf{q}}(\Omega)}\|\mathbf{z}_{2} - \mathbf{z}_{1}\|_{\mathbf{L}^{\mathsf{p}}(\Omega)}
\lesssim
\|\mathbf{z}_{2} - \mathbf{z}_{1}\|_{\mathbf{L}^{\mathsf{p}}(\Omega)}\|\mathbf{g}\|_{\mathbf{L}^2(\Omega)}^2.
\end{equation}
Thus, it is sufficient to bound $\|\mathbf{z}_{2} - \mathbf{z}_{1}\|_{\mathbf{L}^{\mathsf{p}}(\Omega)}$. Note that $(\mathbf{z}_{2} - \mathbf{z}_{1},r_{2} - r_{1})\in \mathbf{W}_0^{1,\mathsf{p}}\times L_0^\mathsf{p}(\Omega)$ can be seen as the solution to
\begin{multline*}
\nu(\nabla \mathbf{w}, \nabla (\mathbf{z}_{2} - \mathbf{z}_{1}))_{\mathbf{L}^2(\Omega)} +
b(\mathbf{y}_{2};\mathbf{w},\mathbf{z}_{2} - \mathbf{z}_{1})  + b(\mathbf{w};\mathbf{y}_{2},\mathbf{z}_{2} - \mathbf{z}_{1}) - (r_{2} - r_{1},\text{div } \mathbf{w})_{L^2(\Omega)} 
\\
= b(\mathbf{y}_{1} - \mathbf{y}_{2};\mathbf{w},\mathbf{z}_{1}) + b(\mathbf{w};\mathbf{y}_{1} - \mathbf{y}_{2},\mathbf{z}_{1})
+ \sum_{t\in\mathcal{D}}\langle (\mathbf{y}_{2} - \mathbf{y}_{1})(t)\delta_{t},\mathbf{w}\rangle_{\mathbf{W}^{-1,\mathsf{p}}(\Omega),\mathbf{W}^{1,\mathsf{q}}_0(\Omega)},
\end{multline*}
and $(s,\text{div}(\mathbf{z}_{2} - \mathbf{z}_{1}))_{L^2(\Omega)} = 0$ for all $(\mathbf{w},s)\in \mathbf{W}_{0}^{1,\mathsf{q}}(\Omega)\times L_0^{\mathsf{q}}(\Omega)$. The results in Theorem \ref{thm:well-posedness-adjoint} guarantee that the above problem is well posed. In particular, we have a stability bound which can be combined with H\"older's inequality to obtain 
\begin{equation*}
\|\nabla(\mathbf{z}_{2} - \mathbf{z}_{1})\|_{\mathbf{L}^{\mathsf{p}}(\Omega)} 
\lesssim
\|\mathbf{y}_{1} - \mathbf{y}_{2}\|_{\mathbf{L}^{\infty}(\Omega)}(1 + \|\mathbf{z}_{1}\|_{\mathbf{L}^{\mathsf{p}}(\Omega)}) + \|\nabla(\mathbf{y}_{1} - \mathbf{y}_{2})\|_{\mathbf{L}^{\mathsf{q}}(\Omega)} \|\mathbf{z}_{1}\|_{\mathbf{L}^{\mathsf{p}}(\Omega)}.
\end{equation*}
The term $\| \mathbf{z}_{1} \|_{\mathbf{L}^{\mathsf{p}}(\Omega)}$ can be controlled by the optimal control problem data as described above. The terms $\|\mathbf{y}_{1} - \mathbf{y}_{2}\|_{\mathbf{L}^{\infty}(\Omega)}$ and $\|\nabla(\mathbf{y}_{1} - \mathbf{y}_{2})\|_{\mathbf{L}^{\mathsf{q}}(\Omega)}$ can be bounded using the estimates in \eqref{eq:Lipschitz_NS_aux}. A collection of these arguments yields 
$\|\nabla(\mathbf{z}_{2} - \mathbf{z}_{1})\|_{\mathbf{L}^{\mathsf{p}}(\Omega)} 
\lesssim \|\mathbf{u}_{2} - \mathbf{u}_{1}\|_{\mathbf{L}^{2}(\Omega)}$. We now invoke a Poincar\'e inequality and replace the bound previously obtained into \eqref{eq:estimation_of_III} to obtain $|\mathbf{III}|  \lesssim  \|\mathbf{u}_{1} - \mathbf{u}_{2}\|_{\mathbf{L}^{2}(\Omega)}\|\mathbf{g}\|_{\mathbf{L}^2(\Omega)}^2$. Finally, we control $\mathbf{IV}$: 
\begin{equation*}
|\mathbf{IV}|
\lesssim
\|\boldsymbol{\varphi}_1  - \boldsymbol{\varphi}_2\|_{\mathbf{L}^{\infty}(\Omega)}(\|\boldsymbol{\varphi}_1\|_{\mathbf{L}^{\infty}(\Omega)} + \|\boldsymbol{\varphi}_2\|_{\mathbf{L}^{\infty}(\Omega)})
\lesssim 
\|\mathbf{u}_{1} - \mathbf{u}_{2}\|_{\mathbf{L}^{2}(\Omega)}\|\mathbf{g}\|_{\mathbf{L}^2(\Omega)}^2,
\end{equation*}
upon using $\mathbf{W}_0^{1,\mathsf{q}}(\Omega)\hookrightarrow \mathbf{C}(\bar{\Omega})$ and estimate \eqref{eq:estimate_varphis}. Note that the results of Theorem \ref{thm:well-posedness-derivative} guarantee that $\|\boldsymbol{\varphi}_1 \|_{\mathbf{L}^{\infty}(\Omega)} + \|\boldsymbol{\varphi}_2 \|_{\mathbf{L}^{\infty}(\Omega)}  \lesssim 
\|\nabla \boldsymbol{\varphi}_1 \|_{\mathbf{L}^{\mathsf{q}}(\Omega)} + \|\nabla \boldsymbol{\varphi}_2 \|_{\mathbf{L}^{\mathsf{q}}(\Omega)}
\lesssim
\|\mathbf{g}\|_{\mathbf{L}^2(\Omega)}$.
The desired bound \eqref{eq:continuity_of_j2} follows from \eqref{eq:j''u1_j''u2} and a collection of the the estimates obtained for $\mathbf{I}$, $\mathbf{II}$, $\mathbf{III}$, and $\mathbf{IV}$. This concludes the proof.
\end{proof}

We end this section with the following convergence property.

\begin{lemma}[convergence property]
Let $(\bar{\mathbf{y}}, \bar{p}, \bar{\mathbf{u}})$ be a local nonsingular solution to \eqref{eq:weak_cost}--\eqref{eq:weak_st_eq}. If $\mathbf{g}_{k} \rightharpoonup \mathbf{g}$ in $\mathbf{L}^2(\Omega)$ as $k \uparrow \infty$, then $j''(\bar{\mathbf{u}})\mathbf{g}^2 \leq \liminf_{k \uparrow \infty}j''(\bar{\mathbf{u}})\mathbf{g}_{k}^2$.
\label{lemma:covergence_property}
\end{lemma}
\begin{proof}
Given the identity \eqref{eq:charac_j2}, we decompose $j''(\bar{\mathbf{u}})\mathbf{g}^2 - j''(\bar{\mathbf{u}})(\mathbf{g}_{k})^2$ as follows:
\begin{multline}\label{eq:j''u-j''uk}
j''(\bar{\mathbf{u}})\mathbf{g}^2 - j''(\bar{\mathbf{u}})(\mathbf{g}_{k})^2
=
\alpha (\|\mathbf{g}\|_{\mathbf{L}^2(\Omega)}^2 - \|\mathbf{g}_{k}\|_{\mathbf{L}^2(\Omega)}^2) - 2b(\boldsymbol{\varphi};\boldsymbol{\varphi} - \boldsymbol{\varphi}_{k},\bar{\mathbf{z}})\\
 - 2b(\boldsymbol{\varphi} - \boldsymbol{\varphi}_{k};\boldsymbol{\varphi}_{k},\bar{\mathbf{z}}) + \sum_{t\in \mathcal{D}}(\boldsymbol{\varphi} - \boldsymbol{\varphi}_{k})(t)\cdot(\boldsymbol{\varphi} + \boldsymbol{\varphi}_{k})(t) =:\mathbf{I}_k + \mathbf{II}_k + \mathbf{III}_k + \mathbf{IV}_k.
\end{multline}
Here, $(\boldsymbol{\varphi},\zeta) = \mathcal{S}'(\bar{\mathbf{u}}) \mathbf{g}$, $(\boldsymbol{\varphi}_k,\zeta_k) = \mathcal{S}'(\bar{\mathbf{u}}) \mathbf{g}_k$, and $(\bar{\mathbf{z}},\bar{r})$ denotes the solution to \eqref{eq:adj_eq} where $\mathbf{y}$ is replaced by $\bar{\mathbf{y}}$. Let us now note that $(\boldsymbol{\varphi} - \boldsymbol{\varphi}_{k},\zeta - \zeta_{k})$ solves the following weak problem: Find $(\boldsymbol{\varphi} - \boldsymbol{\varphi}_{k},\zeta - \zeta_{k}) \in \mathbf{H}_0^1(\Omega)\times L_0^2(\Omega)$ such that
\begin{multline*}
\nu(\nabla(\boldsymbol{\varphi} - \boldsymbol{\varphi}_{k}), \nabla \mathbf{v})_{\mathbf{L}^2(\Omega)} + b(\bar{\mathbf{y}};\boldsymbol{\varphi} - \boldsymbol{\varphi}_{k},\mathbf{v}) + b(\boldsymbol{\varphi} - \boldsymbol{\varphi}_{k};\bar{\mathbf{y}},\mathbf{v}) \\
- (\zeta - \zeta_{k}, \textnormal{div }\mathbf{v})_{L^2(\Omega)}
= (\mathbf{g} - \mathbf{g}_{k},\mathbf{v})_{\mathbf{L}^2(\Omega)}, \qquad  (q,\textnormal{div}(\boldsymbol{\varphi} - \boldsymbol{\varphi}_{k}))_{L^2(\Omega)} = 0,
\end{multline*}
for all $(\mathbf{v},q)\in\mathbf{H}_0^1(\Omega)\times L_0^2(\Omega)$. Since $\mathbf{g} - \mathbf{g}_k \in \mathbf{W}^{-1,\mathsf{q}}(\Omega)$, where $\mathsf{q} > d$ is arbitrarily close to $d$, the results of Theorem \ref{thm:well-posedness-derivative} guarantee that $\boldsymbol{\varphi} - \boldsymbol{\varphi}_{k}\in \mathbf{W}_0^{1,\mathsf{q}}(\Omega)$ and that the convergence property $\mathbf{g}_k \rightharpoonup \mathbf{g}$ in $\mathbf{L}^2(\Omega)$ implies that $\boldsymbol{\varphi}_{k}\rightharpoonup \boldsymbol{\varphi}$ in $\mathbf{W}_0^{1,\mathsf{q}}(\Omega)$ as $k \uparrow \infty$.

We are now in a position to examine $\mathbf{I}_k$, $\mathbf{II}_k$, $\mathbf{III}_k$, and $\mathbf{IV}_k$ as $k \uparrow \infty$. First, $\mathbf{g}_k \rightharpoonup \mathbf{g}$ in $\mathbf{L}^2(\Omega)$ combined with the fact that the square of $\|\cdot\|_{\mathbf{L}^2(\Omega)}$ is weakly lower semicontinuous in $\mathbf{L}^2(\Omega)$ shows that $\liminf_{k \uparrow \infty}\mathbf{I}_k \leq 0$. The terms $\mathbf{II}_k$, $\mathbf{III}_k$, and $\mathbf{IV}_k$ can be treated simultaneously in view of 
\[
\boldsymbol{\varphi}_{k}\rightharpoonup \boldsymbol{\varphi} \textrm{ in } \mathbf{W}_0^{1,\mathsf{q}}(\Omega) \textrm{ as } k \uparrow \infty,
~~
\|\nabla\boldsymbol{\varphi}\|_{\mathbf{L}^{\mathsf{q}}(\Omega)} \lesssim \|\mathbf{g}\|_{\mathbf{L}^2(\Omega)},
~~
\|\nabla\boldsymbol{\varphi}_k\|_{\mathbf{L}^{\mathsf{q}}(\Omega)} 
\lesssim \mathfrak{M} \quad \forall k \in \mathbb{N},
\]
and the compact Sobolev embedding $\mathbf{W}_0^{1,\mathsf{q}}(\Omega)\hookrightarrow \mathbf{C}(\Omega)$. These arguments allow us to show that $|\mathbf{II}_k|$, $|\mathbf{III}_k|$, $|\mathbf{IV}_k| \rightarrow 0$
as $k \uparrow \infty$. This concludes the proof.
\end{proof}


\subsubsection{Second order necessary conditions}

Let $(\bar{\mathbf{y}}, \bar{p}, \bar{\mathbf{u}})$ be a local nonsingular solution to \eqref{eq:weak_cost}--\eqref{eq:weak_st_eq}. Recall that $\bar{\mathbf{u}}$ satisfies \eqref{eq:var_ineq}; see Theorem \ref{thm:optimality_cond}. Define $\bar{\mathbf{d}} :=  \bar{\mathbf{z}} + \alpha \bar{\mathbf{u}}$. It follows immediately from the inequality \eqref{eq:var_ineq} that, a.e.~$x\in \Omega$,
\begin{equation}\label{eq:derivative_j}
\bar{\mathbf{d}}_{i}(x) 
\begin{cases}
= 0  \quad \text{if} \quad \mathbf{a}_{i} < \bar{\mathbf{u}}_{i}(x)  < \mathbf{b}_{i}, 
\\
\geq 0 \quad   \text{if} \quad \bar{\mathbf{u}}_{i}(x)  = \mathbf{a}_{i}, 
\\
\leq  0 \quad \text{if} \quad \bar{\mathbf{u}}_{i}(x)  = \mathbf{b}_{i}.
\end{cases}
\end{equation}
Here, $i\in\{1,\ldots,d\}$ and $\bar{\mathbf{d}}_{i}$ is the $i$-th component of the vector-valued function $\bar{\mathbf{d}}$.

Let us now introduce the \emph{cone of critical directions}
\begin{equation}\label{def:critical_cone}
\mathbf{C}_{\bar{\mathbf{u}}}:=\{\mathbf{g}\in \mathbf{L}^2(\Omega) \text{ satisfying } \eqref{eq:sign_cond} \text{ and } \mathbf{g}_{i}(x) = 0 \text{ if } \bar{\mathbf{d}}_{i}(x) \neq 0\},
\end{equation}
where $i\in\{1,\ldots,d\}$ and condition \eqref{eq:sign_cond} reads as follows:
\begin{equation}
\label{eq:sign_cond}
\mathbf{g}_{i}(x)
\geq 0 \text{ a.e.}~x\in\Omega \text{ if } \bar{\mathbf{u}}_{i}(x)=\mathbf{a}_{i},
\qquad
\mathbf{g}_{i}(x)
\leq 0 \text{ a.e.}~x\in\Omega \text{ if } \bar{\mathbf{u}}_{i}(x)=\mathbf{b}_{i}.
\end{equation}

We now follow \cite{MR2338434,MR2902693} and derive second order necessary optimality conditions.

\begin{theorem}[second order necessary optimality conditions]
\label{thm:nec_opt_cond}
Let $(\bar{\mathbf{y}}, \bar{p}, \bar{\mathbf{u}})$ be a local nonsingular solution to \eqref{eq:weak_cost}--\eqref{eq:weak_st_eq}. Then,
$
j''(\bar{\mathbf{u}})\mathbf{g}^2 \geq 0 
$
for all $\mathbf{g}\in \mathbf{C}_{\bar{\mathbf{u}}}$.
\end{theorem}
\begin{proof}
Let $\mathbf{g}$ be an arbitrary element in the cone of critical directions $\mathbf{C}_{\bar{\mathbf{u}}}$. With $\mathbf{g}$ at hand, we introduce the vector-valued function $\mathbf{g}^k$, for $k \in \mathbb{N}$, as follows:
\begin{equation*}
\mathbf{g}_{i}^{k}(x):=
\begin{cases}
\qquad \quad 0 \quad &\text{ if }\quad x: \mathbf{a}_{i} < \bar{\mathbf{u}}_i(x) < \mathbf{a}_{i} + k^{-1}, 
\quad 
\mathbf{b}_{i} - k^{-1} < \bar{\mathbf{u}}_{i}(x) < \mathbf{b}_{i}, 
\\
\Pi_{[-k,k]}(\mathbf{g}_{i}(x)) &\text{ otherwise}. 
\end{cases}
\end{equation*}
Here, $i\in\{1,\ldots,d\}$. Since $\mathbf{g}$ belongs to $\mathbf{C}_{\bar{\mathbf{u}}}$, the definition of $\mathbf{g}^k$ immediately implies that, for every $k \in \mathbb{N}$, $\mathbf{g}^{k} \in \mathbf{C}_{\bar{\mathbf{u}}}$. On the other hand,  for every $i\in\{1,\ldots,d\}$ and $k \in \mathbb{N}$, $|\mathbf{g}_{i}^{k}(x)| \leq |\mathbf{g}_{i}(x)|$ and, as $k \uparrow \infty$, $\mathbf{g}_{i}^{k}(x)\rightarrow \mathbf{g}_{i}(x)$ for a.e.~$x \in \Omega$. Therefore $\mathbf{g}^{k} \to \mathbf{g}$ in $\mathbf{L}^2(\Omega)$ as $k \uparrow \infty$. We now note that based on simple calculations, we can conclude that, for $\gamma \in (0, k^{-2}]$, $\bar{\mathbf{u}} + \gamma \mathbf{g}^{k}\in\mathbf{U}_{ad}$. So we can rely on the fact that $\bar{\mathbf{u}}$ is a local minimum to conclude that $j(\bar{\mathbf{u}}) \leq j(\bar{\mathbf{u}}+\gamma \mathbf{g}^{k})$ for $\gamma$ small enough. With this bound at hand, we apply Taylor's theorem to $j$ at $\bar{\mathbf{u}}$ and use $j'(\bar{\mathbf{u}})\mathbf{g}^{k}=0$, which follows from the fact that $\mathbf{g}^{k} \in \mathbf{C}_{\bar{\mathbf{u}}}$, to conclude that, for $\gamma$ sufficiently small and $\theta_{k} \in (0,1)$,
\begin{equation*}
0
\leq 
j(\bar{\mathbf{u}}+\gamma \mathbf{g}^{k})-j(\bar{\mathbf{u}}) 
=
\gamma j'(\bar{\mathbf{u}})\mathbf{g}^{k}+\frac{\gamma^2}{2}j''(\bar{\mathbf{u}}+\gamma\theta_{k}\mathbf{g}^{k})(\mathbf{g}^{k})^2
=
\frac{\gamma^2}{2}j''(\bar{\mathbf{u}}+\gamma\theta_{k}\mathbf{g}^{k})(\mathbf{g}^{k})^2.
\end{equation*}
From the previous inequality and relations we can deduce that $j''(\bar{\mathbf{u}}+\gamma\theta_{k}\mathbf{g}^{k})(\mathbf{g}^{k})^2 \geq 0$. We invoke \eqref{eq:continuity_of_j2} and let $\gamma\downarrow 0$ to obtain $j''(\bar{\mathbf{u}})(\mathbf{g}^{k})^2 \geq 0$. We now let $k \uparrow \infty$ and invoke similar arguments to those elaborated in the proof of Lemma \ref{lemma:covergence_property} to conclude that $j''(\bar{\mathbf{u}})\mathbf{g}^2 \geq 0$. Since $\mathbf{g}$ is arbitrary, we have thus arrived at the desired result.
\end{proof}


\subsubsection{Second order sufficient conditions}
We now formulate and derive our second order sufficient conditions for local optimality with a minimal gap with respect to the necessary conditions derived in Theorem \ref{thm:nec_opt_cond}.

\begin{theorem}[second order sufficient optimality conditions]
\label{thm:suff_opt_cond}
Let $(\bar{\mathbf{y}}, \bar{p}, \bar{\mathbf{u}})$ be a local nonsingular solution to \eqref{eq:weak_cost}--\eqref{eq:weak_st_eq}. If $j''(\bar{\mathbf{u}})\mathbf{g}^2 > 0$ for all $\mathbf{g}\in \mathbf{C}_{\bar{\mathbf{u}}}\setminus\{\mathbf{0}\}$, then there exist $\mathfrak{m} > 0$ and $\mathfrak{s} > 0$ such that
\begin{equation}\label{eq:optimal_minimum}
j(\mathbf{u})\geq j(\bar{\mathbf{u}})+\tfrac{\mathfrak{m}}{2}\|\mathbf{u}-\bar{\mathbf{u}}\|_{\mathbf{L}^2(\Omega)}^2\quad \forall \mathbf{u}\in \mathbf{U}_{ad}: \|\mathbf{u}-\bar{\mathbf{u}}\|_{\mathbf{L}^2(\Omega)}\leq \mathfrak{s}.
\end{equation}
In particular, $\bar{\mathbf{u}}$ is a locally optimal control in the sense of $\mathbf{L}^2(\Omega)$. 
\end{theorem}
\begin{proof}
We proceed by contradiction and assume that for every $k\in\mathbb{N}$ there is an element $\mathbf{u}_{k} \in \mathbf{U}_{ad}$ such that
\begin{equation}\label{eq:contradic_I}
\|\bar{\mathbf{u}} - \mathbf{u}_{k}\|_{\mathbf{L}^2(\Omega)} < k^{-1},
\qquad 
j(\mathbf{u}_{k}) < j(\bar{\mathbf{u}}) + (2k)^{-1}\|\bar{\mathbf{u}} -\mathbf{u}_{k}\|_{\mathbf{L}^2(\Omega)}^2.
\end{equation}
Define
$
\rho_{k}:=\|\mathbf{u}_{k}-\bar{\mathbf{u}}\|_{\mathbf{L}^2(\Omega)}
$
and
$
\mathbf{g}^{k}:=(\mathbf{u}_{k}-\bar{\mathbf{u}})/\rho_{k}
$
for $k \in \mathbb{N}$. Note that there exists a nonrelabeled subsequence $\{\mathbf{g}^{k}\}_{k \in \mathbb{N}} \subset \mathbf{L}^2(\Omega)$ such that $\mathbf{g}^{k}\rightharpoonup \mathbf{g}$ in $\mathbf{L}^2(\Omega)$ as $k \uparrow \infty$.

We now proceed in three steps.

\emph{Step 1:} $\mathbf{g}\in \mathbf{C}_{\bar{\mathbf{u}}}$. We note that the set of elements satisfying condition \eqref{eq:sign_cond} is weakly closed in $\mathbf{L}^2(\Omega)$. Consequently, $\mathbf{g}$ also satisfies condition \eqref{eq:sign_cond}. To verify the remaining condition in \eqref{def:critical_cone}, we apply the mean value theorem and \eqref{eq:contradic_I} to obtain
\begin{equation}\label{eq:contradic_II}
 j'(\tilde{\mathbf{u}}_{k})\mathbf{g}^{k}  = \rho_k^{-1}(j(\mathbf{u}_{k}) - j(\bar{\mathbf{u}}))<  (2k)^{-1}\rho_{k} \to 0,
\quad
k \uparrow \infty,
\end{equation}
where $\tilde{\mathbf{u}}_{k} = \bar{\mathbf{u}} + \theta_{k}(\mathbf{u}_{k} - \bar{\mathbf{u}})$ and $\theta_{k} \in (0,1)$. For each $k$, we define $(\tilde{\mathbf{y}}_{k},\tilde{p}_{k}) := \mathcal{S}\tilde{\mathbf{u}}_{k}$ and let $(\tilde{\mathbf{z}}_{k}, \tilde{r}_{k})$ be the unique solution to \eqref{eq:adj_eq}, where $\mathbf{y}$ is replaced by $\tilde{\mathbf{y}}_{k}$. Since $\tilde{ \mathbf{u}}_k \rightarrow \bar{\mathbf{u}}$ in $\mathbf{L}^2(\Omega)$ as $k \uparrow \infty$, the Lipschitz property of Theorem \ref{thm:Lipschitz_property} implies that $\tilde{\mathbf{y}}_{k} \to \bar{\mathbf{y}}$ in $\mathbf{H}_0^{1}(\Omega) \cap \mathbf{C}(\bar{\Omega})$ as $k \uparrow \infty$. This implies, in view of Theorem \ref{thm:well-posedness-adjoint}, that $\tilde{\mathbf{z}}_{k}\to \bar{\mathbf{z}}$ in $\mathbf{W}_0^{1,\mathsf{p}}(\Omega)$ as $k \uparrow\infty$. Here, $\mathsf{p} < d/(d-1)$ is arbitrarily close to $d/(d-1)$. In view of the continuous embedding $\mathbf{W}_0^{1,\mathsf{p}}(\Omega)\hookrightarrow \mathbf{L}^2(\Omega)$ \cite[Theorem 4.12]{MR2424078}, we deduce that $\tilde{\mathbf{z}}_{k}\to \bar{\mathbf{z}}$ in $\mathbf{L}^2(\Omega)$ as $k \uparrow\infty$. Consequently, $\tilde{\mathbf{z}}_k + \alpha \tilde{\mathbf{u}}_k =: \tilde{\mathbf{d}}_k \rightarrow \bar{\mathbf{d}} = \bar{\mathbf{z}} + \alpha \bar{\mathbf{u}}$ in $\mathbf{L}^2(\Omega)$ as $k \uparrow\infty$. This convergence result combined with $\mathbf{g}^{k}\rightharpoonup \mathbf{g}$ in $\mathbf{L}^2(\Omega)$ as $k \uparrow\infty$ and \eqref{eq:contradic_II} yield
\begin{equation*}
j'(\bar{\mathbf{u}})\mathbf{g}  = (\bar{\mathbf{d}},\mathbf{g})_{\mathbf{L}^2(\Omega)} = \lim_{k \uparrow \infty}
(\tilde{\mathbf{d}}_k,\mathbf{g}^k)_{\mathbf{L}^2(\Omega)} = \lim_{k \uparrow \infty}j'(\tilde{\mathbf{u}}_{k})\mathbf{g}^{k} \leq 0.
\end{equation*}
On the other hand, from \eqref{eq:var_ineq} we obtain 
$
(\bar{\mathbf{d}}, \mathbf{g}^k)_{\mathbf{L}^2(\Omega)} = \rho_k^{-1} (\bar{\mathbf{d}}, \mathbf{u}_k - \bar{\mathbf{u}})_{\mathbf{L}^2(\Omega)}$ $\geq 0.
$
This implies $(\bar{\mathbf{d}},\mathbf{g})_{\mathbf{L}^2(\Omega)} \geq 0$. Consequently, $(\bar{\mathbf{d}}, \mathbf{g})_{\mathbf{L}^2(\Omega)} = 0$. Now, since $\mathbf{g}$ satisfies the sign condition \eqref{eq:sign_cond}, the characterization of $\bar{\mathbf{d}}$ in \eqref{eq:derivative_j} allows us to conclude that
\begin{equation*}
(\bar{\mathbf{d}},\mathbf{g})_{\mathbf{L}^2(\Omega)}
=
\sum_{i=1}^d(\bar{\mathbf{d}}_{i},\mathbf{g}_{i})_{L^2(\Omega)} = \sum_{i=1}^d\int_{\Omega} |\bar{\mathbf{d}}_{i}\mathbf{g}_{i}|  = 0.
\end{equation*}
This proves that $\bar{\mathbf{d}}_{i}(x) \neq 0$ implies that $\mathbf{g}_{i}(x) = 0$ for a.e.~$x \in \Omega$ and $i\in\{1,\ldots,d\}$. Since we have already proved that $\mathbf{g}$ satisfies \eqref{eq:sign_cond}, we conclude that $\mathbf{g}\in \mathbf{C}_{\bar{\mathbf{u}}}$.

\emph{Step 2:} We prove that $\mathbf{g}=\mathbf{0}$. To do this, we use Taylor's theorem, the fact that $j'(\bar{\mathbf{u}})(\mathbf{u}_{k} - \bar{\mathbf{u}}) \geq 0$, for every $k\in\mathbb{N}$, and the inequality on the right-hand side of \eqref{eq:contradic_I}:
\begin{equation*}
2^{-1}\rho_{k}^2j''(\hat{\mathbf{u}}_{k})(\mathbf{g}^{k})^2 = j(\mathbf{u}_{k}) - j(\bar{\mathbf{u}}) - j'(\bar{\mathbf{u}})(\mathbf{u}_{k} - \bar{\mathbf{u}})  < (2k)^{-1}\rho_{k}^2,
\end{equation*}
where $\hat{\mathbf{u}}_{k} = \bar{\mathbf{u}} + \theta_{k}(\mathbf{u}_{k} - \bar{\mathbf{u}})$ and $\theta_{k} \in (0,1)$. Thus, $\lim_{k\uparrow\infty}j''(\hat{\mathbf{u}}_{k})(\mathbf{g}^{k})^2 \leq 0$. We now prove that $j''(\bar{\mathbf{u}})\mathbf{g}^2 \leq \liminf_{k\uparrow\infty}j''(\hat{\mathbf{u}}_{k})(\mathbf{g}^{k})^2$. Define, for $k\in\mathbb{N}$, $(\hat{\mathbf{y}}_k,\hat{p}_k):= \mathcal{S}(\hat{\mathbf{u}}_k)$ and $(\hat{\mathbf{z}}_k,\hat{r}_k)$ as the solution \eqref{eq:adj_eq}, where $\mathbf{y}$ is replaced by $ \hat{\mathbf{y}}_k$. Invoke \eqref{eq:charac_j2} and write
\[
j''(\hat{\mathbf{u}}_{k})(\mathbf{g}^{k})^2 = \alpha \| \mathbf{g}^{k} \|^2_{\mathbf{L}^2(\Omega)} -2b(\boldsymbol{\varphi}_k; \boldsymbol{\varphi}_k,\hat{\mathbf{z}}_k) + \sum_{t \in \mathcal{D}} \boldsymbol{\varphi}_k^2(t),
\]
where $(\boldsymbol{\varphi}_k,\zeta_k):= \mathcal{S}'(\hat{\mathbf{u}}_{k})\mathbf{g}^{k}$, i.e., $(\boldsymbol{\varphi}_k,\zeta_k)$ solves \eqref{eq:first_deriv_S*}, where $\mathbf{g}$ and $\mathbf{y}$ are replaced by $\mathbf{g}^k$ and $\hat{\mathbf{y}}_k$, respectively. Since $\hat{\mathbf{u}}_{k} \rightarrow \bar{\mathbf{u}}$ and $\mathbf{g}^k \rightharpoonup \mathbf{g}$ in $\mathbf{L}^2(\Omega)$, as $k \uparrow \infty$, we have 
\begin{equation}
\hat{\mathbf{y}}_k \rightarrow \bar{\mathbf{y}} \textrm{ in } \mathbf{W}_0^{1,\mathsf{q}}(\Omega) \cap \mathbf{C}(\bar \Omega),
\quad
\hat{\mathbf{z}}_k \rightarrow \bar{\mathbf{z}} \textrm{ in } \mathbf{W}_0^{1,\mathsf{p}}(\Omega), 
\quad
\boldsymbol{\varphi}_k \rightharpoonup  \boldsymbol{\varphi} \textrm{ in } \mathbf{W}_0^{1,\mathsf{q}}(\Omega)
\label{eq:convergence_properties}
\end{equation}
as $k \uparrow \infty$. From the fact that $\mathsf{q}>d$ and the compactness of the embedding $\mathbf{W}_0^{1,\mathsf{q}}(\Omega) \hookrightarrow \mathbf{C}(\bar \Omega)$ \cite[Theorem 6.3, Part III]{MR2424078} we conclude that $\| \boldsymbol{\varphi}_k - \boldsymbol{\varphi} \|_{\mathbf{L}^{\infty}(\Omega)} \rightarrow 0$ as $k\uparrow \infty$. This convergence property in combination with the ones in \eqref{eq:convergence_properties} yield
\[
-2b(\boldsymbol{\varphi}_k; \boldsymbol{\varphi}_k,\hat{\mathbf{z}}_k) + \sum_{t \in \mathcal{D}} \boldsymbol{\varphi}_k^2(t) 
\rightarrow 
-2b(\boldsymbol{\varphi}; \boldsymbol{\varphi},\bar{\mathbf{z}}) + \sum_{t \in \mathcal{D}} \boldsymbol{\varphi}^2(t),
\quad
k \uparrow \infty.
\]
The fact that the square of $\| \cdot \|_{\mathbf{L}^2(\Omega)}$ is weakly lower semicontinuous allows us to conclude. Thus, a collection of the derived estimates shows that $j''(\bar{\mathbf{u}})\mathbf{g}^2 \leq 0$. Therefore, the second order condition $j''(\bar{\mathbf{u}})\mathbf{g}^2 > 0$ for all $\mathbf{g}\in\mathbf{C}_{\bar{\mathbf{u}}}\setminus\{\mathbf{0}\}$ implies that $\mathbf{g}=\mathbf{0}$.

\emph{Step 3:} $\alpha \leq 0$. Finally, we notice that
\begin{equation}\label{eq:j''_=_alpha}
\alpha = \alpha \|\mathbf{g}^{k}\|^{2}_{\mathbf{L}^2(\Omega)}
=
j''(\hat{\mathbf{u}}_{k})(\mathbf{g}^{k})^2
+ 2b(\boldsymbol{\varphi}_{k};\boldsymbol{\varphi}_{k},\hat{\mathbf{z}}_{k}) - \sum_{t\in \mathcal{D}}\boldsymbol{\varphi}_{k}^2(t),
\end{equation} 
with $(\boldsymbol{\varphi}_k,\zeta_k)$ given as in the previous step. Arguments similar to those elaborated in Step 2, which are based on the convergence properties \eqref{eq:convergence_properties} reveal that $\alpha = \liminf_{k\uparrow\infty}j''(\hat{\mathbf{u}}_{k})(\mathbf{g}^{k})^2 \leq 0$, which is a contradiction. This concludes the proof.
\end{proof}

We conclude this section with an equivalent representation of the conditions stated in Theorem \ref{thm:suff_opt_cond}. To present it,  we introduce, for $\tau > 0$, the cone $\mathbf{C}_{\bar{\mathbf{u}}}^\tau:=\{\mathbf{g}\in \mathbf{L}^{2}(\Omega) \textnormal{ satisfying \eqref{eq:sign_cond} and } \mathbf{g}_{i}(x)=0 \textnormal{ if } |\bar{\mathbf{d}}_{i}(x)|>\tau\}$, where $i\in \{1,\ldots,d\}$.

\begin{theorem}[equivalent optimality conditions]\label{thm:equivalent_opt_cond}
Let $(\bar{\mathbf{y}}, \bar{p}, \bar{\mathbf{u}})$ be a local nonsingular solution to \eqref{eq:weak_cost}--\eqref{eq:weak_st_eq}. Then, the following statements are equivalent:
\begin{equation}
\label{eq:second_order_2_2}
j''(\bar{\mathbf{u}})\mathbf{g}^2 > 0 \, \forall \mathbf{g} \in \mathbf{C}_{\bar{\mathbf{u}}}\setminus \{\mathbf{0}\}
\Longleftrightarrow
\exists \mu, \tau >0: 
\,\,
j''(\bar{\mathbf{u}})\mathbf{g}^2 \geq \mu \|\mathbf{g}\|_{\mathbf{L}^2(\Omega)}^2 
\,\, 
\forall \mathbf{g} \in \mathbf{C}_{\bar{\mathbf{u}}}^\tau.
\end{equation}
\end{theorem}
\begin{proof}
The proof of \eqref{eq:second_order_2_2} follows from a combination of the arguments elaborated in the proofs of \cite[Theorem 3.10]{MR2338434} and Theorem \ref{thm:suff_opt_cond}. For brevity, we skip details.
\end{proof}


\section{Finite element approximation}\label{sec:fem}
In the following sections, we present discretization methods based on finite elements for the state equations, the optimal control problem \eqref{eq:weak_cost}--\eqref{eq:weak_st_eq}, and the adjoint equations. In order to derive convergence results and obtain error estimates, we assume in all the following sections that $\Omega$ is a \emph{convex polytope}. In particular, this ensures the regularity properties for solutions of the Navier--Stokes system described in Remark \ref{remark:further_reg_NS}.

Let us introduce the discrete setting in which we will be working \cite{MR2373954,MR0520174,MR2050138}. We denote by $\mathscr{T}_h = \{ T\}$ a conforming partition, or mesh, of $\bar{\Omega}$ into \emph{closed simplices} $T$ of size $h_T = \text{diam}(T)$. Here, $h:=\max \{ h_T: T \in \mathscr{T}_h \}$. By $\mathbb{T} = \{\mathscr{T}_h \}_{h>0}$, we denote a collection of conforming and quasi-uniform meshes $\mathscr{T}_h$.

We denote by $\mathbf{V}_h$ and $Q_h$ the finite element spaces approximating the velocity field and the pressure, respectively, constructed over the mesh $\mathscr{T}_h$. In particular, in our work we will consider the following well-known finite element pairs: 
\begin{itemize}
\item[(a)] \emph{The MINI element}. This pair is considered in \cite[\S 4.2.4]{MR2050138} and is defined by
\begin{align}\label{def:discrete_spaces_P1B}
\begin{split}
Q_h &=\{ q_h \in C(\bar{\Omega})\, : \, q_h|_T \in \mathbb{P}_{1}(T) \ \forall \: T \in\mathscr{T}_h \}  \cap L_0^2(\Omega), \\
\mathbf{V}_h &=\{ \mathbf{v}_h \in \mathbf{C}(\bar{\Omega})\, : \, \mathbf{v}_h|_T \in [ \mathbb{P}_1(T) \oplus \mathbb{B}(T) ]^d \ \forall \: T\in\mathscr{T}_h \}\cap \mathbf{H}_0^1(\Omega),
\end{split}
\end{align}
where $\mathbb{B}(T)$ denotes the space spanned by local bubble functions.
\item[(b)] \emph{The lowest order Taylor--Hood pair.}  This pair is defined by \cite[\S 4.2.5]{MR2050138}
\begin{align}\label{def:discrete_spaces_TH}
\begin{split}
Q_h &=\{ q_h \in C(\bar{\Omega})\, : \, q_h|_T \in \mathbb{P}_{1}(T) \ \forall \: T \in\mathscr{T}_h \}  \cap L_0^2(\Omega), \\
\mathbf{V}_h &=\{ \mathbf{v}_h \in \mathbf{C}(\bar{\Omega})\, : \, \mathbf{v}_h|_T \in [\mathbb{P}_2(T)]^d \ \forall \: T\in\mathscr{T}_h \}\cap \mathbf{H}_0^1(\Omega).
\end{split}
\end{align}
\end{itemize}

In the following analysis, the pair $(\mathbf{V}_{h},Q_{h})$ will represent indistinctly both the \emph{MINI element} and the \emph{lowest order Taylor--Hood element}. An important property satisfied by these pairs is the following compatibility condition: Let $\mathsf{r} \in (1,\infty)$ and let $\mathsf{s} \in (1,\infty)$ be the H\"older conjugate of $\mathsf{r}$. Then, there exists a positive constant $\tilde{\beta}$, independent of $h$, such that \cite[Lemmas 4.20 and 4.24]{MR2050138}
\begin{equation}\label{eq:infsup_div}
\inf_{q_h\in Q_{h}}\sup_{\mathbf{v}_h\in\mathbf{V}_{h}}\frac{(q_{h},\text{div }\mathbf{v}_{h})_{L^2(\Omega)}}{\|\nabla\mathbf{v}_{h}\|_{\mathbf{L}^{\mathsf{r}}(\Omega)}\|q_{h}\|_{\mathrm{L}^{\mathsf{s}}(\Omega)}}\geq\tilde{\beta}.
\end{equation}


\subsection{Discrete state equations}\label{sec:fem_state}

Let $\mathbf{u}\in \mathbf{L}^2(\Omega)$. We introduce the following finite element approximation of problem \eqref{eq:weak_st_eq}: Find $(\mathbf{y}_{h},p_{h})\in \mathbf{V}_{h}\times Q_{h}$ such that 
\begin{multline}
\label{eq:discrete_state_eq}
\nu(\nabla \mathbf{y}_{h},\nabla \mathbf{v}_{h})_{\mathbf{L}^2(\Omega)} + b(\mathbf{y}_{h};\mathbf{y}_{h},\mathbf{v}_{h}) - (p_{h},\text{div }\mathbf{v}_{h})_{L^2(\Omega)}
= (\mathbf{u},\mathbf{v}_{h})_{\mathbf{L}^2(\Omega)},
\\
(q_{h},\text{div }\mathbf{y}_{h})_{L^2(\Omega)} = 0
\qquad
\forall (\mathbf{v}_{h},q_{h})\in \mathbf{V}_{h}\times Q_{h}.
\end{multline}
For any $\mathbf{u}\in\mathbf{L}^2(\Omega)$, problem \eqref{eq:discrete_state_eq} does not necessarily have a unique solution $(\mathbf{y}_{h},p_{h})$. However, if a local nonsingular solution $(\bar{\mathbf{y}}, \bar{p},\bar{\mathbf{u}})$ to \eqref{eq:weak_cost}--\eqref{eq:weak_st_eq} is given and $\mathbf{u}$ is close enough to the locally optimal control $\bar{\mathbf{u}}$, then there is a unique discrete solution to \eqref{eq:discrete_state_eq} that is sufficiently close to $(\bar{\mathbf{y}}, \bar{p})$. 
This is summarized in the following result.

\begin{theorem}[existence of a unique discrete solution]\label{thm:existence_of_d_s}
Let $(\bar{\mathbf{y}}, \bar{p}, \bar{\mathbf{u}})$ be a local nonsingular solution to \eqref{eq:weak_cost}--\eqref{eq:weak_st_eq}. Then, there exist $\mathfrak{s},\mathfrak{r}> 0$, independent of $h$, and $h_{\star} > 0$ such that for all $h \in (0, h_{\star})$ and $\mathbf{u}\in B_{\mathfrak{s}}(\bar{\mathbf{u}})\subset\mathbf{L}^2(\Omega)$, problem  \eqref{eq:discrete_state_eq} admits a unique solution $(\mathbf{y}_{h},p_{h})\in B_{\mathfrak{r}}(\bar{\mathbf{y}})\times B_{\mathfrak{r}}(\bar{p})\subset \mathbf{H}_0^1(\Omega)\times L_0^2(\Omega)$.
\end{theorem}
\begin{proof}
See \cite[Theorem 4.8]{MR2338434}.
\end{proof}

We conclude this section with the following instrumental error estimate.

\begin{theorem}[error estimate]\label{thm:convergence_in_Linfty}
Let $(\bar{\mathbf{y}}, \bar{p}, \bar{\mathbf{u}})$ be a local nonsingular solution to \eqref{eq:weak_cost}--\eqref{eq:weak_st_eq}. Let $\mathfrak{s},\mathfrak{r}$, and $h_{\star}$ be as in the statement of Theorem \ref{thm:existence_of_d_s} and let $\mathbf{u}\in B_{\mathfrak{s}}(\bar{\mathbf{u}})$. Let $(\mathbf{y}_h,p_h)\in \mathbf{V}_h\times Q_{h}$ be the unique solution to  \eqref{eq:discrete_state_eq}. Let $(\mathbf{y},p)\in \mathbf{H}_0^1(\Omega)\times L_0^2(\Omega)$ be the unique solution to \eqref{eq:weak_st_eq} upon redefining $\mathfrak{s}$ if necessary. Then, we have
\begin{equation*}
\|\mathbf{y} - \mathbf{y}_{h}\|_{\mathbf{L}^{\infty}(\Omega)} \lesssim h^{2-\frac{d}{2}}\|\mathbf{u}\|_{\mathbf{L}^2(\Omega)} \quad \forall h < h_{\star}.
\end{equation*}
\end{theorem}
\begin{proof}
Let $\mathbf{I}_{h}: \mathbf{H}^2(\Omega) \cap \mathbf{H}_0^1(\Omega) \to \mathbf{V}_{h}$ be the Lagrange interpolation operator. An application of \cite[Remark 1.112 and Corollary 1.109]{MR2050138}, a standard inverse estimate, and \cite[estimate (4.4)]{MR2338434} yield
\begin{align*}
\|\mathbf{y} - \mathbf{y}_{h}\|_{\mathbf{L}^{\infty}(\Omega)} 
&\lesssim 
\|\mathbf{y} - \mathbf{I}_{h}\mathbf{y}\|_{\mathbf{L}^{\infty}(\Omega)} + \|\mathbf{I}_{h}\mathbf{y} - \mathbf{y}_{h}\|_{\mathbf{L}^{\infty}(\Omega)} \\
&\lesssim  
h^{2-\frac{d}{2}}|\mathbf{y}|_{\mathbf{H}^2(\Omega)} + h^{-\frac{d}{2}}\|\mathbf{I}_{h}\mathbf{y} - \mathbf{y}_{h}\|_{\mathbf{L}^{2}(\Omega)} \lesssim h^{2-\frac{d}{2}}\|\mathbf{u}\|_{\mathbf{L}^2(\Omega)},
\end{align*}
upon using the fact that $\mathbf{y}\in \mathbf{H}^2(\Omega)$ and that $\|\mathbf{y} \|_{\mathbf{H}^2(\Omega)} \lesssim \|\mathbf{u}\|_{\mathbf{L}^2(\Omega)}$; see Remark \ref{remark:further_reg_NS}.
\end{proof}


\subsection{Discrete optimal control problems}\label{sec:fem_ocp}

In what follows, we propose two strategies for discretizing problem \eqref{eq:weak_cost}--\eqref{eq:weak_st_eq}, namely a \emph{fully discrete approach} in which the control variable is \emph{discretized} with piecewise constant functions, and a \emph{semi-discrete approach} in which the control variable is \emph{not discretized} \cite{MR2122182}.


\subsubsection{A fully discrete scheme}
\label{sec:fully_disc_scheme}
We propose the following \emph{fully discrete} approximation of problem \eqref{eq:weak_cost}--\eqref{eq:weak_st_eq}: Find $\min J(\mathbf{y}_h,\mathbf{u}_h)$ subject to
\begin{multline}
\label{eq:discrete_state_eq_fully}
\nu(\nabla \mathbf{y}_{h},\nabla \mathbf{v}_{h})_{\mathbf{L}^2(\Omega)} + b(\mathbf{y}_{h};\mathbf{y}_{h},\mathbf{v}_{h}) - (p_{h},\text{div }\mathbf{v}_{h})_{L^2(\Omega)} \\
= (\mathbf{u}_h,\mathbf{v}_{h})_{\mathbf{L}^2(\Omega)},
\qquad (q_{h},\text{div }\mathbf{y}_{h})_{L^2(\Omega)} = 0,
\end{multline}
for all $(\mathbf{v}_{h},q_{h})\in \mathbf{V}_{h}\times Q_{h}$, and the discrete constraints $\mathbf{u}_{h}\in\mathbf{U}_{ad,h}$. Here, $\mathbf{U}_{ad,h}:= \mathbf{U}_h\cap \mathbf{U}_{ad}$, where $\mathbf{U}_h = \{\mathbf{u}_{h} \in \mathbf{L}^{\infty}(\Omega)  : \mathbf{u}_{h}|_{T}\in [\mathbb{P}_{0}(T)]^{d} ~ \forall T\in\mathscr{T}_{h}\}$. 

We now show the existence of solutions to the \emph{fully discrete} optimal control problem. Moreover, we prove that strict local nonsingular solutions of \eqref{eq:weak_cost}--\eqref{eq:weak_st_eq} can be approximated by local solutions of the \emph{fully discrete} optimal control problems.

\begin{theorem}[existence and convergence]\label{thm:exist_and_conv_sol}
Let $(\bar{\mathbf{y}}, \bar{p}, \bar{\mathbf{u}})$ be a local nonsingular solution to \eqref{eq:weak_cost}--\eqref{eq:weak_st_eq}. Then, there exists $h_{\nabla} > 0$ such that for all $ h \in (0, h_{\nabla})$ the fully discrete problem has a solution $(\bar{\mathbf{y}}_{h}, \bar{p}_{h}, \bar{\mathbf{u}}_{h})$. Moreover, if $(\bar{\mathbf{y}}, \bar{p}, \bar{\mathbf{u}})$ is a strict local minimum of \eqref{eq:weak_cost}--\eqref{eq:weak_st_eq}, then there exists a sequence $\{ (\bar{\mathbf{y}}_h, \bar{p}_h, \bar{\mathbf{u}}_h)\}_{h<h_{\nabla}}$ of local minima of the fully discrete optimal control problems such that
\begin{equation}\label{eq:convergence_ct_h_in_L2}
\lim_{h\rightarrow 0} J(\bar{\mathbf{y}}_{h},\bar{\mathbf{u}}_{h}) = J(\bar{\mathbf{y}},\bar{\mathbf{u}}), \qquad \lim_{h\rightarrow 0} \|  \bar{\mathbf{u}} - \bar{\mathbf{u}}_{h} \|_{\mathbf{L}^2(\Omega)} = 0.
\end{equation}
\end{theorem}
\begin{proof}
We proceed on the basis of two steps.

\emph{Step 1:} \emph{Existence of a solution:} Let $\Pi_{\mathbf{L}^2}: \mathbf{L}^2(\Omega) \rightarrow \mathbf{U}_{h}$ be the orthogonal projection operator and define $\hat{\mathbf{u}}_{h}=\Pi_{\mathbf{L}^2}\bar{\mathbf{u}}\in\mathbf{U}_{ad,h}$. Since $\bar{\mathbf{u}}\in \mathbf{W}^{1,\mathsf{p}}(\Omega)$, we have that $\|\bar{\mathbf{u}} - \hat{\mathbf{u}}_{h}\|_{\mathbf{L}^2(\Omega)} \to 0$  as $h \rightarrow 0$; recall that $\mathsf{p} < d/(d - 1)$ is arbitrarily close to $d/(d - 1)$. Then, there exists $h_{\Delta} >0$ such that $\hat{\mathbf{u}}_{h}\in B_{\mathfrak{s}}(\bar{\mathbf{u}})$ for all $h \in (0,h_{\Delta})$. Here, $B_{\mathfrak{s}}(\bar{\mathbf{u}})$ is as in the statement of Theorem \ref{thm:existence_of_d_s}. This theorem guarantees that if $h < h_{\nabla}:= \min\{ h_{\star}, h_{\Delta} \}$, then there exists a unique solution $(\hat{\mathbf{y}}_{h},\hat{p}_{h})$ to \eqref{eq:discrete_state_eq_fully}, where $\mathbf{u}_{h}$ is replaced by $\hat{\mathbf{u}}_{h}$. It follows that $(\hat{\mathbf{y}}_{h},\hat{\mathbf{u}}_{h})$ is a feasible pair. The existence of a discrete solution follows from the fact that we are minimizing a continuous and coercive function on a nonempty closed subset of a finite dimensional space.

\emph{Step 2:} \emph{Convergence}. Let us now assume that, in addition, $(\bar{\mathbf{y}}, \bar{p}, \bar{\mathbf{u}})$ is a \emph{strict} local minimum of \eqref{eq:weak_cost}--\eqref{eq:weak_st_eq} in $B_{\mathfrak{r}}(\bar{\mathbf{y}})\times B_{\mathfrak{r}}(\bar{p}) \times \mathbf{U}_{ad} \cap B_{\mathfrak{s}}(\bar{\mathbf{u}})$, where $\mathfrak{r}$ and $\mathfrak{s}$ are as in the statement of Theorem \ref{thm:existence_of_d_s}. 
Let us now introduce, for $h < h_{\nabla}$, the problem
\begin{equation}\label{eq:discrete_continuous_problem}
\min \{ j_{h}(\mathbf{u}_h): \, \mathbf{u}_h\in \mathbf{U}_{ad,h} \cap B_{\mathfrak{s}}(\bar{\mathbf{u}}) \}.
\end{equation}
Here, $j_{h}(\mathbf{u}_h) := J(\mathbf{y}_{h}(\mathbf{u}_h),\mathbf{u}_h)$, where $(\mathbf{y}_{h}(\mathbf{u}_h),p_{h}(\mathbf{u}_h))\in\mathbf{V}_{h}\times Q_{h}$ corresponds to the unique solution to \eqref{eq:discrete_state_eq}, where $\mathbf{u}$ is replaced by $\mathbf{u}_h$. Problem \eqref{eq:discrete_continuous_problem} admits at least one optimal solution $\bar{\mathbf{u}}_h$. We then have a sequence of optimal solutions $\{ \bar{\mathbf{u}}_h \}_{0<h<h_{\nabla}}$ such that it is uniformly bounded in $\mathbf{L}^2(\Omega)$. We can thus extract a nonrelabeled subsequence such that $ \bar{\mathbf{u}}_h \rightharpoonup \hat{\mathbf{u}}$ in $\mathbf{L}^2(\Omega)$ as $h \rightarrow 0$. Note that $\hat{\mathbf{u}} \in B_{\mathfrak{s}}(\bar{\mathbf{u}})$. 
We now prove that $\hat{\mathbf{u}}=\bar{\mathbf{u}}$. Since $\bar{\mathbf{u}}_{h} \rightharpoonup \hat{\mathbf{u}}$ in $\mathbf{L}^2(\Omega)$ as $h \rightarrow 0$, the convergence result from \cite[Lemma 4.10]{MR2338434} guarantees that $\bar{\mathbf{y}}_{h} \to \mathbf{y}(\hat{\mathbf{u}})$ in $\mathbf{C}(\bar{\Omega})$ as $h \rightarrow 0$. We can thus obtain that
\begin{equation*}
j(\hat{\mathbf{u}}) \leq \liminf_{h\rightarrow 0}j_{h}(\bar{\mathbf{u}}_{h}) \leq \limsup_{h\rightarrow 0}j_{h}(\bar{\mathbf{u}}_{h}) \leq \limsup_{h\rightarrow 0}j_{h}\left(\Pi_{\mathbf{L}^2}\bar{\mathbf{u}}\right) = j(\bar{\mathbf{u}}),
\end{equation*}
upon using that $\| \bar{\mathbf{u}} - \Pi_{\mathbf{L}^2}\bar{\mathbf{u}} \|_{\mathbf{L}^2(\Omega)} \to 0$ as $h \to 0$. Since $\hat{\mathbf{u}} \in \mathbf{U}_{ad} \cap B_{\mathfrak{s}}(\bar{\mathbf{u}})$, $j(\hat{\mathbf{u}}) \leq j(\bar{\mathbf{u}})$, and $\bar{\mathbf{u}}$ is a strict local minimum in $\mathbf{U}_{ad} \cap B_{\mathfrak{s}}(\bar{\mathbf{u}})$, we must have that $\hat{\mathbf{u}}=\bar{\mathbf{u}}$ and that $\bar{\mathbf{u}}_{h} \rightharpoonup \hat{\mathbf{u}}$ in $\mathbf{L}^2(\Omega)$ as $h \rightarrow 0$ for the entire sequence. Consequently, $j_{h}(\bar{\mathbf{u}}_{h}) \rightarrow j(\bar{\mathbf{u}})$ as $h \rightarrow 0$. From this, and from the fact that $\bar{\mathbf{y}}_{h} \to \mathbf{y}(\bar{\mathbf{u}})$ in $\mathbf{C}(\bar{\Omega})$, it can be deduced that the square of $\|\bar{\mathbf{u}}_{h}\|_{\mathbf{L}^2(\Omega)}$ converges to the square of $\|\bar{\mathbf{u}}\|_{\mathbf{L}^2(\Omega)}$ as $h \to 0$. As a result, $\bar{\mathbf{u}}_{h} \to \bar{\mathbf{u}}$ in $\mathbf{L}^2(\Omega)$ as $h \rightarrow 0$. Thus, the constraint $\bar{\mathbf{u}}_{h} \in  B_{\mathfrak{s}}(\bar{\mathbf{u}})$ is not active for $h$ that is small enough, and $(\bar{\mathbf{y}}_{h},\bar{\mathbf{u}}_{h})$ is a local solution to the fully discrete problem. 
\end{proof}


\subsubsection{A semidiscrete scheme}
\label{sec:semi_disc_scheme}
We propose the following \emph{semidiscrete} approximation of problem \eqref{eq:weak_cost}--\eqref{eq:weak_st_eq}: Find $\min J(\mathbf{y}_h,\boldsymbol{u})$ subject to 
\begin{multline}
\label{eq:discrete_state_eq_semi}
\nu(\nabla \mathbf{y}_{h},\nabla \mathbf{v}_{h})_{\mathbf{L}^2(\Omega)} + b(\mathbf{y}_{h};\mathbf{y}_{h},\mathbf{v}_{h}) - (p_{h},\text{div }\mathbf{v}_{h})_{L^2(\Omega)} \\
= (\boldsymbol{u},\mathbf{v}_{h})_{\mathbf{L}^2(\Omega)},
\qquad (q_{h},\text{div }\mathbf{y}_{h})_{L^2(\Omega)} = 0,
\end{multline}
for all $(\mathbf{v}_{h},q_{h})\in \mathbf{V}_{h}\times Q_{h}$, and $\boldsymbol u\in\mathbf{U}_{ad}$. 

The \emph{semidiscrete} scheme discretizes only the state spaces, i.e., the control space is not discretized. The scheme induces a discretization of optimal controls by projecting the optimal discrete adjoint state into the admissible control set $\mathbf{U}_{ad}$. Since an optimal control $\bar{\boldsymbol u}$ implicitly depends on $h$, we will use the notation $\bar{\boldsymbol u}_{h}$ in the following.

We now provide a version of Theorem \ref{thm:exist_and_conv_sol} for the \emph{semidiscrete} scheme.

\begin{theorem}[existence and convergence]\label{thm:exist_and_conv_sol_var}
Let $(\bar{\mathbf{y}},\bar{p},\bar{\mathbf{u}})$ be a local nonsingular solution to \eqref{eq:weak_cost}--\eqref{eq:weak_st_eq}. Then, there exists $h_{\Box} > 0$ such that for all $h \in (0,h_{\Box})$ the semidiscrete problem has a solution $(\bar{\mathbf{y}}_{h}, \bar{p}_{h}, \bar{\boldsymbol{u}}_h)$. Moreover, if $(\bar{\mathbf{y}}, \bar{p}, \bar{\mathbf{u}})$ is a strict local minimum of \eqref{eq:weak_cost}--\eqref{eq:weak_st_eq}, then there exists a sequence $\{ (\bar{\mathbf{y}}_{h}, \bar{p}_{h}, \bar{\boldsymbol{u}}_h) \}_{h<h_{\Box}}$ of local minimina of the semidiscrete optimal control problems such that the convergence properties stated in \eqref{eq:convergence_ct_h_in_L2} hold with $\bar{\mathbf{u}}_h$ replaced by $\bar{\boldsymbol{u}}_h$.
\end{theorem}
\begin{proof}
The existence of a feasible pair is direct. The existence of a solution to the semidiscrete problem follows from the arguments given in the proof of \cite[Theorem 4.11]{MR2338434}. The convergence properties \eqref{eq:convergence_ct_h_in_L2} follow the arguments from Theorem \ref{thm:exist_and_conv_sol}. For brevity, we skip the details.
\end{proof}


\subsection{Discrete adjoint equations}

In this section, assuming a suitable discrete inf-sup condition (estimate \eqref{eq:discrete_inf_sup}), we provide a well-posedness result for a discretization of the adjoint equations \eqref{eq:adj_eq} with finite elements and derive error estimates.

Let $\mathfrak{s}$ and $h_{\star}$ be as given in the statement of Theorem \ref{thm:existence_of_d_s}. Let $\mathbf{u}\in B_{\mathfrak{s}}(\bar{\mathbf{u}})$. Let $\mathsf{p} < d/(d-1)$ be arbitrarily close to $d/(d-1)$. We consider the following finite element discretization of the adjoint equations \eqref{eq:adj_eq}: Find $(\mathbf{z}_{h},r_{h}) \in \mathbf{V}_{h}\times Q_h$ such that
\begin{multline}\label{eq:adj_eq_discrete}
\nu(\nabla \mathbf{w}_{h}, \nabla \mathbf{z}_{h})_{\mathbf{L}^2(\Omega)} +
b(\mathbf{y}_{h};\mathbf{w}_{h},\mathbf{z}_{h}) + b(\mathbf{w}_{h};\mathbf{y}_{h},\mathbf{z}_{h}) - (r_{h},\text{div } \mathbf{w}_{h})_{L^2(\Omega)} \\ =  \sum_{t\in\mathcal{D}}\langle (\mathbf{y}_{h}(t)-\mathbf{y}_{t})\delta_{t},\mathbf{w}_{h}\rangle_{\mathbf{W}^{-1,\mathsf{p} }(\Omega),\mathbf{W}^{1,\mathsf{q} }_0(\Omega)}, \qquad
(s_{h},\text{div } \mathbf{z}_{h})_{L^2(\Omega)} = 0
\end{multline}
for all $(\mathbf{w}_{h},s_{h})\in \mathbf{V}_{h} \times Q_{h}$. Here, $\mathsf{p}^{-1} + \mathsf{q}^{-1} = 1$, $h < h_{\star}$, and $(\mathbf{y}_{h},p_{h}) \in \mathbf{V}_{h} \times Q_{h}$ corresponds to the unique solution to \eqref{eq:discrete_state_eq}.

Let $\mathsf{r} \in (1,\infty)$ and let $\mathfrak{L} :\mathbf{W}^{-1,\mathsf{r}}(\Omega) \to \mathbf{W}_0^{1,\mathsf{r}}(\Omega)\times L_0^{\mathsf{r}}(\Omega)$ be the linear and bounded operator defined by $\mathfrak{L}\mathbf{g} := (\mathfrak{Z}, \mathfrak{p})$, where $(\mathfrak{Z}, \mathfrak{p})$ corresponds to the solution of the following Stokes system: Find $(\mathfrak{Z}, \mathfrak{p})$ such that
\begin{equation}
-\nu \Delta \mathfrak{Z} + \nabla \mathfrak{p} = \mathbf{g} \text{ in } \Omega, \quad
\text{div }\mathfrak{Z} = 0 \text{ in } \Omega, 
\quad \mathfrak{Z}=\mathbf{0} \text{ on }\partial\Omega.
\label{eq:Stokes_Z}
\end{equation}
Since $\Omega$ is convex, the first item in \cite[Section 5.5]{MR2321139} shows that $\mathfrak{L}$ is an isomorphism when $\mathsf{r}>2$. The well-posedness of \eqref{eq:Stokes_Z} when $\mathsf{r}<2$ follows from the equivalent characterization of well-posedness via inf-sup conditions. The case $\mathsf{r} = 2$ is trivial.

Let us now introduce, for $h>0$, the discrete operator $\mathfrak{L}_h: \mathbf{W}^{-1,\mathsf{r}}(\Omega)\to \mathbf{V}_{h}\times Q_{h}$ defined by $\mathfrak{L}_h\mathbf{g} := (\mathfrak{Z}_{h}, \mathfrak{p}_{h})$, where $(\mathfrak{Z}_{h},\mathfrak{p}_{h})\in \mathbf{V}_{h} \times Q_{h}$ solves the discrete problem
\begin{equation*}
\nu(\nabla \mathfrak{Z}_{h}, \nabla \mathbf{v}_{h})_{\mathbf{L}^2(\Omega)} - (\mathfrak{p}_{h}, \text{div }\mathbf{v}_{h})_{L^2(\Omega)} = \langle \mathbf{g} ,\mathbf{v}_{h}\rangle, \qquad (q_{h}, \text{div }\mathfrak{Z}_{h})_{L^2(\Omega)} = 0,
\end{equation*}
for all $(\mathbf{v}_{h},q_{h})\in \mathbf{V}_{h}\times Q_{h}$. Define $\mathbf{V}_{h}^{0}:=\{\mathbf{v}_{h}\in\mathbf{V}_{h} : (q_{h},\text{div }\mathbf{v}_{h})_{L^2(\Omega)} = 0 \, \forall q_{h}\in Q_{h}\}$. In what follows we assume that there is a constant $\bar{\beta}>0$, independent of $h$, so that
\begin{equation}\label{eq:discrete_inf_sup}
\inf_{\mathbf{w}_{h}\in\mathbf{V}_{h}^{0}}\sup_{\mathbf{v}_{h}\in\mathbf{V}_{h}^{0}}\frac{\nu(\nabla \mathbf{w}_{h}, \nabla \mathbf{v}_{h})_{\mathbf{L}^2(\Omega)}}{\|\nabla \mathbf{w}_{h}\|_{\mathbf{L}^\mathsf{r}(\Omega)}\|\nabla \mathbf{v}_{h}\|_{\mathbf{L}^\mathsf{s}(\Omega)}} \geq \bar{\beta},
\end{equation}
where $\mathsf{r}^{-1} + \mathsf{s}^{-1} = 1$. From this condition in conjunction with \eqref{eq:infsup_div} we can deduce that $\mathfrak{L}_h$ is an isomorphism \cite[Corollary 2.2]{MR972452} (see also \cite[Exercise 2.14]{MR2050138}) and
\begin{equation}\label{eq:estimate_Lh}
\|\mathfrak{L}_{h}\|_{\mathcal{L}(\mathbf{W}^{-1,\mathsf{r}}(\Omega),\mathbf{W}_0^{1,\mathsf{r}}(\Omega)\times L_0^{\mathsf{r}}(\Omega))}\leq C, 
\end{equation}
where $\mathcal{L}(\mathbf{W}^{-1,\mathsf{r}}(\Omega),\mathbf{W}_0^{1,\mathsf{r}}(\Omega)\times L_0^{\mathsf{r}}(\Omega))$ corresponds to the space of linear and continuous operators from $\mathbf{W}^{-1,\mathsf{r}}(\Omega)$ to $\mathbf{W}_0^{1,\mathsf{r}}(\Omega)\times L_0^{\mathsf{r}}(\Omega)$.

Inspired by \cite[page 957]{MR2338434} we now introduce, for $\mathbf{y} \in \mathbf{H}_0^1(\Omega)$, the operator
\begin{equation}\label{def:bilinear_B_p_q}
\mathfrak{B}^{\star}(\mathbf{y}): \mathbf{W}_0^{1,\mathsf{p}}(\Omega)\to \mathbf{W}^{-1,\mathsf{p}}(\Omega),
\qquad
\langle \mathfrak{B}^{\star}(\mathbf{y}) \mathbf{v}, \mathbf{w} \rangle := b(\mathbf{y};\mathbf{w},\mathbf{v}) + b(\mathbf{w};\mathbf{y},\mathbf{v}),
\end{equation}
for all $(\mathbf{v},\mathbf{w}) \in \mathbf{W}_0^{1,\mathsf{p}}(\Omega) \times \mathbf{W}_0^{1,\mathsf{q}}(\Omega)$. Note that $\mathfrak{B}^{\star}(\mathbf{y}) \in \mathcal{L}(\mathbf{W}_0^{1,\mathsf{p}}(\Omega), \mathbf{W}^{-1,\mathsf{p}}(\Omega))$.

\begin{lemma}[auxiliary result]\label{lemma:aux_result_I}
Let $\mathsf{p} < d/(d-1)$ be arbitrarily close to $d/(d-1)$ and let $(\mathbf{y},p) \in \mathbf{H}_0^1(\Omega) \times L_0^2(\Omega)$ be a solution to the Navier--Stokes equations \eqref{eq:weak_st_eq} such that $\mathbf{y}$ is regular. Then, for every $\delta > 0$ there exist $h_{\delta}>0$ and $\mathfrak{s}_{\delta}>0$ such that
\begin{equation}\label{eq:estimate_L-Lh}
\mathfrak{I}_h:= \|\mathfrak{L}[\mathfrak{B}^{\star}(\mathbf{y})] - \mathfrak{L}_{h}[\mathfrak{B}^{\star}(\tilde{\mathbf{y}})]\|_{\mathcal{L}(\mathbf{W}_0^{1,\mathsf{p}}(\Omega),\mathbf{W}_0^{1,\mathsf{p}}(\Omega)\times L_0^{\mathsf{p}}(\Omega))} < \delta 
\end{equation}
for all $h \in (0,h_{\delta})$ and for all $\tilde{\mathbf{y}}\in \mathbf{H}_0^1(\Omega)$ satisfying $\|\nabla (\mathbf{y} - \tilde{\mathbf{y}})\|_{\mathbf{L}^2(\Omega)} \leq \mathfrak{s}_{\delta}$.
\end{lemma}
\begin{proof}
We proceed as in \cite[Lemma 4.6]{MR2338434} and control $\mathfrak{I}_h$ as follows:
\begin{multline*}
\mathfrak{I}_h
\leq  
\sup_{\|\nabla\mathbf{v}\|_{\mathbf{L}^{\mathsf{p}}(\Omega)}\leq 1}
\|(\mathfrak{L} - \mathfrak{L}_{h})[\mathfrak{B}^{\star}(\mathbf{y})\mathbf{v}]\|_{\mathbf{W}_0^{1,\mathsf{p}}(\Omega)\times L_0^{\mathsf{p}}(\Omega)} 
\\ 
+ \sup_{\|\nabla\mathbf{v}\|_{\mathbf{L}^{\mathsf{p}}(\Omega)}\leq 1}
\|\mathfrak{L}_{h}[(\mathfrak{B}^{\star}(\mathbf{y}) - \mathfrak{B}^{\star}(\tilde{\mathbf{y}}))\mathbf{v}]\|_{\mathbf{W}_0^{1,\mathsf{p}}(\Omega)\times L_0^{\mathsf{p}}(\Omega)} =: \mathbf{I}_h + \mathbf{II}_h.
\end{multline*}
To control $\mathbf{I}_h$, we must essentially control the error that occurs in the finite element approximation of the Stokes system \eqref{eq:Stokes_Z}. Using estimates \eqref{eq:infsup_div} and \eqref{eq:discrete_inf_sup} and applying \cite[Theorems 2.2 and 2.3 and Proposition 2.1]{MR972452}, we obtain the error bound
\[
\mathbf{I}_h
\lesssim 
\inf_{\mathbf{w}_{h}\in \mathbf{V}_{h}}
\|\nabla(\mathfrak{Z}_{\mathbf{v}} - \mathbf{w}_{h})\|_{\mathbf{L}^{\mathsf{p}}(\Omega)}
+ 
\inf_{q_{h}\in Q_{h}}\|\mathfrak{p}_{\mathbf{v}} - q_{h}\|_{L^\mathsf{p}(\Omega)},
\]
where $(\mathfrak{Z}_{\mathbf{v}}, \mathfrak{p}_{\mathbf{v}}) := \mathfrak{L}[\mathfrak{B}^{\star}(\mathbf{y})\mathbf{v})]$. Since $(\mathfrak{Z}_{\mathbf{v}}, \mathfrak{p}_{\mathbf{v}}) \in \mathbf{W}_0^{1,\mathsf{p}}(\Omega)\times L_0^{\mathsf{p}}(\Omega)$, a density argument like the one elaborated in the proof of \cite[estimate (1.99)]{MR2050138} shows that $\mathbf{I}_h < \delta/2$ for all $h>0$ sufficiently small. To control the term $\mathbf{II}_h$, we first use the estimate \eqref{eq:estimate_Lh} and then similar arguments as in the proof of Theorem \ref{thm:Lipschitz_property}:
\begin{equation*}
\mathbf{II}_h \lesssim  \sup_{\|\nabla\mathbf{v}\|_{\mathbf{L}^\mathsf{p}(\Omega)}\leq 1}\|(\mathfrak{B}^{\star}(\mathbf{y}) - \mathfrak{B}^{\star}(\tilde{\mathbf{y}}))\mathbf{v})\|_{\mathbf{W}^{-1,\mathsf{p}}(\Omega)} \lesssim \|\nabla (\mathbf{y} - \tilde{\mathbf{y}})\|_{\mathbf{L}^2(\Omega)}.
\end{equation*}
Consider $\mathfrak{s}_{\delta}$ sufficiently small so that $\mathbf{II}_{h} < \delta/2$. This completes the proof.
\end{proof}

Let us now introduce, for $\mathbf{y} \in \mathbf{H}_0^1(\Omega)$, the mapping
\[
 \mathfrak{N}_{\mathbf{y}} : \mathbf{W}_{0}^{1,\mathsf{p}}(\Omega)\times L_0^{\mathsf{p}}(\Omega) \to \mathbf{W}_{0}^{1,\mathsf{p}}(\Omega)\times L_0^{\mathsf{p}}(\Omega),
 \quad
 \mathfrak{N}_{\mathbf{y}}(\mathbf{z},r) = (\mathbf{z},r) + \mathfrak{L}[\mathfrak{B}^{\star}(\mathbf{y})\mathbf{z}].
\]

\begin{lemma}[auxiliary result]\label{lemma:N_y_automorphism}
Let $\mathsf{p} < d/(d-1)$ be arbitrarily close to $d/(d-1)$ and let $(\mathbf{y},p) \in \mathbf{H}_0^1(\Omega) \times L_0^2(\Omega)$ be a solution to the Navier--Stokes equations \eqref{eq:weak_st_eq}. If $\mathbf{y}$ is regular, then the mapping $\mathfrak{N}_{\mathbf{y}}$ is an automorphism.
\end{lemma}
\begin{proof}
To prove the result, we proceed in two steps.

\emph{Step 1. $\mathfrak{N}_{\mathbf{y}}$ is surjective.} Let $(\tilde{\mathbf{z}},\tilde{r})\in \mathbf{W}_{0}^{1,\mathsf{p}}(\Omega)\times L_0^{\mathsf{p}}(\Omega)$. We prove the existence of $(\mathbf{z},r)\in \mathbf{W}_{0}^{1,\mathsf{p}}(\Omega)\times L_0^{\mathsf{p}}(\Omega)$ such that $\mathfrak{N}_{\mathbf{y}}(\mathbf{z},r) = (\tilde{\mathbf{z}},\tilde{r})$. For this purpose, we let $(\mathbf{z}_0,r_{0})$ be the solution to: Find $(\mathbf{z}_0,r_{0})\in \mathbf{W}_{0}^{1,\mathsf{p}}(\Omega)\times L_0^{\mathsf{p}}(\Omega)$ such that
\begin{multline}\label{eq:eq_z0_r0}
\nu(\nabla \mathbf{w}, \nabla \mathbf{z}_0)_{\mathbf{L}^2(\Omega)} +
b(\mathbf{y};\mathbf{w},\mathbf{z}_0) + b(\mathbf{w};\mathbf{y},\mathbf{z}_0) - (r_0,\text{div } \mathbf{w})_{L^2(\Omega)} \\ =  -\langle \mathfrak{B}^{\star}(\mathbf{y}) \tilde{\mathbf{z}}, \mathbf{w} \rangle,
\qquad
(s,\text{div } \mathbf{z}_0)_{L^2(\Omega)} = 0,
\end{multline}
for all $(\mathbf{w},s)\in \mathbf{W}_{0}^{1,\mathsf{q} }(\Omega)\times L_0^{\mathsf{q}}(\Omega)$. Since $\mathfrak{B}^{\star}(\mathbf{y}) \tilde{\mathbf{z}} \in \mathbf{W}^{-1,\mathsf{p}}(\Omega)$, a direct application of Theorem \ref{thm:well-posedness-adjoint} shows that problem \eqref{eq:eq_z0_r0} is well-posed.
We now rewrite problem \eqref{eq:eq_z0_r0} as a Stokes problem with $-\mathfrak{B}^{\star}(\mathbf{y})( \tilde{\mathbf{z}} + \mathbf{z}_0)$ as a forcing term in the momentum equation and invoke the definition of the mapping $\mathfrak{L}$ to deduce that $(\mathbf{z}_0,r_0) = - \mathfrak{L}[\mathfrak{B}^{\star}(\mathbf{y})
(\tilde{\mathbf{z}} + \mathbf{z}_0)]$. Consequently, $(\mathbf{z},r)=(\mathbf{z}_0 + \tilde{\mathbf{z}},r_0 + \tilde{r})$. In fact, $\mathfrak{N}_{\mathbf{y}}(\mathbf{z},r) = (\mathbf{z}_0 + \tilde{\mathbf{z}},r_0 + \tilde{r}) + \mathfrak{L}[\mathfrak{B}^{\star}(\mathbf{y})(\mathbf{z}_0 + \tilde{\mathbf{z}})] = (\mathbf{z}_0 + \tilde{\mathbf{z}},r_0 + \tilde{r}) - (\mathbf{z}_0,r_0) = (\tilde{\mathbf{z}},\tilde{r})$.

\emph{Step 2. $\mathfrak{N}_{\mathbf{y}}$ is injective.} Let $(\mathbf{z},r)\in \mathbf{W}_{0}^{1,\mathsf{p}}(\Omega)\times L_0^{\mathsf{p}}(\Omega)$ be such that $\mathfrak{N}_{\mathbf{y}}(\mathbf{z},r) = (\mathbf{0},0)$. It follows from the definition of $\mathfrak{N}_{\mathbf{y}}$ that $(\mathbf{z},r)$ solves problem \eqref{eq:aux_eq_adj} with $\mathbf{h} = \mathbf{0}$. An application of Theorem \ref{thm:well-posedness-adjoint} shows that $(\mathbf{z},r) = (\mathbf{0},0)$.
\end{proof}

We are now in a position to present the most important result of this section.

\begin{theorem}[existence of a unique discrete solution]\label{thm:existence_of_d_s_adj}
Let $(\bar{\mathbf{y}}, \bar{p}, \bar{\mathbf{u}})$ be a local nonsingular solution to \eqref{eq:weak_cost}--\eqref{eq:weak_st_eq}. Let $\mathsf{p} < d/(d-1)$ be arbitrarily close to $d/(d-1)$. Then, there exist $\mathtt{s}>0$ and $h_{\dagger}>0$ such that for all $\mathbf{u}\in B_{\mathtt{s}}(\bar{\mathbf{u}})$ and $h < h_{\dagger}$, the problem \eqref{eq:adj_eq_discrete} admits a unique discrete solution $(\mathbf{z}_{h},r_{h}) \in \mathbf{V}_{h}\times Q_h$.
\end{theorem}
\begin{proof}
To prove the desired result, we proceed in three steps.

\textit{Step 1. $\mathfrak{N}_{\mathbf{y}}$ is an automorphism}. Let $\mathbf{u}\in \mathcal{O}(\bar{\mathbf{u}})$,  where $\mathcal{O}(\bar{\mathbf{u}})$ is given as in Theorem \ref{thm:properties_C_to_S}. The pair $(\mathbf{y},p) = \mathcal{S}(\mathbf{u}) \in\mathbf{H}_0^1(\Omega)\times L_0^2(\Omega)$ solves \eqref{eq:weak_st_eq} uniquely in $\mathcal{O}(\bar{\mathbf{y}}) \times \mathcal{O}(\bar{p})$. Moreover, $\mathcal{S}'(\mathbf{u})$ is an isomorphism and hence $\mathbf{y}$ is a regular solution. Therefore, a direct application of Lemma \ref{lemma:N_y_automorphism} yields that the operator $\mathfrak{N}_{\mathbf{y}}$ is an automorphism.

\textit{Step 2. $\mathfrak{N}_{\mathbf{y}_{h}}^{h}$ is an automorphism}. We introduce, for $\mathbf{v} \in \mathbf{H}_0^1(\Omega)$, the mapping
\[
 \mathfrak{N}_{\mathbf{v}}^{h}:\mathbf{W}_0^{1,\mathsf{p}}(\Omega)\times L_0^{\mathsf{p}}(\Omega)\to\mathbf{W}_0^{1,\mathsf{p}}(\Omega)\times L_0^{\mathsf{p}}(\Omega),
 \quad
 \mathfrak{N}_{\mathbf{v}}^{h}(\mathbf{z},r) = (\mathbf{z}, r) + \mathfrak{L}_{h}[\mathfrak{B}^{\star}(\mathbf{v})\mathbf{z}].
\]
Let $\mathfrak{s}$ and $h_{\star}$ be as in Theorem \ref{thm:existence_of_d_s} and let $\mathbf{u}\in B_{\mathfrak{s}}(\bar{\mathbf{u}}) \subset \mathcal{O}(\bar{\mathbf{u}})$; this may further restrict $\mathfrak{s}$. Theorem \ref{thm:existence_of_d_s} guarantees the existence of a unique solution $(\mathbf{y}_h,p_h)$ to problem \eqref{eq:discrete_state_eq} for $h < h_{\star}$. We now prove the inequality
\begin{equation}
 \|\mathfrak{N}_{\mathbf{y}} - \mathfrak{N}_{\mathbf{y}_{h}}^{h}\|_{\mathcal{L}(\mathbf{W}_0^{1,\mathsf{p}}(\Omega)\times L_0^{\mathsf{p}}(\Omega))} < 2^{-1}\|\mathfrak{N}_{\mathbf{y}}^{-1}\|^{-1}_{\mathcal{L}(\mathbf{W}_0^{1,\mathsf{p}}(\Omega)\times L_0^{\mathsf{p}}(\Omega))},
 \label{eq:difference_N_y}
\end{equation}
and apply \cite[Lemma 4.5]{MR2338434} to deduce that $\mathfrak{N}_{\mathbf{y}_{h}}^{h}$ is an isomorphism. 
By definition,
\begin{equation*}
\|
\mathfrak{N}_{\mathbf{y}} - \mathfrak{N}_{\mathbf{y}_{h}}^{h}
\|_{\mathcal{L}(\mathbf{W}_0^{1,\mathsf{p}}(\Omega)\times L_0^{\mathsf{p}}(\Omega))} 
= 
\|\mathfrak{L}[\mathfrak{B}^{\star}(\mathbf{y})] - \mathfrak{L}_{h}[\mathfrak{B}^{\star}(\mathbf{y}_{h})]\|_{\mathcal{L}(\mathbf{W}_0^{1,\mathsf{p}}(\Omega),\mathbf{W}_0^{1,\mathsf{p}}(\Omega)\times L_0^{\mathsf{p}}(\Omega))}.
\end{equation*}
Let $\delta = (2\|\mathfrak{N}_{\mathbf{y}}^{-1}\|)^{-1}$, where $\| \cdot \|$ is the norm in $\mathcal{L}(\mathbf{W}_0^{1,\mathsf{p}}(\Omega)\times L_0^{\mathsf{p}}(\Omega))$. From Theorem \ref{thm:existence_of_d_s} and \cite[estimate (4.5)]{MR2338434} we obtain the existence of $\mathtt{s} \leq \mathfrak{s}$ such that for all $\mathbf{u} \in B_{\mathtt{s}}(\bar{\mathbf{u}})$ it holds that $\|\nabla(\mathbf{y} - \mathbf{y}_{h})\|_{\mathbf{L}^2(\Omega)} \leq \mathfrak{s}_{\delta}$, where $\mathfrak{s}_{\delta}$ is as in the statement of Lemma \ref{lemma:aux_result_I}. This Lemma guarantees the existence of $h_{\dagger} < \min \{ h_{\delta}, h_{\star} \}$ so that the inequality \eqref{eq:difference_N_y} holds. This proves that $\mathfrak{N}_{\mathbf{y}_{h}}^{h}$ is an automorphism \cite[Lemma 4.5]{MR2338434}.

\textit{Step 3. Well-posedness of \eqref{eq:adj_eq_discrete}}. Let $\mathbf{u} \in B_{\mathtt{s}}(\bar{\mathbf{u}})$ and $h < h_{\dagger}$. We prove that $(\mathbf{z}_{h},r_{h})\in\mathbf{V}_{h}\times Q_{h}$ solves \eqref{eq:adj_eq_discrete} if and only if it satisfies $\mathfrak{N}_{\mathbf{y}_{h}}^{h}(\mathbf{z}_{h},r_{h}) = (\mathfrak{Z}_{h},\mathfrak{p}_{h})$, where $(\mathfrak{Z}_{h},\mathfrak{p}_{h}) \in \mathbf{V}_{h}\times Q_{h}$ corresponds to the unique solution to the discrete problem
\begin{equation*}
\nu (\nabla \mathfrak{Z}_{h}, \nabla \mathbf{v}_{h})_{\mathbf{L}^2(\Omega)} + (\mathfrak{p}_{h},\text{div } \mathbf{v}_{h})_{L^2(\Omega)}  =  \sum_{t\in\mathcal{D}}\langle (\mathbf{y}_{h}(t) - \mathbf{y}_{t})\delta_{t},\mathbf{v}_{h}\rangle, 
\, \, \, \,
(q_{h},\text{div } \mathfrak{Z}_{h})_{L^2(\Omega)} = 0,
\end{equation*}
for all $(\mathbf{v}_{h},q_{h})\in\mathbf{V}_{h} \times Q_{h}$. Here, $(\mathbf{y}_{h},p_{h})$ corresponds to the unique solution of problem \eqref{eq:discrete_state_eq}. To do this, we note that the relation $\mathfrak{N}_{\mathbf{y}_{h}}^{h}(\mathbf{z}_{h},r_{h}) = (\mathfrak{Z}_{h},\mathfrak{p}_{h})$ is equivalent to 
\begin{equation*}
(\mathbf{z}_{h},r_{h}) = \mathfrak{L}_{h}\left[\sum_{t\in\mathcal{D}}(\mathbf{y}_{h}(t) - \mathbf{y}_{t})\delta_{t} - [\mathfrak{B}^{\star}(\mathbf{y}_{h})\mathbf{z}_{h}]\right],
\end{equation*}
which in turns is equivalent to the fact that $(\mathbf{z}_{h},r_{h})$ solves problem \eqref{eq:adj_eq_discrete}. Since $\mathfrak{N}_{\mathbf{y}_{h}}^{h}$ is an automorphism, we conclude that problem \eqref{eq:adj_eq_discrete} is well posed.
\end{proof}


\subsection{First order optimality conditions}
With the well-posedness of problem \eqref{eq:adj_eq_discrete}, we are now in a position to establish first order optimality conditions.

\begin{theorem}[first order optimality conditions: the fully discrete scheme]
\label{thm:discre_opt_cond} 
If $h < \min \{ h_{\nabla}, h_{\dagger} \}$ and $(\bar{\mathbf{y}}_{h}, \bar{p}_{h}, \bar{\mathbf{u}}_{h})$ is a local solution of the fully discrete problem, then
\begin{equation}\label{eq:var_ineq_fully}
(\bar{\mathbf{z}}_{h} + \alpha\bar{\mathbf{u}}_{h}, \mathbf{u}_{h} - \bar{\mathbf{u}}_{h})_{\mathbf{L}^2(\Omega)} \geq 0 \quad \forall \mathbf{u}_h \in\mathbf{U}_{ad,h},
\end{equation}
where $(\bar{\mathbf{z}}_{h},\bar{r}_{h})\in\mathbf{V}_{h}\times Q_{h}$ is the solution of \eqref{eq:adj_eq_discrete}, where $\mathbf{y}_{h}$ is  replaced by $\bar{\mathbf{y}}_{h}$.
\end{theorem}
\begin{proof}
The proof is standard. For brevity, we omit the details.
\end{proof}

\begin{theorem}[first order optimality conditions: the semidiscrete scheme] 
If $h < \min \{ h_{\Box}, h_{\dagger} \}$ and $(\bar{\mathbf{y}}_{h}, \bar{p}_{h}, \bar{\boldsymbol{u}}_{h})$ is a local solution of the semidiscrete problem, then
\begin{equation}\label{eq:var_ineq_semi}
(\bar{\mathbf{z}}_{h} + \alpha\bar{\boldsymbol{u}}_h, \boldsymbol{u} - \bar{\boldsymbol{u}}_h)_{\mathbf{L}^2(\Omega)} \geq 0 \quad \forall \boldsymbol{u} \in\mathbf{U}_{ad},
\end{equation}
where $(\bar{\mathbf{z}}_{h},\bar{r}_{h})\in\mathbf{V}_{h}\times Q_{h}$ is the solution of \eqref{eq:adj_eq_discrete}, where $\mathbf{y}_{h}$ is replaced by $\bar{\mathbf{y}}_{h}$.
\end{theorem}
\begin{proof}
The proof is standard. For brevity, we omit the details.
\end{proof}

We conclude this section with the following projection formula for the \emph{semidiscrete scheme}: Since $\bar{\boldsymbol u}_h$ satisfies \eqref{eq:var_ineq_semi}, it holds that $\bar{\boldsymbol u}_h = \Pi_{[\textbf{a},\textbf{b}]}(-\alpha^{-1}\bar{\mathbf{z}}_{h})$.


\subsection{Auxiliary error estimates}

Let $(\bar{\mathbf{y}}, \bar{p}, \bar{\mathbf{u}})$ be a local nonsingular solution to \eqref{eq:weak_cost}--\eqref{eq:weak_st_eq} and let $\mathsf{p} < d/(d-1)$ be arbitrarily close to $d/(d-1)$. Let $h < h_{\dagger}$, where $h_{\dagger}$ is as in Theorem \ref{thm:existence_of_d_s_adj}.  We introduce the pair $(\hat{\mathbf{z}}_{h},\hat{r}_{h}) \in \mathbf{V}_{h}\times Q_h$ as the solution to
\begin{multline}\label{eq:adj_eq_hat}
\nu(\nabla \mathbf{w}_{h}, \nabla \hat{\mathbf{z}}_{h})_{\mathbf{L}^2(\Omega)} + b(\bar{\mathbf{y}};\mathbf{w}_{h},\hat{\mathbf{z}}_{h}) + b(\mathbf{w}_{h};\bar{\mathbf{y}},\hat{\mathbf{z}}_{h}) - (\hat{r}_{h},\text{div } \mathbf{w}_{h})_{L^2(\Omega)} \\ =  \sum_{t\in\mathcal{D}}\langle (\bar{\mathbf{y}}(t)-\mathbf{y}_{t})\delta_{t},\mathbf{w}_{h}\rangle_{\mathbf{W}^{-1,\mathsf{p} }(\Omega),\mathbf{W}^{1,\mathsf{q} }_0(\Omega)}, \qquad
(s_{h},\text{div } \hat{\mathbf{z}}_{h})_{L^2(\Omega)} = 0,
\end{multline}
for all $(\mathbf{w}_{h},s_{h})\in \mathbf{V}_{h} \times Q_{h}$. We note that \eqref{eq:adj_eq_hat} is well posed (cf.~Theorem \ref{thm:existence_of_d_s_adj}) and that $(\hat{\mathbf{z}}_{h},\hat{r}_{h})$ is a finite element approximation of $(\bar{\mathbf{z}},\bar{r})\in \mathbf{W}_0^{1,\mathsf{p}}(\Omega)\times L_0^{\mathsf{p}}(\Omega)$.

\begin{lemma}[auxiliary error estimate]\label{lemma:convergence_discr_adj}
Let $(\bar{\mathbf{y}}, \bar{p}, \bar{\mathbf{u}})$ be a local nonsingular solution to \eqref{eq:weak_cost}--\eqref{eq:weak_st_eq} and let $(\bar{\mathbf{z}},\bar{r})\in\mathbf{W}_0^{1,\mathsf{p}}(\Omega)\times L_0^{\mathsf{p}}(\Omega)$ be the associated adjoint state solving problem \eqref{eq:adj_eq}, where $\mathbf{y}$ is replaced by $\bar{\mathbf{y}}$. Then, we have the error estimate
\begin{equation}\label{eq:error_estimate_adj_hat}
\|\bar{\mathbf{z}} - \hat{\mathbf{z}}_{h}\|_{\mathbf{L}^{2}(\Omega)} \lesssim h^{2-\frac{d}{2}}\left[ \|\bar{\mathbf{u}}\|_{\mathbf{L}^{2}(\Omega)} + \sum_{t\in\mathcal{D}}|\mathbf{y}_{t}| \right] \quad \forall h <  h_{\dagger}. 
\end{equation}
\end{lemma}
\begin{proof}
Let $(\hat{\boldsymbol\varphi},\hat{\zeta})\in \mathbf{W}_0^{1,\mathsf{q}}(\Omega)\times L_0^{\mathsf{q}}(\Omega)$ be the unique solution to \eqref{eq:first_deriv_S_W01p}, where $\mathbf{y}$ and $\mathbf{g}$ are replaced by $\bar{\mathbf{y}}$ and $\bar{\mathbf{z}} - \hat{\mathbf{z}}_{h}$, respectively; observe that $\bar{\mathbf{z}} - \hat{\mathbf{z}}_{h} \in \mathbf{W}^{-1,\mathsf{q}}(\Omega)$. Let $(\hat{\boldsymbol\varphi}_{h},\hat{\zeta}_{h})\in\mathbf{V}_h\times Q_{h}$ be its finite element approximation given as the solution to
\begin{multline}\label{eq:hat_varphi_discrete}
\nu(\nabla \hat{\boldsymbol{\varphi}}_{h}, \nabla \mathbf{v}_{h})_{\mathbf{L}^2(\Omega)} + b(\bar{\mathbf{y}};\hat{\boldsymbol{\varphi}}_{h},\mathbf{v}_{h}) + b(\hat{\boldsymbol{\varphi}}_{h};\bar{\mathbf{y}},\mathbf{v}_{h}) \\
- (\hat{\zeta}_{h}, \textnormal{div }\mathbf{v}_{h})_{L^2(\Omega)}
= (\bar{\mathbf{z}} - \hat{\mathbf{z}}_{h},\mathbf{v}_{h})_{\mathbf{L}^2(\Omega)}, \qquad  (q_{h},\textnormal{div }\hat{\boldsymbol{\varphi}}_{h})_{L^2(\Omega)} = 0,
\end{multline}
for all $(\mathbf{v}_{h},q_{h})\in \mathbf{V}_{h}\times Q_{h}$. Similar arguments to those used to establish that the discrete adjoint equations \eqref{eq:adj_eq_discrete} are well posed show that \eqref{eq:hat_varphi_discrete} is also well posed. We now set $(\mathbf{v},q)$ $=$ $(\bar{\mathbf{z}} - \hat{\mathbf{z}}_{h},0)$ as a test function in \eqref{eq:first_deriv_S_W01p}, where $\mathbf{y}$ and $\mathbf{g}$ are replaced by $\bar{\mathbf{y}}$ and $\bar{\mathbf{z}} - \hat{\mathbf{z}}_{h}$, respectively, and use Galerkin orthogonality to obtain
\begin{equation*}
\|\bar{\mathbf{z}} - \hat{\mathbf{z}}_{h}\|_{\mathbf{L}^2(\Omega)}^2
=
\sum_{t\in\mathcal{D}}\langle (\bar{\mathbf{y}}(t) - \mathbf{y}_{t})\delta_{t}, \hat{\boldsymbol\varphi} - \hat{\boldsymbol\varphi}_{h}\rangle \lesssim \|\hat{\boldsymbol\varphi} - \hat{\boldsymbol\varphi}_{h}\|_{\mathbf{L}^{\infty}(\Omega)}\left[\|\bar{\mathbf{y}}\|_{\mathbf{L}^{\infty}(\Omega)} + \sum_{t\in\mathcal{D}}|\mathbf{y}_{t}|\right].
\end{equation*}
Note that $\|\hat{\boldsymbol\varphi} - \hat{\boldsymbol\varphi}_{h}\|_{\mathbf{L}^{\infty}(\Omega)} \lesssim h^{2-\frac{d}{2}}\|\bar{\mathbf{z}} - \hat{\mathbf{z}}_{h}\|_{\mathbf{L}^2(\Omega)}$. This error bound follows the same arguments as in the proof of Theorem \ref{thm:convergence_in_Linfty} upon using $\|\hat{\boldsymbol\varphi} - \hat{\boldsymbol\varphi}_{h}\|_{\mathbf{L}^{2}(\Omega)} \lesssim h^{2} \| \hat{\boldsymbol\varphi}\|_{\mathbf{H}^2(\Omega)}$. Note that since $\Omega$ is convex and $\bar{\mathbf{z}} - \hat{\mathbf{z}}_{h} \in \mathbf{L}^2(\Omega)$, $\| \hat{\boldsymbol\varphi}\|_{\mathbf{H}^2(\Omega)} \lesssim \|\bar{\mathbf{z}} - \hat{\mathbf{z}}_{h}\|_{\mathbf{L}^2(\Omega)}$. Finally, the trivial bound $\|\bar{\mathbf{y}}\|_{\mathbf{L}^{\infty}(\Omega)}\lesssim \|\bar{\mathbf{u}}\|_{\mathbf{L}^{2}(\Omega)}$ allows us to conclude \eqref{eq:error_estimate_adj_hat}.
\end{proof}

\begin{theorem}[auxiliary error bound]\label{thm:convergence_adj_eq}
Let $(\bar{\mathbf{y}}, \bar{p}, \bar{\mathbf{u}})$ be a local nonsingular solution to \eqref{eq:weak_cost}--\eqref{eq:weak_st_eq}. Let $h < \min \{h_{\nabla} ,h_{\dagger} \}¸ $ and let $(\bar{\mathbf{y}}_h, \bar{p}_h, \bar{\mathbf{u}}_h)$ be a solution to the fully discrete problem as stated in Theorem \ref{thm:exist_and_conv_sol}. Then, we have the error estimate
\begin{equation*}
\|\bar{\mathbf{z}} - \bar{\mathbf{z}}_{h}\|_{\mathbf{L}^{2}(\Omega)} \lesssim h^{2-\frac{d}{2}} + \|\bar{\mathbf{u}} - \bar{\mathbf{u}}_{h}\|_{\mathbf{L}^{2}(\Omega)}.
\end{equation*}
Here, $(\bar{\mathbf{z}},\bar{r})$ corresponds to the optimal adjoint state and $(\bar{\mathbf{z}}_{h},\bar{r}_{h})$ denotes its finite element approximation.
\end{theorem}
\begin{proof}
We begin the proof with an application of the error estimate \eqref{eq:error_estimate_adj_hat}:
\begin{equation}\label{eq:estimate_barz-hatz}
\|\bar{\mathbf{z}} - \bar{\mathbf{z}}_{h}\|_{\mathbf{L}^{2}(\Omega)} \leq \|\bar{\mathbf{z}} - \hat{\mathbf{z}}_{h}\|_{\mathbf{L}^{2}(\Omega)} + \|\hat{\mathbf{z}}_{h} - \bar{\mathbf{z}}_{h}\|_{\mathbf{L}^{2}(\Omega)} \lesssim h^{2-\frac{d}{2}} + \|\hat{\mathbf{z}}_{h} - \bar{\mathbf{z}}_{h}\|_{\mathbf{L}^{2}(\Omega)}.
\end{equation} 
Here, $(\hat{\mathbf{z}}_{h},\hat{r}_{h})$ denotes the solution to \eqref{eq:adj_eq_hat}. It thus suffices to bound $\bar{\mathbf{z}}_{h} - \hat{\mathbf{z}}_{h}$ in $\mathbf{L}^2(\Omega)$ To do this, we note that $( \bar{\mathbf{z}}_{h} - \hat{\mathbf{z}}_{h}, \bar{r}_{h} - \hat{r}_{h})\in \mathbf{V}_{h}\times Q_{h}$ solves 
\begin{multline*}
\nu(\nabla \mathbf{w}_{h}, \nabla (\bar{\mathbf{z}}_{h} - \hat{\mathbf{z}}_{h}))_{\mathbf{L}^2(\Omega)} + b(\bar{\mathbf{y}}_{h};\mathbf{w}_{h},\bar{\mathbf{z}}_{h} - \hat{\mathbf{z}}_{h}) + b(\mathbf{w}_{h};\bar{\mathbf{y}}_{h},\bar{\mathbf{z}}_{h} - \hat{\mathbf{z}}_{h}) \\ 
- (\bar{r}_{h} - \hat{r}_{h}, \text{div }\mathbf{w}_{h})_{L^2(\Omega)} 
\!
=
\!
\sum_{t\in\mathcal{D}}\langle(\bar{\mathbf{y}}_{h} - \bar{\mathbf{y}})(t)\delta_{t}, \mathbf{w}_{h}\rangle 
+ 
b(\bar{\mathbf{y}} - \bar{\mathbf{y}}_{h}; \mathbf{w}_{h},\hat{\mathbf{z}}_{h}) 
+ 
b(\mathbf{w}_{h};\bar{\mathbf{y}} - \bar{\mathbf{y}}_{h},\hat{\mathbf{z}}_{h})
\end{multline*}
and $(s_{h},\text{div}(\bar{\mathbf{z}}_{h} - \hat{\mathbf{z}}_{h}))_{L^2(\Omega)} = 0$ for all $(\mathbf{w}_{h}, s_{h})\in \mathbf{V}_{h}\times Q_{h}$. An application of Theorem \ref{thm:existence_of_d_s_adj} shows that this problem is well-posed. In particular, we have the stability bound
\begin{equation*}
\|\nabla (\bar{\mathbf{z}}_{h} - \hat{\mathbf{z}}_{h})\|_{\mathbf{L}^{\mathsf{p}}(\Omega)} \lesssim 
\|\bar{\mathbf{y}} - \bar{\mathbf{y}}_{h}\|_{\mathbf{L}^{\infty}(\Omega)}(1 + \|\hat{\mathbf{z}}_{h}\|_{\mathbf{L}^{\mathsf{p}}(\Omega)}) + \|\nabla(\bar{\mathbf{y}} - \bar{\mathbf{y}}_{h})\|_{\mathbf{L}^2(\Omega)}\|\hat{\mathbf{z}}_{h}\|_{\mathbf{L}^{2}(\Omega)},
\end{equation*}
upon using 
$\mathbf{W}_{0}^{1,\mathsf{q}}(\Omega)\hookrightarrow \mathbf{C}(\bar{\Omega})$, where $\mathsf{q} > d$. We now use $\mathbf{W}_0^{1,\mathsf{p}}(\Omega)\hookrightarrow \mathbf{L}^2(\Omega)$, a Poincar\'e's inequality, and the well-posedness of \eqref{eq:adj_eq_hat} to obtain $\|\nabla (\bar{\mathbf{z}}_{h} - \hat{\mathbf{z}}_{h})\|_{\mathbf{L}^{\mathsf{p}}(\Omega)} \lesssim \|\bar{\mathbf{y}} - \bar{\mathbf{y}}_{h}\|_{\mathbf{L}^{\infty}(\Omega)} + \|\nabla(\bar{\mathbf{y}} - \bar{\mathbf{y}}_{h})\|_{\mathbf{L}^2(\Omega)}$. 

If necessary, we restrict $h$ further so that $\bar{\mathbf{u}}_h \in \mathcal{O}(\bar{\mathbf{u}})$. Let $(\hat{\mathbf{y}},\hat{p})\in \mathbf{H}_0^1(\Omega)\times L_0^2(\Omega)$ be the unique solution to problem \eqref{eq:weak_st_eq}, where $\mathbf{u}$ is replaced by $\bar{\mathbf{u}}_{h}$ (cf.~Theorem \ref{thm:properties_C_to_S}). We now use basic triangle inequalities, the Lipschitz properties from Proposition \ref{pro:Lipschitz_property} and Theorem \ref{thm:Lipschitz_property}, Theorem \ref{thm:convergence_in_Linfty}, and standard error estimates for the finite element discretizations we consider for the Navier--Stokes equations to obtain
\begin{multline*}
\|\nabla (\bar{\mathbf{z}}_{h} - \hat{\mathbf{z}}_{h})\|_{\mathbf{L}^{\mathsf{p}}(\Omega)} \lesssim
\|\bar{\mathbf{y}} - \hat{\mathbf{y}}\|_{\mathbf{L}^{\infty}(\Omega)} + \|\hat{\mathbf{y}} - \bar{\mathbf{y}}_{h}\|_{\mathbf{L}^{\infty}(\Omega)} \\ + \|\nabla(\bar{\mathbf{y}} - \hat{\mathbf{y}})\|_{\mathbf{L}^2(\Omega)} + \|\nabla(\hat{\mathbf{y}} - \bar{\mathbf{y}}_{h})\|_{\mathbf{L}^2(\Omega)} \lesssim \|\bar{\mathbf{u}} - \bar{\mathbf{u}}_{h}\|_{\mathbf{L}^{2}(\Omega)} + h^{2-\frac{d}{2}} + h.
\end{multline*}
The embedding $\mathbf{W}_0^{1,\mathsf{p}}(\Omega)\hookrightarrow \mathbf{L}^2(\Omega)$ shows that $\|\bar{\mathbf{z}}_{h} - \hat{\mathbf{z}}_{h}\|_{\mathbf{L}^{2}(\Omega)} \lesssim  \|\nabla (\bar{\mathbf{z}}_{h} - \hat{\mathbf{z}}_{h})\|_{\mathbf{L}^{\mathsf{p}}(\Omega)} \lesssim \|\bar{\mathbf{u}} - \bar{\mathbf{u}}_{h}\|_{\mathbf{L}^{2}(\Omega)} + h^{\iota}$, where $\iota = 2 - d/2$. Replace this estimate into \eqref{eq:estimate_barz-hatz} to conclude.
\end{proof}


\section{Error estimates}\label{sec:error_estimates}
In this section, we derive error estimates for the fully discrete and semidiscrete schemes presented in sections \ref{sec:fully_disc_scheme} and \ref{sec:semi_disc_scheme}, respectively.


\subsection{Error estimates: the fully discrete scheme} 

Let $(\bar{\mathbf{y}}, \bar{p}, \bar{\mathbf{u}})$ be a strict local minimum of \eqref{eq:weak_cost}--\eqref{eq:weak_st_eq} and let $\{ (\bar{\mathbf{y}}_h, \bar{p}_h, \bar{\mathbf{u}}_h)\}_{h<h_{\nabla}}$ be a sequence of local minima of the fully discrete optimal control problems such that $\bar{\mathbf{u}}_{h} \to \bar{\mathbf{u}}$ in $\mathbf{L}^2(\Omega)$ as $h \rightarrow 0$. The main goal of this section is to derive the error bound
\begin{equation}\label{eq:error_estimate_fully}
\|\bar{\mathbf{u}}-\bar{\mathbf{u}}_h\|_{\mathbf{L}^2(\Omega)}
\lesssim h^{\ell} 
\quad 
\forall h < h_{\ddagger},
\quad \
h_{\ddagger}>0,
\end{equation}
where $\ell = 1 - d/p + d/2$; $p<d/(d-1)$ is arbitrarily close to $d/(d-1)$.

The following result is useful for deriving the error estimate \eqref{eq:error_estimate_fully}.

\begin{lemma}[auxiliary error estimate]\label{lemma:aux_error_estimate}
If $\bar{\mathbf{u}}\in\mathbf{U}_{ad}$ satisfies the second order optimality conditions \eqref{eq:second_order_2_2}, and \eqref{eq:error_estimate_fully} is false, then there exists $h_{\ddagger} > 0$ such that
\begin{equation}\label{eq:aux_estimate}
\textgoth{C}
\|\bar{\mathbf{u}}-\bar{\mathbf{u}}_h\|_{\mathbf{L}^2(\Omega)}^2
\leq 
[j'(\bar{\mathbf{u}}_h)-j'(\bar{\mathbf{u}})](\bar{\mathbf{u}}_h-\bar{\mathbf{u}}) \quad \forall h < h_{\ddagger},
\quad
\textgoth{C} = 2^{-1} \min\{\mu,\alpha\}.
\end{equation}
Here, $\alpha$ is the control cost and $\mu$ is the constant appearing in \eqref{eq:second_order_2_2}.
\end{lemma}
\begin{proof}
We follow \cite[Section 7]{MR2272157} and proceed by contradiction. Let us assume that \eqref{eq:error_estimate_fully} does not hold. So, we can extract a subsequence $\{h_k\}_{k \in \mathbb{N}} \subset \mathbb{R}^{+}$ such that 
\begin{equation}\label{eq:estimate_contradiction}
\lim_{ h_{k}\rightarrow 0}\|\bar{\mathbf{u}}-\bar{\mathbf{u}}_{h_k}\|_{\mathbf{L}^2(\Omega)} = 0, 
\qquad
\lim_{h_{k} \rightarrow 0}  (h_{k}^{\ell})^{-1}  \|\bar{\mathbf{u}}-\bar{\mathbf{u}}_{h_{k}}\|_{\mathbf{L}^2(\Omega)}=+\infty.
\end{equation}
In what follows, to simplify notation we omit the subindex $k$.

Define $\mathbf{g}_h:= (\bar{\mathbf{u}}_h-\bar{\mathbf{u}})/\|\bar{\mathbf{u}}_h-\bar{\mathbf{u}}\|_{\mathbf{L}^2(\Omega)}$. Since $\{ \mathbf{g}_h \}_{0<h<h_{\nabla}}$ is uniformly bounded in $\mathbf{L}^2(\Omega)$, we can assume that $\mathbf{g}_{h} \rightharpoonup \mathbf{g}$ in $\mathbf{L}^{2}(\Omega)$ as $h \to 0$, up to a subsequence if necessary. We now prove that $\mathbf{g}\in \mathbf{C}_{\bar{\mathbf{u}}}$, where $\mathbf{C}_{\bar{\mathbf{u}}}$ is defined in \eqref{def:critical_cone}. Since for any $h \in (0,h_{\nabla})$, $\bar{\mathbf{u}}_h \in \mathbf{U}_{ad,h}\subset \mathbf{U}_{ad}$, $\mathbf{g}_h$ satisfies the sign conditions \eqref{eq:sign_cond}. The weak limit $\mathbf{g}$ also satisfies \eqref{eq:sign_cond}. We now show that $\bar{\mathbf{d}}_{i}(x) \neq 0$ implies that $\mathbf{g}_{i}(x) = 0$ for a.e.~$x\in\Omega$ and $i\in\{1,\ldots,d\}$. To this end, we introduce $\bar{\mathbf{d}}_h:= \bar{\mathbf{z}}_{h} + \alpha\bar{\mathbf{u}}_h$. Recall that $\bar{\mathbf{d}}= \bar{\mathbf{z}} + \alpha\bar{\mathbf{u}}$. Relying on $\|\bar{\mathbf{u}} - \bar{\mathbf{u}}_{h}\|_{\mathbf{L}^2(\Omega)}\to 0$ as $h \to 0$ and Theorem \ref{thm:convergence_adj_eq}, we obtain
\begin{equation*}
\| \bar{\mathbf{d}} - \bar{\mathbf{d}}_h\|_{\mathbf{L}^2(\Omega)} 
\leq 
\| \bar{\mathbf{z}} - \bar{\mathbf{z}}_h\|_{\mathbf{L}^2(\Omega)} + 
\alpha\| \bar{\mathbf{u}} - \bar{\mathbf{u}}_h\|_{\mathbf{L}^2(\Omega)} \to 0, \quad  h \to 0.
\end{equation*}
From this it follows that $\bar{\mathbf{d}}_{h} \to \bar{\mathbf{d}}$ in $\mathbf{L}^2(\Omega)$ as $h \to 0$. Consequently, we obtain
\begin{equation*}
(\bar{\mathbf{d}},\mathbf{g})_{\mathbf{L}^2(\Omega)}
\!
=
\!
\lim_{h \to 0}(\bar{\mathbf{d}}_{h},\mathbf{g}_{h})_{\mathbf{L}^2(\Omega)} 
\!
= 
\!
\lim_{h \to 0}\frac{(\bar{\mathbf{d}}_h,\Pi_{\mathbf{L}^2}(\bar{\mathbf{u}})-\bar{\mathbf{u}})_{\mathbf{L}^2(\Omega)} \!+\! (\bar{\mathbf{d}}_h,\bar{\mathbf{u}}_h - \Pi_{\mathbf{L}^2}(\bar{\mathbf{u}}))_{\mathbf{L}^2(\Omega)}}{\|\bar{\mathbf{u}}_h-\bar{\mathbf{u}}\|_{\mathbf{L}^2(\Omega)}},
\end{equation*}
where $\Pi_{\mathbf{L}^2}: \mathbf{L}^2(\Omega) \rightarrow \mathbf{U}_h$ is the orthogonal projection operator. 
Since $\Pi_{\mathbf{L}^2}(\bar{\mathbf{u}})\in \mathbf{U}_{ad,h}$, the discrete variational inequality \eqref{eq:var_ineq_fully} yields $ (\bar{\mathbf{d}}_h,\bar{\mathbf{u}}_h - \Pi_{\mathbf{L}^2}(\bar{\mathbf{u}}))_{\mathbf{L}^2(\Omega)}  \leq 0$. We now note that there exists $h_{\bowtie}>0$ such that $\{ \|\bar{\mathbf{d}}_{h}\|_{\mathbf{L}^2(\Omega)} \}_{h < h_{\bowtie}}$ is uniformly bounded. In fact, $\|\bar{\mathbf{d}}_{h}\|_{\mathbf{L}^2(\Omega)}\leq \|\bar{\mathbf{d}}_{h}-\bar{\mathbf{d}}\|_{\mathbf{L}^2(\Omega)}+\|\bar{\mathbf{d}}\|_{\mathbf{L}^2(\Omega)} \lesssim 1$ for every $h < h_{\bowtie}$. Thus,
\begin{equation*}
(\bar{\mathbf{d}},\mathbf{g})_{\mathbf{L}^2(\Omega)}
\leq 
\lim_{h \to 0}\frac{(\bar{\mathbf{d}}_h,\Pi_{\mathbf{L}^2}(\bar{\mathbf{u}})-\bar{\mathbf{u}})_{\mathbf{L}^2(\Omega)}}{\|\bar{\mathbf{u}}_h-\bar{\mathbf{u}}\|_{\mathbf{L}^2(\Omega)}}
\lesssim
\lim_{h\downarrow 0}
\frac{\|\Pi_{\mathbf{L}^2}(\bar{\mathbf{u}})-\bar{\mathbf{u}}\|_{\mathbf{L}^2(\Omega)}}{\|\bar{\mathbf{u}}_h-\bar{\mathbf{u}}\|_{\mathbf{L}^2(\Omega)}}
=0.
\end{equation*}
To obtain the latter equality, we used \eqref{eq:estimate_contradiction} and the bound $\|\bar{\mathbf{u}} - \Pi_{\mathbf{L}^2}(\bar{\mathbf{u}})\|_{\mathbf{L}^2(\Omega)} \lesssim h^{\ell}$, which follows from the fact that $\bar{\mathbf{u}}\in\mathbf{W}^{1,\mathsf{p}}(\Omega)$, where $\mathsf{p} < d/(d-1)$ is arbitrarily close to $d/(d-1)$. On the other hand, since $\mathbf{g}_i$ satisfies \eqref{eq:sign_cond} and $\mathbf{d}_i$ satisfies \eqref{eq:derivative_j}, we have $\bar{\mathbf{d}}_{i}(x)\mathbf{g}_{i}(x) \geq 0$ for each $i\in\{1,\ldots,d\}$. Thus, $\int_{\Omega} |\bar{\mathbf{d}}_{1}\mathbf{g}_{1}| + \cdots + |\bar{\mathbf{d}}_{d}\mathbf{g}_{d}| = 0$. We can deduce that if $\bar{\mathbf{d}}_{i}(x)\neq 0$, then $\mathbf{g}_{i}(x) = 0$ for a.e.~$x\in\Omega$ and $i\in\{1,\ldots,d\}$. 

We now derive \eqref{eq:aux_estimate} by using \eqref{eq:second_order_2_2}. As a first step, we note that
\begin{equation}\label{eq:difference_of_j}
[j'(\bar{\mathbf{u}}_h)-j'(\bar{\mathbf{u}})](\bar{\mathbf{u}}_h-\bar{\mathbf{u}})= j''(\hat{\mathbf{u}}_h)(\bar{\mathbf{u}}_h-\bar{\mathbf{u}})^2, 
\quad \hat{\mathbf{u}}_h = \bar{\mathbf{u}} + \theta_h (\bar{\mathbf{u}}_h - \bar{\mathbf{u}}),
\quad \theta_h \in (0,1). 
\end{equation}
If necessary, we restrict $h$ further so that $\hat{\mathbf{u}}_h \in \mathcal{O}(\bar{\mathbf{u}})$. Let $(\mathbf{y}(\hat{\mathbf{u}}_h),p(\hat{\mathbf{u}}_{h}))$ be the unique solution to \eqref{eq:weak_st_eq}, where $\mathbf{u}$ is replaced by $\hat{\mathbf{u}}_{h}$ (cf.~Theorem \ref{thm:properties_C_to_S}). Let $(\mathbf{z}(\hat{\mathbf{u}}_h),r(\hat{\mathbf{u}}_h))$ be the unique solution to \eqref{eq:adj_eq} where $\mathbf{y}$ is replaced by $\mathbf{y}(\hat{\mathbf{u}}_h)$ (cf.~Theorem \ref{thm:well-posedness-adjoint}). Note that $\mathbf{y}(\hat{\mathbf{u}}_h)$ is regular (cf.~Theorem \ref{thm:properties_C_to_S}). The Lipschitz property of Theorem \ref{thm:Lipschitz_property}, the compact embedding $\mathbf{W}_{0}^{1,\kappa}(\Omega)\hookrightarrow \mathbf{C}(\bar{\Omega})$, where $\kappa > d$, and the convergence $\bar{\mathbf{u}}_h \rightarrow \bar{\mathbf{u}}$ in $\mathbf{L}^2(\Omega)$ as $h \to 0$, show that $\mathbf{y}(\hat{\mathbf{u}}_h) \rightarrow \bar{\mathbf{y}}$ in $\mathbf{W}_{0}^{1,\kappa}(\Omega) \cap \mathbf{C}(\bar{\Omega})$ when $h \to 0$. Similarly, we obtain that $\mathbf{z}(\hat{\mathbf{u}}_h) \rightarrow \bar{\mathbf{z}}$ in $\mathbf{W}_{0}^{1,\mathsf{p}}(\Omega)$ when $h \to 0$, where $\mathsf{p} < d/(d-1)$ is arbitrarily close to $d/(d-1)$. Let us now define $(\boldsymbol\varphi(\mathbf{g}_h),\zeta(\mathbf{g}_h))$ as the unique solution to \eqref{eq:first_deriv_S*} replacing $\mathbf{y}$ and $\mathbf{g}$ by $\mathbf{y}(\hat{\mathbf{u}}_h)$ and $\mathbf{g}_h$, respectively. As in Step 2 of the proof of Theorem \ref{thm:suff_opt_cond}, we can show that $\mathbf{g}_{h} \rightharpoonup \mathbf{g}$ in $\mathbf{L}^2(\Omega)$ as $h \to 0$ guarantees that $\boldsymbol\varphi(\mathbf{g}_h) \to \boldsymbol\varphi$ in $\mathbf{C}(\bar{\Omega})$. Here, $\boldsymbol\varphi$ solves \eqref{eq:first_deriv_S*} where $\mathbf{y}$ is replaced by $\bar{\mathbf{y}}$. Therefore, \eqref{eq:charac_j2}, the convergence properties derived for $\{ \mathbf{z}(\hat{\mathbf{u}}_{h}) \}_h$ and $\{ \boldsymbol\varphi(\mathbf{g}_{h}) \}_h$, where $h < \min \{ h_{\nabla}, h_{\dagger}, h_{\bowtie} \}$, the definition of $\mathbf{g}_h$, and the second order optimality conditions \eqref{eq:second_order_2_2} allow us to conclude that
\begin{multline*}
\lim_{h \to 0}j''(\hat{\mathbf{u}}_{h})\mathbf{g}_{h}^2
= 
\lim_{h \to 0}
\left[
\alpha\|\mathbf{g}_{h}\|_{\mathbf{L}^2(\Omega)}^2 
- 
2b(\boldsymbol\varphi(\mathbf{g}_h);\boldsymbol\varphi(\mathbf{g}_h),\mathbf{z}(\hat{\mathbf{u}}_h)) 
+ 
\sum_{t\in \mathcal{D}}\boldsymbol\varphi(\mathbf{g}_h)^2(t)
 \right]
 \\
= \alpha
- 2b(\boldsymbol\varphi;\boldsymbol\varphi,\bar{\mathbf{z}}) + \sum_{t\in \mathcal{D}}\boldsymbol\varphi^2(t)
= 
\alpha + j''(\bar{\mathbf{u}})\mathbf{g}^2 - \alpha\|\mathbf{g}\|_{\mathbf{L}^2(\Omega)}^2
\geq \alpha +(\mu-\alpha) \|\mathbf{g}\|_{\mathbf{L}^2(\Omega)}^2.
\end{multline*}
Thus, since $\|\mathbf{g}\|_{\mathbf{L}^2(\Omega)}\leq 1$, we obtain $\lim_{h \to 0}j''(\hat{\mathbf{u}}_{h})\mathbf{g}_{h}^2 \geq \min\{\mu,\alpha\}>0$. We thus obtain the existence of $h_{\ddagger} \in (0,\min \{ h_{\nabla}, h_{\dagger}, h_{\bowtie} \})$ such that $j''(\hat{\mathbf{u}}_{h})\mathbf{g}_{h}^2 \geq 2^{-1}\min\{\mu,\alpha\}$ for each $h< h_{\ddagger}$. Given the definition of $\mathbf{g}_h$ and \eqref{eq:difference_of_j}, this allows us to conclude.
\end{proof}

We now proceed to derive the error bound \eqref{eq:error_estimate_fully}. 

\begin{theorem}[error estimate]\label{thm:error_estimate_fully}
If $\bar{\mathbf{u}}\in\mathbf{U}_{ad}$ satisfies the second order optimality conditions \eqref{eq:second_order_2_2}, then there exists $h_{\ddagger} > 0$ such that the error estimate \eqref{eq:error_estimate_fully} holds.
\end{theorem}
\begin{proof}
We proceed by contradiction and assume that \eqref{eq:error_estimate_fully} does not hold. Thus, given Lemma \ref{lemma:aux_error_estimate}, the estimate \eqref{eq:aux_estimate} holds for each $h < 
h_{\ddagger}$. If we now set $\mathbf{u}=\bar{\mathbf{u}}_{h}$ in \eqref{eq:variational_inequality} and $\mathbf{u}_{h} = \Pi_{\mathbf{L}^2}(\bar{\mathbf{u}})$ in \eqref{eq:var_ineq_fully}, we obtain $- j'(\bar{\mathbf{u}})(\bar{\mathbf{u}}_{h} - \bar{\mathbf{u}}) \leq 0$ and $(\bar{\mathbf{\mathbf{z}}}_{h} + \alpha\bar{\mathbf{u}}_{h},\Pi_{\mathbf{L}^2}(\bar{\mathbf{u}}) - \bar{\mathbf{u}}_{h})_{\mathbf{L}^2(\Omega)} \geq 0$, respectively. These two bounds in combination with \eqref{eq:aux_estimate} result in
\begin{multline}
\label{eq:FD_ct_error_estimate_0}
\|\bar{\mathbf{u}} - \bar{\mathbf{u}}_{h}\|_{\mathbf{L}^2(\Omega)}^2 
\lesssim j'(\bar{\mathbf{u}}_{h})(\bar{\mathbf{u}}_{h} - \bar{\mathbf{u}}) + (\bar{\mathbf{\mathbf{z}}}_{h} + \alpha\bar{\mathbf{u}}_{h},\Pi_{\mathbf{L}^2}(\bar{\mathbf{u}}) - \bar{\mathbf{u}}_{h})_{\mathbf{L}^2(\Omega)}
\\
= (\tilde{\mathbf{z}}(\bar{\mathbf{u}}_{h}) + \alpha\bar{\mathbf{u}}_{h},\Pi_{\mathbf{L}^2}(\bar{\mathbf{u}}) - \bar{\mathbf{u}})_{\mathbf{L}^2(\Omega)} + (\tilde{\mathbf{z}}(\bar{\mathbf{u}}_{h}) - \bar{\mathbf{z}}_{h},\bar{\mathbf{u}}_{h} - \Pi_{\mathbf{L}^2}(\bar{\mathbf{u}}))_{\mathbf{L}^2(\Omega)} =:\mathbf{I}_h + \mathbf{II}_h;
\end{multline}
$(\tilde{\mathbf{z}}(\bar{\mathbf{u}}_{h}),\tilde{r}(\bar{\mathbf{u}}_{h})$ is the solution to \eqref{eq:adj_eq} (cf.~Theorem \ref{thm:well-posedness-adjoint}), where $\mathbf{y}$ is replaced by $\mathbf{y}(\bar{\mathbf{u}}_{h})$, and $(\mathbf{y}(\bar{\mathbf{u}}_{h}),p(\bar{\mathbf{u}}_{h}))$ is the solution to \eqref{eq:weak_st_eq} (cf.~Theorem \ref{thm:properties_C_to_S}), where $\mathbf{u}$ is replaced by $\bar{\mathbf{u}}_{h}$. If necessary, we further restrict $h$ so that $\bar{\mathbf{u}}_h \in \mathcal{O}(\bar{\mathbf{u}})$ and $\mathbf{y}(\bar{\mathbf{u}}_{h})$ is regular.

We first estimate $\mathbf{I}_h$ in \eqref{eq:FD_ct_error_estimate_0}. Using standard properties of  $\Pi_{\mathbf{L}^2}$ and the fact that $\bar{\mathbf{u}},\tilde{\mathbf{z}}(\bar{\mathbf{u}}_{h}) \in\mathbf{W}^{1,\mathsf{p}}(\Omega)$, where $\mathsf{p} < d/(d-1)$ is arbitrarily close to $d/(d-1)$, we obtain
\begin{align*}
\mathbf{I}_h &= (\tilde{\mathbf{z}}(\bar{\mathbf{u}}_{h})  - \Pi_{\mathbf{L}^2}(\tilde{\mathbf{z}}(\bar{\mathbf{u}}_{h})),\Pi_{\mathbf{L}^2}(\bar{\mathbf{u}}) - \bar{\mathbf{u}})_{\mathbf{L}^2(\Omega)} \\
&\leq \|\tilde{\mathbf{z}}(\bar{\mathbf{u}}_{h})  - \Pi_{\mathbf{L}^2}(\tilde{\mathbf{z}}(\bar{\mathbf{u}}_{h}))\|_{\mathbf{L}^2(\Omega)} \|\Pi_{\mathbf{L}^2}(\bar{\mathbf{u}}) - \bar{\mathbf{u}}\|_{\mathbf{L}^2(\Omega)} \lesssim h^{2\ell} \quad \forall h < h_{\ddagger}.
\end{align*}

We now bound $\mathbf{II}_h$. Note that $\mathbf{II}_h = (\tilde{\mathbf{z}}(\bar{\mathbf{u}}_{h}) - \bar{\mathbf{z}}_{h},\Pi_{\mathbf{L}^2}(\bar{\mathbf{u}}_{h} - \bar{\mathbf{u}}))_{\mathbf{L}^2(\Omega)}
$. Thus,
\begin{equation*}
\mathbf{II}_h
\leq 
 \dfrac{1}{2\mathfrak{c}}\|\tilde{\mathbf{z}}(\bar{\mathbf{u}}_{h}) - \bar{\mathbf{z}}_{h}\|_{\mathbf{L}^2(\Omega)}^2 + \dfrac{\mathfrak{c}}{2}\|\bar{\mathbf{u}}_{h} - \bar{\mathbf{u}}\|_{\mathbf{L}^2(\Omega)}^2, \quad
 \mathfrak{c} > 0.
\end{equation*}
To bound $\|\tilde{\mathbf{z}}(\bar{\mathbf{u}}_{h}) - \bar{\mathbf{z}}_{h}\|_{\mathbf{L}^2(\Omega)}$, we introduce 
$(\tilde{\mathbf{z}}_{h},\tilde{r}_{h})\in\mathbf{V}_{h}\times Q_{h}$ 
as the solution to \eqref{eq:adj_eq_hat}, where $\bar{\mathbf{y}}$ is replaced by $\mathbf{y}(\bar{\mathbf{u}}_{h})$ (cf.~Theorem \ref{thm:existence_of_d_s_adj}); $(\tilde{\mathbf{z}}_{h},\tilde{r}_{h})$ is a 
finite element approximation of $(\tilde{\mathbf{z}}(\bar{\mathbf{u}}_{h}),\tilde{r}(\bar{\mathbf{u}}_{h}))$. An analogous estimation as in \eqref{eq:error_estimate_adj_hat} thus yields
\begin{equation*}
\|\tilde{\mathbf{z}}(\bar{\mathbf{u}}_{h}) - \bar{\mathbf{z}}_{h}\|_{\mathbf{L}^2(\Omega)}
\lesssim 
\|\tilde{\mathbf{z}}(\bar{\mathbf{u}}_h) - \tilde{\mathbf{z}}_{h}\|_{\mathbf{L}^2(\Omega)} + \|\tilde{\mathbf{z}}_{h} - \bar{\mathbf{z}}_{h}\|_{\mathbf{L}^2(\Omega)}
\lesssim
h^{2 - \frac{d}{2}} + \|\tilde{\mathbf{z}}_{h} - \bar{\mathbf{z}}_{h}\|_{\mathbf{L}^2(\Omega)}.
\end{equation*}

Finally, we control $\|\tilde{\mathbf{z}}_{h} - \bar{\mathbf{z}}_{h}\|_{\mathbf{L}^2(\Omega)}$. To accomplish this task, we invoke a stability estimate for the problem that $(\tilde{\mathbf{z}}_{h} - \bar{\mathbf{z}}_{h}, \tilde{r}_{h}-\bar{r}_{h})$ solves and obtain that $\|\tilde{\mathbf{z}}_{h} - \bar{\mathbf{z}}_{h}\|_{\mathbf{L}^2(\Omega)} \lesssim \|\nabla(\tilde{\mathbf{z}}_{h} - \bar{\mathbf{z}}_{h})\|_{\mathbf{L}^{\mathsf{p}}(\Omega)} \lesssim \|\mathbf{y}(\bar{\mathbf{u}}_h) - \bar{\mathbf{y}}_{h}\|_{\mathbf{L}^{\infty}(\Omega)}$. 
With these bounds, using the error estimate from Theorem \ref{thm:convergence_in_Linfty}, we obtain $\|\tilde{\mathbf{z}}_{h} - \bar{\mathbf{z}}_{h}\|_{\mathbf{L}^2(\Omega)} \lesssim h^{\iota}$, where $\iota = 2-d/2$. It follows
\begin{equation}\label{eq:estimate_fully_II}
\mathbf{II}_h
\leq \mathfrak{E}h^{4 - d} + \frac{\mathfrak{c}}{2}\|\bar{\mathbf{u}}_{h} - \bar{\mathbf{u}}\|_{\mathbf{L}^2(\Omega)}^2, \quad \mathfrak{E}>0.
\end{equation}

We substitute the bounds for $\mathbf{I}_h$ and $\mathbf{II}_h$ in \eqref{eq:FD_ct_error_estimate_0} and take $\mathfrak{c}$ sufficiently small to obtain \eqref{eq:error_estimate_fully} (note that $2\ell < 4-d$). This is a contradiction and completes the proof.
\end{proof}


\subsection{Error estimates: the semidiscrete scheme} 
Let $(\bar{\mathbf{y}}, \bar{p}, \bar{\mathbf{u}})$ be a strict local minimum of \eqref{eq:weak_cost}--\eqref{eq:weak_st_eq} and let $\{ (\bar{\mathbf{y}}_h, \bar{p}_h, \bar{\boldsymbol{u}}_{h})\}_{h<h_{\Box}}$ be a sequence of local minima of the semidiscrete optimal control problems such that $\bar{\boldsymbol{u}}_{h} \to \bar{\mathbf{u}}$ in $\mathbf{L}^2(\Omega)$ as $h \rightarrow 0$; cf.~Theorem \ref{thm:exist_and_conv_sol_var}. The following result is an aid to the main error estimate.

\begin{lemma}[auxiliary error estimate]
If $\bar{\mathbf{u}}\in\mathbf{U}_{ad}$ satisfies the second order optimality conditions  \eqref{eq:second_order_2_2}, then there exists $h_{\S} > 0$ such that
\begin{equation}\label{eq:aux_estimate_semi}
\textgoth{C}
\|\bar{\mathbf{u}}-\bar{\boldsymbol u}_h\|_{\mathbf{L}^2(\Omega)}^2 \leq [j'(\bar{\boldsymbol u}_h)-j'(\bar{\mathbf{u}})](\bar{\boldsymbol u}_h-\bar{\mathbf{u}}) \quad \forall h < h_{\S},
\quad
\textgoth{C} = 2^{-1} \min\{\mu,\alpha\}.
\end{equation}
Here, $\alpha$ is the control cost and $\mu$ is the constant appearing in \eqref{eq:second_order_2_2}.
\end{lemma}
\begin{proof}
Define $\mathbf{g}_h:= (\bar{\boldsymbol u}_h-\bar{\mathbf{u}})/\|\bar{\boldsymbol u}_h-\bar{\mathbf{u}}\|_{\mathbf{L}^2(\Omega)}$. Since $\{ \mathbf{g}_h \}_{0<h<h_{\Box}}$ is uniformly bounded in $\mathbf{L}^2(\Omega)$, possibly up to a subsequence, we can assume that $\mathbf{g}_{h}  \rightharpoonup \mathbf{g}$ in $\mathbf{L}^{2}(\Omega)$ as $h \to 0$. In the following, we prove that $\mathbf{g}\in \mathbf{C}_{\bar{\mathbf{u}}}$. We proceed as in the proof of Lemma \ref{lemma:aux_error_estimate} and deduce that $\mathbf{g}$ satisfies \eqref{eq:sign_cond}.
To show that $\bar{\mathbf{d}}_{i}(x) \neq 0$ implies that $\mathbf{g}_{i}(x) = 0$ for a.e.~$x\in\Omega$ and $i\in\{1,\ldots,d \}$, we introduce $\bar{\boldsymbol{d}}_h:= \bar{\mathbf{z}}_{h} + \alpha \bar{\boldsymbol u}_{h}$. Recall that $\bar{\mathbf{d}}= \bar{\mathbf{z}} + \alpha \bar{\mathbf{u}}$. Let us now use $\|\bar{\mathbf{u}} - \bar{\boldsymbol u}_{h}\|_{\mathbf{L}^2(\Omega)}\to 0$ as $h \to 0$ and Theorem \ref{thm:convergence_adj_eq} to obtain
\begin{equation*}
\| \bar{\mathbf{d}} - \bar{\boldsymbol{d}}_h\|_{\mathbf{L}^2(\Omega)} 
\leq 
\| \bar{\mathbf{z}} - \bar{\mathbf{z}}_h\|_{\mathbf{L}^2(\Omega)} + 
\alpha\| \bar{\mathbf{u}} - \bar{\boldsymbol{u}}_h\|_{\mathbf{L}^2(\Omega)} \to 0, \quad h \to 0.
\end{equation*}
From this it follows that $\bar{\boldsymbol{d}}_{h} \to \bar{\mathbf{d}}$ in $\mathbf{L}^2(\Omega)$ as $h\to 0$. Consequently, we obtain
\begin{equation*}
(\bar{\mathbf{d}},\mathbf{g})_{\mathbf{L}^2(\Omega)}
=\lim_{h \to 0}(\bar{\boldsymbol{d}}_h,\mathbf{g}_h)_{\mathbf{L}^2(\Omega)}
=\lim_{h \to 0}\frac{1}{\|\bar{\boldsymbol{u}}_h-\bar{\mathbf{u}}\|_{\mathbf{L}^2(\Omega)}}(\bar{\mathbf{z}}_{h} + \alpha \bar{\boldsymbol u}_{h}, \bar{\boldsymbol{u}}_{h} - \bar{\mathbf{u}})_{\mathbf{L}^2(\Omega)}\leq 0,
\end{equation*}
after applying the variational inequality \eqref{eq:var_ineq_semi}. We proceed as in the proof
of Lemma \ref{lemma:aux_error_estimate} and deduce that if $\bar{\mathbf{d}}_{i}(x)\neq 0$, then $\mathbf{g}_{i}(x) = 0$ for a.e.~$x\in\Omega$ and $i\in\{1,\ldots,d\}$. Therefore, $\mathbf{g} \in \mathbf{C}_{\bar{\mathbf{u}}}$.
The rest of the proof follows the same arguments developed in the proof of Lemma \ref{lemma:aux_error_estimate}. For brevity, we omit details.
\end{proof}

The following error estimate is the most important result of this section and improves the error bound of Theorem \ref{thm:error_estimate_fully}.

\begin{theorem}[error estimate]
If $\bar{\mathbf{u}}\in\mathbf{U}_{ad}$ satisfies the second order optimality conditions  \eqref{eq:second_order_2_2}, then there exists $h_{\S} > 0$ such that
\begin{equation}\label{eq:error_estimate_semi}
\|\bar{\mathbf{u}}-\bar{\boldsymbol u}_h\|_{\mathbf{L}^2(\Omega)} 
\lesssim
h^{2 - \frac{d}{2}} \quad \forall h < h_{\S}.
\end{equation}
\end{theorem}
\begin{proof}
Setting $\mathbf{u}=\bar{\boldsymbol{u}}_{h}$ in \eqref{eq:variational_inequality} and $\mathbf{u} = \bar{\mathbf{u}}$ in \eqref{eq:var_ineq_semi}, we get $ -j'(\bar{\mathbf{u}})(\bar{\boldsymbol{u}}_{h} - \bar{\mathbf{u}}) \leq 0$ and $- (\bar{\mathbf{z}}_{h} + \alpha\bar{\boldsymbol u}_{h},\bar{\boldsymbol{u}}_{h} - \bar{\mathbf{u}})_{\mathbf{L}^2(\Omega)}\geq 0$, respectively. 
It follows from \eqref{eq:aux_estimate_semi} that
\begin{equation*}
\|\bar{\mathbf{u}}-\bar{\boldsymbol u}_h\|_{\mathbf{L}^2(\Omega)}^2 \lesssim j'(\bar{\boldsymbol u}_h)(\bar{\boldsymbol u}_h-\bar{\mathbf{u}}) - (\bar{\mathbf{z}}_{h} + \alpha\bar{\boldsymbol u}_{h},\bar{\boldsymbol{u}}_{h} - \bar{\mathbf{u}})_{\mathbf{L}^2(\Omega)}
= 
(\tilde{\mathbf{z}}(\bar{\boldsymbol u}_h) - \bar{\mathbf{z}}_{h},\bar{\boldsymbol{u}}_{h} - \bar{\mathbf{u}})_{\mathbf{L}^2(\Omega)}.
\end{equation*}
Here, $(\tilde{\mathbf{z}}(\bar{\boldsymbol u}_h),\tilde{r}(\bar{\boldsymbol u}_h))$ is as in the proof of Theorem \ref{thm:error_estimate_fully}. The proof of Theorem \ref{thm:error_estimate_fully} also shows that $\| \tilde{\mathbf{z}}(\bar{\boldsymbol u}_h) - \bar{\mathbf{z}}_{h}\|_{\mathbf{L}^2(\Omega)} \lesssim h^{\iota}$, where $\iota = 2 - d/2$. This concludes the proof.
\end{proof}

\bibliographystyle{siamplain}
\bibliography{bi_NS_tracking}

\end{document}